\documentclass{article}

\usepackage{lineno,hyperref}
\usepackage{amsmath}
\usepackage{amssymb}
\usepackage{amsthm}
\modulolinenumbers[5]
\usepackage{graphicx}

\newtheorem{lemma}{Lemma}
\newtheorem*{lemma*}{Lemma}

\newtheorem{theorem}{Theorem}
\newtheorem*{theorem*}{Theorem}

\theoremstyle{definition}
\newtheorem{example}{Example}

\makeatletter
\newcommand{\Ast}{\mathop{\scalebox{1.5}{\raisebox{-0.18ex}{\(\ast\)}}}}
\newcommand{\bigast}{\mathop{\mathpalette\big@st\relax}\displaylimits}
\newcommand{\big@st}[2]{\vcenter{\m@th\hbox{\scalebox{\ifx#1\displaystyle2.1\else1.5\fi}{\(#1\Ast\)}}}}
\makeatother

\newcommand\dotminus{\mathbin{\dot{-}}}

\newcommand\dotoplus{\mathbin{\dot{\oplus}}}

\newcommand\sMat[4]{\bigl(
	\begin{smallmatrix}
		#1 \mathstrut & #2 \mathstrut \\
		#3 \mathstrut & #4 \mathstrut
	\end{smallmatrix}
\bigr)}

\DeclareMathOperator\mat{M}

\DeclareMathOperator\elem{E}

\DeclareMathOperator\glin{GL}

\DeclareMathOperator\stlin{St}
\DeclareMathOperator\symp{Sp}

\DeclareMathOperator\orth{O}
\DeclareMathOperator\sorth{SO}

\DeclareMathOperator\storth{StO}

\DeclareMathOperator\Ker{Ker}
\DeclareMathOperator\sign{sign}

\newcommand\eps{\varepsilon}

\newcommand\inv[1]{
	\!\;\overline{
		\!\!\:#1\vphantom !\!\!\:
	}\;\!
}
\newcommand\op{{\mathrm{op}}}

\DeclareMathOperator\Aut{Aut}

\DeclareMathOperator\height{ht}

\newcommand{\up}[2]{{^{#1}\!{#2}}}

\newcommand{\Set}{\mathbf{Set}}

\newcommand{\interval}[2]{\,{]}{#1}, {#2}{[}\,}

\title{
	Root graded groups revisited
}

\author{
	Egor Voronetsky \\
	Chebyshev Laboratory, \\
	St. Petersburg State University, \\
	14th Line V.O., 29B, \\
	Saint Petersburg 199178 Russia \\
}

\begin{document}
\maketitle

\begin{flushright}
	\textit{Dedicated to the memory of Nikolai Vavilov}
\end{flushright}

\begin{abstract}
	A group \(G\) is called root graded if it has a family of subgroups
	\(
		G_\alpha
	\) indexed by roots from a root system \(\Phi\) satisfying natural conditions similar to Chevalley groups over commutative unital rings. For any such group there is a corresponding algebraic structure (commutative unital ring, associative unital ring, etc.) encoding the commutator relations between
	\(
		G_\alpha
	\). We give a complete description of varieties of such structures for irreducible root systems of rank \(\geq 3\) excluding
	\(
		\mathsf H_3
	\) and
	\(
		\mathsf H_4
	\). Moreover, we provide a construction of root graded groups for all algebraic structures from these varieties.
\end{abstract}

\section{Introduction}

Let \(K\) be a commutative unital ring and \(\Phi\) be a root system. Structure of the Chevalley group
\(
	G(\Phi, K)
\) (with respect to some weight lattice) is well-known \cite{vavilov-plotkin}. Excepting some small rank cases,
\(
	G(\Phi, K)
\) has the largest perfect subgroup
\(
	\elem(\Phi, K)
\) generated by all root elements (the elementary subgroup), this subgroup is normal, the factor-group is solvable, and there is a classification of all subgroups of
\(
	G(\Phi, K)
\) normalized by
\(
	\elem(\Phi, K)
\) in terms of ideals of \(K\) and relative \(\mathrm K_1\)-functors. Some of these results require that \(K\) is finite-dimensional. If some structure constants are not invertible in \(K\), then the classification of normal subgroups involves not only ideals, but also so-called admissible pairs \cite{abe}. Hence it is natural to consider a generalization of Chevalley groups where root subgroups with various root lengths are parameterized by different objects.

A root graded group \(G\) is an abstract group with root subgroups
\(
	G_\alpha
\) such that, informally, they satisfy Chevalley commutator formula and they a permuted by some Weyl elements
\(
	n_\alpha \in G_\alpha\, G_{-\alpha}\, G_\alpha
\). All Chevalley groups
\(
	G(\Phi, K)
\) and their elementary subgroups are naturally root graded. Moreover, isotropic reductive groups over semi-local rings with connected spectra are also root graded, see \cite{borel-tits}, \cite[Exp. XXVI]{sga3}, and \cite{tits-indices}.

Each root graded group \(G\) determines a certain multi-sorted algebraic structure consisting of its root subgroups. Operations are given by the multiplications of
\(
	G_\alpha
\) and the commutator maps between various root subgroups. Moreover, Weyl elements determine some distinguished elements in the root subgroups and identifications between root subgroups corresponding to roots of the same length (i.e. from the same orbit under the action of the Weyl group). The goal of this paper is to describe the varieties of such algebraic structures and to prove the existence theorem, namely, that every such structure comes from a root graded group.

We consider only non-crystallographic root systems \(\Phi\) of rank
\(
	\geq 3
\) except
\(
	\mathsf H_3
\) and
\(
	\mathsf H_4
\). Description of varieties (the ``coordinatisation theorem'') is mostly done by T. Wiedemann in \cite{wiedemann} (see also \cite{shi} for the simply laced case) except the case of non-crystallographic
\(
	\mathsf F_4
\) and the case of
\(
	\mathsf B_\ell
\) with
\(
	\lambda \neq -1
\) (in our terminology from \S 7) not reduced to orthogonal groups over commutative rings. The existence theorem for root systems
\(
	\mathsf A_\ell
\),
\(
	\mathsf D_\ell
\),
\(
	\mathsf E_\ell
\) is well-known since the algebraic structures are associative and commutative unital rings. For
\(
	\mathsf B_\ell
\) with
\(
	\ell \geq 4
\) the existence follows e.g. from the general theory of odd unitary groups \cite{bc-comm-rel}. Finally, for
\(
	\mathsf B_3
\) and
\(
	\mathsf F_4
\) only partial (though rather non-trivial) results were known, see \cite{wiedemann} for details. There are also similar results for classical groups with non-degenerated commutator relations in the sense of \cite[chapter I]{loos-neher}, see \cite{a-comm-rel, bc-comm-rel}. Further references including history of these objects may be found in \cite{wiedemann}.

The paper is organized as follows. In \S 2 and \S 3 we discuss some general definitions and show how to explicitly construct a variety by the class of \(\Phi\)-graded groups. In \S 4 we consider the important class of groups with
\(
	\mathsf B_\ell
\)-commutator relations and
\(
	\mathsf A_{\ell - 1}
\)-Weyl elements, i.e. almost
\(
	\mathsf B_\ell
\)-graded groups but with some Weyl elements missing. This class has a relatively simple associated variety with \(6\) sorts and without complicated sign conventions in the axioms and commutator formulae. In \S 5 we apply these objects to prove the coordinatisation theorems for simply laces root systems. In the remaining sections 6--9 we focus on
\(
	\mathsf B_3
\)- and
\(
	\mathsf F_4
\)-graded groups. Since the associated varieties contain alternative rings, we recall necessary preliminaries in \S 6. The sections 7 and 8 contain the coordinatisation theorems for
\(
	\mathsf B_\ell
\)- and
\(
	\mathsf F_4
\)-graded groups respectively. Finally, in \S 9 we prove the existence theorem in full generality.

Our paper is mostly self-contained, we need only some basic results about alternative rings and root systems. Also, we use the existence of Chevalley groups of types
\(
	\mathsf E_\ell
\). Actually, it is possible to use only
\(
	G(\mathsf E_\ell, \mathbb C)
\) (in \S 5 and \S 8) since the method used in our existence theorem is applicable to
\(
	\mathsf E_\ell
\) as well, though we do not write this explicitly.

In general we say that
\(
	(\Phi, \Psi)
\)-rings are the objects coordinatising groups with \(\Phi\)-commutator relations and \(\Psi\)-Weyl elements for a root subsystem
\(
	\Psi \subseteq \Phi
\). Thus our main theorems concern
\(
	(\Phi, \Phi)
\)-rings. For
\(
	\Phi = \mathsf A_\ell
\) they are just associative unital rings, for
\(
	\Phi = \mathsf D_\ell
\) and
\(
	\Phi = \mathsf E_\ell
\) such objects are commutative unital rings, for
\(
	\Phi = \mathsf B_\ell
\) they consist of an alternative unital ring (associative for
\(
	\ell \geq 4
\)) with a pseudo-involution and an ``odd form parameter''. Finally,
\(
	(\mathsf F_4, \mathsf F_4)
\)-rings consist of two alternative unital rings with involutions and some additional maps. There is a symmetry between these two rings is due to the outer automorphism of the non-crystallographic root system
\(
	\mathsf F_4
\).

For convenience, let us compare Wiedemann's alternative rings with Jordan modules
\(
	(R, J)
\) \cite[definitions 8.1.5, 8.3.12, and 8.6.1]{wiedemann} and our
\(
	(\mathsf B_3, \mathsf B_3)
\)-rings
\(
	(R, \Delta)
\) from \S 7, i.e. alternative rings with odd form parameters, in the particular case
\(
	\lambda = -1
\). Wiedemann's nuclear involution
\(
	x^\sigma
\) is our pseudo-involution
\(
	x^*
\) (an actual involution since \(\lambda\) is central). The non-abelian groups \(J\) and \(\Delta\) correspond to each other, Wiedemann's ``square-module'' operation \(\phi(u, x)\) is our
\(
	u \cdot x
\),
\(
	\pi_1(u)
\) is our
\(
	\rho(u)
\),
\(
	T_1(x)
\) is our
\(
	\phi(x)
\), and
\(
	\psi(u, v)
\) is our
\(
	-\langle u, v \rangle
\). In the theory of odd form rings \cite{thesis} actually
\(
	\langle u, v \rangle = \inv{\pi(u)}\, \pi(v)
\) for some map
\(
	\pi \colon \Delta \to R_{01}
\) for a right \(R\)-module
\(
	R_{01}
\), it is explicitly constructed in \cite{bc-comm-rel} (but only if \(R\) is associative, e.g. for
\(
	\ell \geq 4
\)). There exist
\(
	(\mathsf B_\ell, \mathsf B_\ell)
\)-rings with
\(
	\lambda \neq -1
\) (and not consisting of a commutative ring with quadratic form) by example \ref{bb-ring-init} in \S 7, moreover, the corresponding
\(
	\mathsf B_\ell
\)-graded groups contain non-weakly-balanced Weyl triples in the sense of \cite[definition 2.2.12]{wiedemann}.

\section{Root graded groups}

Recall that a (non-crystallographic) \textit{root system}
\(
	\Phi \subseteq \mathbb R^\ell \setminus \{0\}
\) is a finite spanning subset of a Euclidean vector space such that
\begin{itemize}

	\item
	if
	\(
		\alpha \in \Phi
	\) and
	\(
		\beta \in \Phi \cap \mathbb R \alpha
	\), then
	\(
		\beta = \pm \alpha
	\);

	\item
	if
	\(
		\alpha, \beta \in \Phi
	\), then
	\(
		s_\alpha(\beta)
		= \beta
		- 2 \frac{
			\alpha \cdot \beta
		}{
			\alpha \cdot \alpha
		} \alpha
		\in \Phi
	\), where
	\(
		\alpha \cdot \beta
	\) is the dot product.

\end{itemize}
We need the notation
\(
	\interval \alpha \beta
	= \Phi \cap (
		\mathbb R_{> 0} \alpha
		+ \mathbb R_{> 0} \beta
	)
\) for linearly independent roots \(\alpha\), \(\beta\). The \textit{Weyl group} of \(\Phi\) is
\(
	\mathrm W(\Phi)
	= \langle
		s_\alpha
		\mid
		\alpha \in \Phi
	\rangle
	\leq \orth(\ell)
\). A root system \(\Phi\) is \textit{irreducible} if it is non-empty (i.e.
\(
	\ell > 0
\)) and it is indecomposable into a disjoint union of non-empty orthogonal subsets. It is well-known that up to the choice of root lengths irreducible root systems are completely determined by their \textit{type}:
\(
	\mathsf A_\ell
\) for
\(
	\ell \geq 1
\),
\(
	\mathsf B_\ell = \mathsf C_\ell
\) for
\(
	\ell \geq 2
\),
\(
	\mathsf D_\ell
\) for
\(
	\ell \geq 4
\),
\(
	\mathsf E_\ell
\) for
\(
	\ell \in \{6, 7, 8\}
\),
\(
	\mathsf F_4
\),
\(
	\mathsf G_2
\),
\(
	\mathsf H_\ell
\) for
\(
	\ell \in \{2, 3, 4\}
\), or
\(
	\mathrm I_2^m
\) for \(m \geq 7\). The lower index of a type denotes the space dimension.

A root system is called \textit{crystallographic} if
\(
	2 \frac{
		\alpha \cdot \beta
	}{
		\alpha \cdot \alpha
	} \in \mathbb Z
\), but without the condition
\(
	\Phi \cap \mathbb R \alpha = \{\pm \alpha\}
\). Such root systems are also classified by the type, but the non-crystallographic type
\(
	\mathsf B_\ell
\) splits into the crystallographic types
\(
	\mathsf B_\ell
\) and
\(
	\mathsf C_\ell
\) (it is distinct from
\(
	\mathsf B_\ell
\) for
\(
	\ell \geq 3
\)). There is an additional crystallographic type
\(
	\mathsf{BC}_\ell
\) for
\(
	\ell \geq 1
\), after identifying co-directional roots it corresponds to the non-crystallographic
\(
	\mathsf B_\ell
\) for
\(
	\ell \geq 2
\) or
\(
	\mathsf A_1
\). The non-crystallographic types
\(
	\mathsf H_\ell
\) and
\(
	\mathsf I_2^m
\) have no crystallographic realizations. For future references we list some root systems of all crystallographic types following Bourbaki \cite[\S XI.4]{bourbaki}.
\begin{align*}
	\mathsf A_\ell
	&= \{
		\mathrm e_i - \mathrm e_j
		\mid
		1 \leq i \neq j \leq \ell + 1
	\},
\\
	\mathsf B_\ell
	&= \{
		\pm \mathrm e_i \pm \mathrm e_j
		\mid
		1 \leq i < j \leq \ell
	\}
	\sqcup \{
		\pm \mathrm e_i
		\mid
		1 \leq i \leq \ell
	\},
\\
	\mathsf C_\ell
	&= \{
		\pm \mathrm e_i \pm \mathrm e_j
		\mid
		1 \leq i < j \leq \ell
	\}
	\sqcup \{
		\pm 2 \mathrm e_i
		\mid
		1 \leq i \leq \ell
	\},
\\
	\mathsf{BC}_\ell
	&= \{
		\pm \mathrm e_i \pm \mathrm e_j
		\mid
		1 \leq i < j \leq \ell
	\}
	\sqcup \{
		\pm \mathrm e_i
		\mid
		1 \leq i \leq \ell
	\}
	\sqcup \{
		\pm 2 \mathrm e_i
		\mid
		1 \leq i \leq \ell
	\},
\\
	\mathsf D_\ell
	&= \{
		\pm \mathrm e_i \pm \mathrm e_j
		\mid
		1 \leq i < j \leq \ell
	\},
\\
	\mathsf E_6
	&= \{
		\pm \mathrm e_i \pm \mathrm e_j
		\mid
		1 \leq i < j \leq 5
	\}
	\sqcup \bigl\{
		\textstyle \pm \frac 1 2 (
			\sum_{i = 1}^5 \sigma_i \mathrm e_i
			+ \mathrm e_6
			+ \mathrm e_7
			- \mathrm e_8
		)
		\mid
		\sigma_i \in \{-1, 1\},\,
		\sum_i \sigma_i \equiv 3 \!\!\!\pmod 4
	\bigr\},
\\
	\mathsf E_7
	&= \{
		\pm \mathrm e_i \pm \mathrm e_j
		\mid
		1 \leq i < j \leq 6
	\}
	\sqcup \{
		\pm (\mathrm e_7 - \mathrm e_8)
	\} \\
	&\sqcup \bigl\{
		\textstyle \pm \frac 1 2 (
			\sum_{i = 1}^6 \sigma_i \mathrm e_i
			+ \mathrm e_7
			- \mathrm e_8
		)
		\mid
		\sigma_i \in \{-1, 1\},\,
		\sum_i \sigma_i \equiv 0 \!\!\!\pmod 4
	\bigr\},
\\
	\mathsf E_8
	&= \{
		\pm \mathrm e_i \pm \mathrm e_j
		\mid
		1 \leq i < j \leq 8
	\}
	\sqcup \bigl\{
		\textstyle \frac 1 2
		\sum_{i = 1}^8 \sigma_i \mathrm e_i
		\mid
		\sigma_i \in \{-1, 1\},\,
		\sum_i \sigma_i \equiv 0 \!\!\!\pmod 4
	\bigr\},
\\
	\mathsf F_4
	&= \{
		\pm \mathrm e_i
		\mid
		1 \leq i \leq 4
	\}
	\sqcup \{
		\pm \mathrm e_i \pm \mathrm e_j
		\mid
		1 \leq i < j \leq 4
	\}
	\sqcup \bigl\{
		\textstyle \frac 1 2 \sum_{i = 1}^4
			\sigma_i \mathrm e_i
		\mid
		\sigma_i \in \{-1, 1\}
	\bigr\},
\\
	\mathsf G_2
	&= \{
		\pm (\mathrm e_i - \mathrm e_j)
		\mid
		1 \leq i < j \leq 3
	\}
	\sqcup \bigl\{
		\pm (
			2 \mathrm e_1
			- \mathrm e_2
			- \mathrm e_3
		),
		\pm (
			2 \mathrm e_2
			- \mathrm e_1
			- \mathrm e_3
		),
		\pm (
			2 \mathrm e_3
			- \mathrm e_1
			- \mathrm e_2
		)
	\bigr\}.
\end{align*}

Until the end of this section \(\Phi\) is an arbitrary non-crystallographic root system. We say that a subset
\(
	\Sigma \subseteq \Phi
\) of a root system is \textit{special closed} if it is an intersection of \(\Phi\) with a convex cone not containing opposite non-zero vectors. An \textit{extreme root}
\(
	\alpha \in \Sigma
\) of a special closed subset
\(
	\Sigma \subseteq \Phi
\) is a root from an extreme ray of the conical convex hull
\(
	\sum_{\alpha \in \Sigma}
		\mathbb R_{\geq 0} \alpha
\) of \(\Sigma\). Every non-empty special closed subset
\(
	\Sigma \subseteq \Phi
\) contains an extreme root \(\alpha\) and for any such \(\alpha\) the set
\(
	\Sigma \setminus \{\alpha\}
\) is also special closed. We say that a linear order on a special closed subset
\(
	\Sigma \subseteq \Phi
\) is \textit{right extreme} if either
\(
	\Sigma = \varnothing
\) or the largest root
\(
	\alpha_{\max} \in \Sigma
\) is extreme and the induced order on
\(
	\Sigma \setminus \{\alpha_{\max}\}
\) is also right extreme. Right extreme orders are opposite to \textit{extreme} orders from \cite[definition 2.3.16]{wiedemann}. Unlike Wiedemann we define group commutators as
\(
	[f, g] = f g f^{-1} g^{-1}
\), so right extreme orders are more convenient. Clearly, every special closed subset admits a right extreme order.

A group \(G\) has \textit{\(\Phi\)-commutator relations} it it has a family of subgroups
\(
	(G_\alpha \leq G)_{\alpha \in \Phi}
\) for
\(
	\alpha \in \Phi
\) such that
\begin{itemize}

	\item
	\(
		[G_\alpha, G_\beta]
		\leq \bigl\langle
			G_\gamma
			\mid
			\gamma \in \interval \alpha \beta
		\bigr\rangle
	\) for all linearly independent
	\(
		\alpha, \beta \in \Phi
	\);

	\item
	\(
		G_\alpha
		\cap \bigl\langle
			G_\beta
			\mid
			\beta \in \Sigma \setminus \{\alpha\}
		\bigr\rangle
		= 1
	\) for all special closed subsets
	\(
		\Sigma \subseteq \Phi
	\) and extreme
	\(
		\alpha \in \Sigma
	\).

\end{itemize}
The second condition is a bit weaker than \cite[definition 2.5.2, (iv)]{wiedemann} and is usually not included in the definition of groups with commutator relations.

Suppose that \(G\) has \(\Phi\)-commutator relations. An \textit{\(\alpha\)-Weyl element} for
\(
	\alpha \in \Phi
\) is an element
\(
	n \in G_\alpha\, G_{-\alpha}\, G_\alpha
\) such that
\(
	\up n{G_\beta} = G_{s_\alpha(\beta)}
\) for all
\(
	\beta \in \Phi
\) (note that such \(n\) is called
\(
	(-\alpha)
\)-Weyl in \cite[definition 2.2.2]{wiedemann}). Clearly, if
\(
	n = f g h
\) is an \(\alpha\)-Weyl element for
\(
	f \in G_\alpha
\),
\(
	g \in G_{-\alpha}
\),
\(
	h \in G_\alpha
\), then
\(
	n^{-1} = h^{-1} g^{-1} f^{-1}
\) is also \(\alpha\)-Weyl element and
\(
	n = g h f^n = \up nh f g
\) is simultaneously
\(
	(-\alpha)
\)-Weyl. A group \(G\) is called \textit{\(\Phi\)-graded} if for every root \(\alpha\) there exists an \(\alpha\)-Weyl element \cite[definition 2.5.2]{wiedemann}.

For example, let \(\Phi\) be a crystallographic root system distinct from
\(
	\mathsf{BC}_\ell
\) and \(K\) be a commutative unital ring. The \textit{Chevalley group}
\(
	G^{\mathrm{ad}}(\Phi, K)
\) is the \(K\)-point group of the split reductive group scheme of adjoint type with the root system \(\Phi\). This group has a \(\Phi\)-grading with natural isomorphisms
\(
	t_\alpha \colon K \to G_\alpha
\) satisfying \textit{Chevalley commutator formula}
\[
	[t_\alpha(x), t_\alpha(y)]
	= \prod_{
		\substack{
			i \alpha + j \beta \in \Phi
		\\
			i, j \in \mathbb N_{> 0}
		}
	} t_{i \alpha + j \beta}(
		N_{\alpha \beta i j} x^i y^j
	)
\]
for non-opposite
\(
	\alpha, \beta \in \Phi
\) and the formula
\[
	\up{
		t_\alpha(a)\,
		t_{-\alpha}(-a^{-1})\,
		t_\alpha(a)
	}{t_\beta(x)}
	= t_{s_\alpha(\beta)}\bigl(
		d_{\alpha \beta}
		a^{
			-2 \frac{
				\alpha \cdot \beta
			}{
				\alpha \cdot \alpha
			}
		}
		x
	\bigr)
\]
for all
\(
	\alpha, \beta \in \Phi
\), where
\(
	a \in K^*
\) and
\(
	x, y \in K
\). Here
\(
	N_{\alpha \beta i j} \in \mathbb Z
\) and
\(
	d_{\alpha \beta} \in \mathbb Z^*
\) are the \textit{structure constants}. Their absolute values depend only on the orbit of
\(
	(\alpha, \beta) \in \Phi \times \Phi
\) under the action of the Weyl group and the values of \(i\), \(j\) in the case of
\(
	N_{\alpha \beta i j}
\). On the other hand, the signs depend not only on the roots and the indices, but also on the choice of
\(
	t_\alpha
\). All Weyl elements are of the type
\(
	t_\alpha(a)\, t_{-\alpha}(-a^{-1})\, t_\alpha(a)
\) with
\(
	a \in K^*
\). The \textit{standard torus} of
\(
	G^{\mathrm{ad}}(\Phi, K)
\) is denoted by
\(
	T^{\mathrm{ad}}(\Phi, K)
\), it stabilizes the root subgroups. We make the identification
\[
	T^{\mathrm{ad}}(\Phi, K)
	\cong \{
		g = (a_\alpha \in K^*)_{\alpha \in \Phi}
		\mid
		a_{i \alpha + j \beta} = a_\alpha^i a_\beta^j
		\text{ for } i \alpha + j \beta \in \Phi
		\text{ and } i, j \in \mathbb N_0
	\},\,
	\up g{t_\alpha(x)} = t_\alpha(a_\alpha x).
\]

As another example, let \(R\) be an associative unital ring and
\(
	m \geq 2
\). The general linear group
\(
	\glin(m, R)
\) has so-called \textit{elementary transvections}
\(
	t_{ij}(x) = 1 + x e_{ij}
\) for
\(
	i \neq j
\), where \(1\) is the identity matrix and
\(
	e_{ij}
\) is a matrix unit. The maps
\(
	t_{ij} \colon R \to \glin(m, R)
\) are injective group homomorphisms, their images gives an
\(
	\mathsf A_{m - 1}
\)-grading on
\(
	\glin(m, R)
\). Namely,
\begin{itemize}

	\item
	\(
		[t_{ij}(x), t_{jk}(y)] = t_{ik}(xy)
	\) for
	\(
		i \neq k
	\);

	\item
	\(
		[t_{ij}(x), t_{kl}(y)] = 1
	\) for
	\(
		j \neq k
	\) and
	\(
		i \neq l
	\);

	\item
	\(
		t_{ij}(a)\, t_{ji}(b)\, t_{ij}(c)
	\) is a Weyl element if and only if
	\(
		a = c \in R^*
	\) and
	\(
		b = -a^{-1}
	\).

\end{itemize}

The next lemma is a strengthening of the second condition from the definition of commutator relations, it is proved in \cite[lemma 2.4.17]{wiedemann} under a bit stronger assumption.
\begin{lemma} \label{comm-rel}
	Let \(G\) be a group with \(\Phi\)-commutator relations and
	\(
		\Sigma \subseteq \Phi
	\) be a special closed subset. Then for any right extreme order on \(\Sigma\) the product map
	\[
		\prod_{\alpha \in \Sigma}^{\Set}
			G_\alpha
		\to G,\,
		(g_\alpha)_{\alpha \in \Sigma}
		\mapsto \prod_{\alpha \in \Sigma}
			g_\alpha
	\]
	with respect to this order is injective and its image
	\(
		G_\Sigma
	\) is a subgroup of \(G\), i.e.
	\(
		G_\Sigma
		= \langle
			G_\alpha
			\mid
			\alpha \in \Sigma
		\rangle
	\).
\end{lemma}
\begin{proof}
	Let us prove that the product map is injective by induction on \(\Sigma\). Indeed, if
	\(
		\prod_{\alpha \in \Sigma} g_\alpha
		= \prod_{\alpha \in \Sigma} f_\alpha
	\) for
	\(
		g_\alpha, f_\alpha \in G_\alpha
	\), then by the definition of \(\Phi\)-commutator relations
	\(
		g_{\alpha_{\max}} = f_{\alpha_{\max}}
	\), so by the induction hypothesis
	\(
		g_\alpha = f_\alpha
	\) for remaining \(\alpha\).

	Now we check that
	\(
		G_\Sigma\, G_\alpha \subseteq G_\Sigma
	\) for all
	\(
		\alpha \in \Sigma
	\) by induction on \(\Sigma\). This is obvious if
	\(
		\alpha = \alpha_{\max}
	\). Otherwise we have
	\[
			G_\Sigma\, G_\alpha
		=
			G_{\Sigma \setminus \{\alpha_{\max}\}}\,
			G_{\alpha_{\max}}\,
			G_\alpha
		\subseteq
			G_{\Sigma \setminus \{\alpha_{\max}\}}\,
			[G_{\alpha_{\max}}, G_\alpha]\,
			G_\alpha\,
			G_{\alpha_{\max}}
		=
			G_{\Sigma \setminus \{\alpha_{\max}\}}\, G_{\alpha_{\max}}
		=
			G_\Sigma
	\]
	by the induction hypothesis.
\end{proof}

\section{Abstract \((\Phi, \varnothing)\)-rings}

We start an algebraic description of \(\Phi\)-graded groups with the following definition. A \textit{\((\Phi, \varnothing)\)-ring} consists of
\begin{itemize}

	\item
	a family
	\(
		(P_\alpha)_{\alpha \in \Phi}
	\) of groups with the group operations \(\dotplus\),

	\item
	for any linearly independent
	\(
		\alpha, \beta \in \Phi
	\), any right extreme order \(\preceq\) on
	\(
		\interval \alpha \beta
	\), and any
	\(
		\gamma \in \interval \alpha \beta
	\) a map
	\(
		C_{\alpha \beta}^{\gamma, {\preceq}}
		\colon P_\alpha \times P_\beta
		\to P_\gamma
	\)

\end{itemize}
such that for any special closed subset
\(
	\Sigma \subseteq \Phi
\) there exists a group \(G_\Sigma\) together with group homomorphisms
\(
	t_\alpha \colon P_\alpha \to G_\Sigma
\) for all
\(
	\alpha \in \Sigma
\) satisfying
\begin{itemize}

	\item
	for any right extreme order \(\leq\) on \(\Sigma\) the product map
	\(
		\prod_{\alpha \in \Sigma}^\Set P_\alpha
		\to G_\Sigma,\,
		(\zeta_\alpha)_\alpha
		\mapsto \prod_\alpha t_\alpha(\zeta_\alpha)
	\) from the set-theoretic product with respect to \(\leq\) is bijective;

	\item
	for any linearly independent
	\(
		\alpha, \beta \in \Sigma
	\) and any right extreme order \(\preceq\) on
	\(
		\interval \alpha \beta
	\) the commutator formula
	\[
		[t_\alpha(\zeta), t_\beta(\eta)]
		= \prod_{
			\gamma \in \interval \alpha \beta
		} t_\gamma\bigl(
			C_{\alpha \beta}^{\gamma, {\preceq}}(
				\zeta,
				\eta
			)
		\bigr)
	\]
	holds for
	\(
		\zeta \in P_\alpha
	\),
	\(
		\eta \in P_\beta
	\).

\end{itemize}
The symbol \(\varnothing\) in the terminology means that we do not assume existence of any Weyl elements. For simple-laced and doubly-laced root systems (i.e. all root systems of rank
\(
	\geq 3
\) excluding
\(
	\mathsf H_3
\) and
\(
	\mathsf H_4
\)) the order on
\(
	\interval \alpha \beta
\) in the last condition is irrelevant and
\(
	C_{\alpha \beta}^{\gamma, {\preceq}}
\) is independent of \(\preceq\).

If \(G\) is a group with \(\Phi\)-commutator relations, then
\(
	(G_\alpha)_{\alpha \in \Phi}
\) is a
\(
	(\Phi, \varnothing)
\)-ring by lemma \ref{comm-rel}, where
\(
	t_\alpha
\) are the canonical inclusions. For Chevalley groups
\(
	G^{\mathrm{ad}}(\Phi, K)
\) we may also take
\(
	P_\alpha = K
\) with standard
\(
	t_\alpha
\). Similarly, if
\(
	\Phi = \mathsf A_{m - 1}
\) and
\(
	G = \glin(m, R)
\) is a general linear group, then we may take
\(
	P_\alpha = R
\) and
\(
	t_{\mathrm e_i - \mathrm e_j}(x) = t_{ij}(x)
\).

In general it is easy to see that the class of
\(
	(\Phi, \varnothing)
\)-rings is a many-sorted variety of algebras. The operations of the corresponding algebraic theory are
\begin{itemize}

	\item
	the group operations (multiplication, inversion, and the identity) on
	\(
		P_\alpha
	\) and

	\item
	the maps
	\(
		C_{\alpha \beta}^\gamma
	\) for every pair
	\(
		(\alpha, \beta) \in \Phi^2
	\) of linearly independent roots and every
	\(
		\gamma \in \interval \alpha \beta
	\), where the right extreme order
	\(
		\interval \alpha \beta
	\) is fixed.

\end{itemize}
The remaining operations
\(
	C_{\alpha \beta}^{\gamma, {\preceq}}
\) (with different \(\preceq\)) may be easily expressed via the given ones. We assume that the fixed orders on
\(
	\interval \alpha \beta
\) and
\(
	\interval \beta \alpha
\) are opposite.

Consider the following string rewriting system. Its objects are formal strings
\(
	t_{\alpha_1}(\zeta_1) \cdots t_{\alpha_m}(\zeta_m)
\) for
\(
	m \geq 0
\),
\(
	\alpha_i \in \Phi
\),
\(
	\zeta_i \in P_{\alpha_i}
\). The rewrite rules are the following operations on substrings.
\begin{align*}
	t_\alpha(\dot 0)
	&\to \varepsilon
	\text{ (the empty string)};
\quad
	t_\alpha(\zeta)\, t_\alpha(\eta)
	\to t_\alpha(\zeta \dotplus \eta);
\\
	t_\alpha(\zeta)\, t_\beta(\eta)
	&\to \bigl(
		\prod_{\gamma \in \interval \alpha \beta}
			t_\gamma\bigl(
				C_{\alpha \beta}^\gamma(\zeta, \eta)
			\bigr)
	\bigr)\,
	t_\beta(\eta)\,
	t_\alpha(\zeta)
	\text{ for } \alpha \nparallel \beta.
\end{align*}

Let
\(
	\Sigma \subseteq \Phi
\) be a special closed subset and \(\leq\) be a right extreme order on \(\Sigma\). We restrict the rewrite rules to the class of formal strings with roots from \(\Sigma\), in the last rule we also impose the condition
\(
	\alpha > \beta
\). Such restricted rewrite system is terminating, i.e. there are no infinite chains
\(
	g_1 \to g_2 \to \ldots
\) of formal strings. A formal string is called \textit{irreducible} if no rewrite rule may be applied to it, i.e. the roots in this string are strictly increasing and the arguments of
\(
	t_\alpha
\) are non-zero.

\begin{theorem} \label{phi-0-ring}
	Let \(\Phi\) be a root system. For every pair
	\(
		(\alpha, \beta) \in \Phi^2
	\) of linearly independent roots choose a right extreme order
	\(
		\preceq_{\alpha \beta}
	\) on
	\(
		\interval \alpha \beta
	\) such that
	\(
		\preceq_{\alpha \beta}
	\) and
	\(
		\preceq_{\beta \alpha}
	\) are opposite. The following is the complete list of axioms on the operations of
	\(
		(\Phi, \varnothing)
	\)-rings.
	\begin{itemize}

		\item
		All
		\(
			P_\alpha
		\) are groups.

		\item
		The identities
		\(
			C_{\alpha \beta}^\gamma(\zeta, \eta)
			= \dotminus C_{\beta \alpha}^\gamma(
				\eta,
				\zeta
			)
		\) hold.

		\item
		For any ordered pair
		\(
			(\alpha, \beta) \in \Phi^2
		\) of linearly independent roots choose a right extreme order \(\leq\) on
		\(
			\Sigma
			= \Phi \cap (
				\mathbb R_{\geq 0} \alpha
				+ \mathbb R_{\geq 0} \beta
			)
		\) such that
		\(
			\alpha > \beta
		\). Apply any rewrite rules to
		\(
			t_\alpha(\zeta)\,
			t_\alpha(\zeta')\,
			t_\beta(\eta)
		\) starting from
		\(
			t_\alpha(\zeta)\, t_\alpha(\zeta')
			\to t_\alpha(\zeta \dotplus \zeta')
		\) and
		\(
			t_\alpha(\zeta')\, t_\beta(\eta)
			\to \ldots
		\) until we obtain two irreducible formal strings. Then the arguments of the corresponding
		\(
			t_\gamma
		\) in these strings are equal (if
		\(
			t_\gamma
		\) appears only in one string, then its argument is zero). Also,
		\(
			C_{\alpha \beta}^\gamma(\dot 0, \eta)
			= \dot 0
		\).

		\item
		For any triple
		\(
			\{\alpha, \beta, \gamma\} \subseteq \Phi
		\) contained in an open half-space choose a right extreme order \(\leq\) on
		\(
			\Sigma
			= \Phi \cap (
				\mathbb R_{\geq 0} \alpha
				+ \mathbb R_{\geq 0} \beta
				+ \mathbb R_{\geq 0} \gamma
			)
		\) (say, such that
		\(
			\alpha > \beta > \gamma
		\)). Apply any rewrite rules to
		\(
			t_\alpha(\zeta)\,
			t_\beta(\eta)\,
			t_\gamma(\theta)
		\) starting from
		\(
			t_\alpha(\zeta)\, t_\beta(\eta)
			\to \ldots
		\) and
		\(
			t_\beta(\eta)\, t_\gamma(\theta)
			\to \ldots
		\) until we obtain two irreducible formal strings. Then the arguments of the corresponding
		\(
			t_\delta
		\) in these strings are equal (if
		\(
			t_\delta
		\) appears only in one string, then its argument is zero).

	\end{itemize}
\end{theorem}
\begin{proof}
	We prove by induction on \(\Sigma\) that the rewrite system restricted to \(\Sigma\) is \textit{confluent} (for any right extreme order on \(\Sigma\)) and there exists a group
	\(
		G_\Sigma
	\) with the properties from the definition of
	\(
		(\Phi, \varnothing)
	\)-rings. The confluence means that if
	\(
		g_1 \leftarrow g \to g_2
	\), then \(g_1\) and \(g_2\) have a common reduction (i.e. they may be made equal using the rewrite rules). This is clear if the rules used in
	\(
		g \to g_1
	\) and
	\(
		g \to g_2
	\) are applied to disjoint substrings. Otherwise we may assume that \(g\) is the union of these substrings. In the calculations below we use the notation
	\[
			[\![t_\alpha(\zeta), t_\beta(\eta)]\!]
		=
			\prod_{
				\gamma \in \interval \alpha \beta
			} t_\gamma\bigl(
				C_{\alpha \beta}^\gamma(\zeta, \eta)
			\bigr)
		\in
			G_{\Sigma'}
	\]
	for distinct roots
	\(
		\alpha, \beta \in \Sigma
	\), where
	\(
		\Sigma' \subset \Sigma
	\) is a proper special closed subset containing
	\(
		\interval \alpha \beta
	\) (so
	\(
		G_{\Sigma'}
	\) exists by the induction hypothesis). Clearly,
	\[
		[\![t_\alpha(\zeta), t_\beta(\eta)]\!]^{-1}
		= [\![t_\beta(\eta), t_\alpha(\zeta)]\!]
	\]
	by the skew-symmetry of
	\(
		C_{\alpha \beta}^\gamma
	\). There are the following possibilities.
	\begin{itemize}

		\item
		If
		\(
			g = t_\alpha(\dot 0)\, t_\alpha(\zeta)
		\) or
		\(
			g = t_\alpha(\zeta)\, t_\alpha(\dot 0)
		\), then
		\(
			g_1 = g_2
		\) since
		\(
			\dot 0
		\) is the neutral element in
		\(
			P_\alpha
		\).

		\item
		If
		\(
			g = t_\alpha(\dot 0)\, t_\beta(\eta)
		\) or
		\(
			g = t_\alpha(\zeta)\, t_\beta(\dot 0)
		\) for
		\(
			\alpha > \beta
		\), then \(g_1\) and \(g_2\) have a common reduction by the identities
		\(
				C_{\alpha \beta}^\gamma(\dot 0, \eta)
			=
				C_{\alpha \beta}^\gamma(\zeta, \dot 0)
			= \dot 0
		\). One of these identities is an axiom and another one follows from the skew-symmetry. From now on we may assume that no factor of \(g\) has zero argument, i.e. the first type of rewrite rules is not applicable to \(g\).

		\item
		If
		\(
			g
			= t_\alpha(\zeta)\,
			t_\alpha(\zeta')\,
			t_\alpha(\zeta'')
		\), then
		\(
			g_1
		\) and
		\(
			g_2
		\) reduce to
		\(
			t_\alpha(
				\zeta \dotplus \zeta' \dotplus \zeta''
			)
		\) since \(\dotplus\) is associative.

		\item
		If
		\(
			g
			= t_\alpha(\zeta)\,
			t_\alpha(\zeta')\,
			t_\beta(\eta)
		\) for
		\(
			\alpha > \beta
		\), then \(g_1\) and \(g_2\) have a common reduction by an axiom.

		\item
		If
		\(
			g
			= t_\alpha(\zeta)\,
			t_\beta(\eta)\,
			t_\beta(\eta')
		\) for
		\(
			\alpha > \beta
		\), then we have
		\begin{align*}
				[\![
					t_\alpha(\zeta),
					t_\beta(\eta)]
				\!]\,
				\up{t_\beta(\eta)}{
					[\![
						t_\alpha(\zeta),
						t_\beta(\eta')
					]\!]
				}
			&=
				\bigl(
					\up{t_\beta(\eta)}{
						[\![
							t_\beta(\eta'),
							t_\alpha(\zeta)
						]\!]
					}\,
					[\![
						t_\beta(\eta),
						t_\alpha(\zeta)
					]\!]
				\bigr)^{-1} \\
			&=
				[\![
					t_\beta(\eta \dotplus \eta'),
					t_\alpha(\zeta)
				]\!]^{-1}
			=
				[\![
					t_\alpha(\zeta),
					t_\beta(\eta \dotplus \eta')
				]\!]
		\end{align*}
		in
		\(
			G_{\interval \alpha \beta \cup \{\beta\}}
		\), the second equality is an axiom.

		\item
		If
		\(
			g
			= t_\alpha(\zeta)\,
			t_\beta(\eta)\,
			t_\gamma(\theta)
		\) for
		\(
			\alpha > \beta > \gamma
		\), then it suffices to prove that
		\[
				[\![
					t_\alpha(\zeta),
					t_\beta(\eta)
				]\!]\,
				\up{t_\beta(\eta)}{
					[\![
						t_\alpha(\zeta),
						t_\gamma(\theta)
					]\!]
				}\,
				[\![
					t_\beta(\eta),
					t_\gamma(\theta)
				]\!]
			=
				\up{t_\alpha(\zeta)}{
					[\![
						t_\beta(\eta),
						t_\gamma(\theta)
					]\!]
				}\,
				[\![
					t_\alpha(\zeta),
					t_\gamma(\theta)
				]\!]\,
				\up{t_\gamma(\theta)}{
					[\![
						t_\alpha(\zeta),
						t_\beta(\eta)
					]\!]
				}
		\]
		in
		\(
			G_{\Sigma' \setminus \Sigma'_{\mathrm{ex}}}
			\rtimes \bigast_{
				\delta \in \Sigma'_{\mathrm{ex}}
			} t_\delta(P_\delta)
		\), where
		\(
			\Sigma' = \Phi \cap (
				\mathbb R_{\geq 0} \alpha
				+ \mathbb R_{\geq 0} \beta
				+ \mathbb R_{\geq 0} \gamma
			)
		\) and
		\(
			\Sigma'_{\mathrm{ex}}
		\) is the set of its extreme roots (\(
			\{\alpha, \beta\}
		\),
		\(
			\{\alpha, \gamma\}
		\), or
		\(
			\{\alpha, \beta, \gamma\}
		\)). We write this identity as
		\[
			[\![
				t_\alpha(\zeta),
				t_\beta(\eta)
			]\!]\,
			\up{t_\beta(\eta)}{
				[\![
					t_\alpha(\zeta),
					t_\gamma(\theta)
				]\!]
			}\,
			[\![
				t_\beta(\eta),
				t_\gamma(\theta)
			]\!]\,
			\up{t_\gamma(\theta)}{
				[\![
					t_\beta(\eta),
					t_\alpha(\zeta)
				]\!]
			}\,
			[\![
				t_\gamma(\theta),
				t_\alpha(\zeta)
			]\!]\,
			\up{t_\alpha(\zeta)}{
				[\![
					t_\gamma(\theta),
					t_\beta(\eta)
				]\!]
			}
			= 1.
		\]
		It is easy to see that this is symmetric under permutations of \(\alpha\), \(\beta\), \(\gamma\). Namely, the left hand side is replaced by a conjugate under even permutations and by a conjugate of the inverse under odd permutations. By an axiom, the identity holds for one of the permutations.

	\end{itemize}

	Now let us check that the group
	\(
		G_\Sigma
	\) exists. Fix a right extreme order \(\leq\) on \(\Sigma\). Let
	\(
		G_\Sigma
	\) be the set of irreducible formal strings with respect to the restricted rewrite system. This is a group, the multiplication is given by concatenation and rewriting. The commutator formula holds by the rewrite rule for
	\(
		t_\alpha(\zeta)\, t_\beta(\eta)
	\) or
	\(
		t_\beta(\eta)\, t_\alpha(\zeta)
	\) and by the definition of
	\(
		f_{\alpha \beta}^{\gamma, \preceq}
	\) for different orders \(\preceq\) (these maps are defined in order to satisfy the commutator formula, they exist by the properties of
	\(
		G_{\interval \alpha \beta}
	\)). The product map
	\(
		\prod_{\alpha \in \Sigma}^{\Set}
			P_\alpha
		\to G_\Sigma
	\) is bijective for any right extreme order \(\leq'\) on \(\Sigma\) by an argument similar to the proof of lemma \ref{comm-rel}.
\end{proof}

For example, in the case
\(
	\Phi = \mathsf A_{n - 1}
\) let
\(
	R_{ij} = P_{\mathrm e_i - \mathrm e_j}
\) and
\(
	x \times_{ijk} y
	= C_{
		\mathrm e_i - \mathrm e_j,
		\mathrm e_j - \mathrm e_k
	}^{\mathrm e_i - \mathrm e_k}(x, y)
\), so
\(
	C_{
		\mathrm e_j - \mathrm e_k,
		\mathrm e_i - \mathrm e_j
	}^{\mathrm e_i - \mathrm e_k}(x, y)
	= -y \times_{ijk} x
\). We get the axioms
\begin{itemize}

	\item
	\(
		R_{ij}
	\) are groups;

	\item
	\(
		(x \times_{ijk} y) \dotplus z
		= z \dotplus (x \times_{ijk} y)
	\) (this is the last axiom for coplanar \(\alpha\), \(\beta\), \(\gamma\));

	\item
	\(
		(x \dotplus x') \times_{ijk} y
		= x \times_{ijk} y \dotplus x' \times_{ijk} y
	\),
	\(
		x \times_{ijk} (y \dotplus y')
		= x \times_{ijk} y \dotplus x \times_{ijk} y'
	\);

	\item
	\(
		(x \times_{ijk} y) \times_{ikl} z
		= x \times_{ijl} (y \times_{jkl} z)
	\) (this is the last axiom for linearly independent \(\alpha\), \(\beta\), \(\gamma\) forming a base of a root subsystem of type
	\(
		\mathsf A_3
	\)).

\end{itemize}

\begin{lemma} \label{weyl-crit}
	Let \(G\) be a \(\Phi\)-graded group, where the rank of \(\Phi\) is at least \(2\). An element
	\(
		n \in G_\alpha\, G_{-\alpha}\, G_\alpha
	\) is \(\alpha\)-Weyl if and only if
	\(
		\up n{G_\beta}
		\leq G_{s_\alpha(\beta)}
	\) for all
	\(
		\beta \in \Phi \setminus \mathbb R \alpha
	\).
\end{lemma}
\begin{proof}
	Suppose that
	\(
		\up n{G_\beta} \leq G_{s_\alpha(\beta)}
	\) for all \(\beta\) linearly independent with \(\alpha\). It suffices to prove that
	\(
		\up n{G_{\pm \alpha}}
		\leq G_{\mp \alpha}
	\). Indeed, since \(G\) is \(\Phi\)-graded, the group
	\(
		P_{\pm \alpha}
	\) is generated by the images of all
	\(
		C_{\beta \gamma}^{\pm \alpha}
	\). Clearly,
	\[
		\up n{t_{\pm \alpha}\bigl(
			C_{\beta \gamma}^{\pm \alpha}(
				P_\beta,
				P_\gamma
			)
		\bigr)}
		\leq t_{\mp \alpha}\bigl(
			C_{
				s_\alpha(\beta), s_\alpha(\gamma)
			}^{\mp \alpha}(
				P_{s_\alpha(\beta)},
				P_{s_\alpha(\gamma)}
			)
		\bigr).\qedhere
	\]
\end{proof}

\section{Groups with \(\mathsf B_\ell\)-commutator relations and \(\mathsf A_{\ell - 1}\)-Weyl elements}

From now on we assume that
\(
	\ell \geq 3
\). Let
\(
	\mathrm e_{-i} = -\mathrm e_i
\) for
\(
	1 \leq i \leq \ell
\). It is convenient to use the following variation of
\(
	(\mathsf B_\ell, \varnothing)
\)-rings called \textit{partial graded odd form rings} of rank \(\ell\) \cite{bc-comm-rel}. Instead of
\(
	P_\alpha
\) for long
\(
	\alpha = \mathrm e_j - \mathrm e_i
\) we use two different groups
\(
	R_{ij}
\) and
\(
	R_{-j, -i}
\) with the group operations \(+\) connected by mutually inverse anti-isomorphisms
\(
	\inv{(-)} \colon R_{ij} \to R_{-j, -i}
\) and
\(
	\inv{(-)} \colon R_{-j, -i} \to R_{ij}
\), here
\(
	i, j \in \{-\ell, \ldots, -1, 1, \ldots, \ell\}
\) have distinct absolute values. The homomorphisms
\(
	t_{\mathrm e_j - \mathrm e_i}
	\colon P_{\mathrm e_j - \mathrm e_i}
	\to G_\Sigma
\) are replaced by
\(
	t_{ij} \colon R_{ij} \to G_\Sigma
\) and
\(
	t_{-j, -i} \colon R_{-j, -i} \to G_\Sigma
\) such that
\[
	t_{ij}(a) = t_{-j, -i}(-\inv a).
\]
Also, we write
\(
	\Delta^0_i
\) instead of
\(
	P_{\mathrm e_i}
\) and
\(
	t_i \colon \Delta^0_i \to G_\Sigma
\) instead of
\(
	t_{\mathrm e_i}
\), here
\(
	i \in \{-\ell, \ldots, -1, 1, \ldots, \ell\}
\). The operations
\(
	C_{\alpha \beta}^\gamma
\) are replaced by certain binary operations from the following commutator formulae.
\begin{align*}
	[t_{ij}(x), t_{jk}(y)]
	&= t_{ik}(x y)
	\text{ for distinct } |i|, |j|, |k|;
\\
	[t_{-i, j}(x), t_{ji}(y)]
	&= t_i(x * y)
	\text{ for } |i| \neq |j|;
\\
	[t_i(u), t_j(v)]
	&= t_{-i, j}(-u \circ v)
	\text{ for } |i| \neq |j|;
\\
	[t_i(u), t_{ij}(x)]
	&= t_{-i, j}(u \triangleleft x)\,
	t_j(\dotminus u \cdot (-x))
	\text{ for } |i| \neq |j|.
\end{align*}

We explicitly list all axioms of these objects using theorem \ref{phi-0-ring} assuming that all indices have distinct absolute values.
\begin{itemize}

	\item
	The objects \(
		R_{ij}
	\) and
	\(
		\Delta^0_i
	\) are groups with the group operations \(+\) and \(\dotplus\) respectively;

	\item
	\(
		\inv{(-)} \colon R_{ij} \to R_{-j, -i}
	\) are involutions, i.e. anti-isomorphisms with the property
	\(
		\inv{\inv x} = x
	\);

	\item
	\(
		R_{ij} \times R_{jk} \to R_{ik},\,
		(x, y) \mapsto x y
	\) are biadditive with central images,
	\(
		\inv{(x y)} = \inv y \inv x
	\);

	\item
	\(
		R_{-i, j} \times R_{ji} \to \Delta^0_i,\,
		(x, y) \mapsto x * y
	\) are biadditive with central images,
	\(
		x * y \dotplus \inv y * \inv x = \dot 0
	\);

	\item
	\(
		\Delta^0_i \times \Delta^0_j \to R_{-i, j},\,
		(u, v) \mapsto u \circ v
	\) are biadditive with central images,
	\(
		\inv{u \circ v} = v \circ u
	\);

	\item
	\(
			u \circ (x * y)
			+ u \triangleleft (-\inv x)
			+ y
		=
			y
			+ u \triangleleft (-\inv x)
	\) holds for
	\(
		u \in \Delta^0_i
	\),
	\(
		x \in R_{-j, -i}
	\),
	\(
		y \in R_{-i, j}
	\);

	\item
	\(
			u \cdot x
			\dotplus v
		=
			v
			\dotplus u \cdot x
			\dotminus \inv x * (u \circ v)
	\) holds for
	\(
		u \in \Delta^0_i
	\),
	\(
		x \in R_{ij}
	\),
	\(
		v \in \Delta^0_j
	\);

	\item
	\(
		\Delta^0_i \times R_{ij} \to R_{-i, j},\,
		(u, x) \mapsto u \triangleleft x
	\) satisfy
	\(
			(u \dotplus v) \triangleleft x
		=
			v \triangleleft x
			+ u \circ (v \cdot (-x))
			+ u \triangleleft x
	\) and
	\(
		u \triangleleft (x + y)
		= u \triangleleft x + u \triangleleft y
	\);

	\item
	\(
		\Delta^0_i \times R_{ij} \to \Delta^0_j,\,
		(u, x) \mapsto u \cdot x
	\) satisfy
	\(
		(u \dotplus v) \cdot x
		= u \cdot x \dotplus v \cdot x
	\) and
	\(
			u \cdot (x + y)
		=
			u \cdot x
			\dotplus \inv y * (u \triangleleft x)
			\dotplus u \cdot y
	\);

	\item
	\(
		(x y)\, z = x\, (y z)
	\) holds for
	\(
		x \in R_{ij}
	\),
	\(
		y \in R_{jk}
	\),
	\(
		z \in R_{kl}
	\) (this identity is vacuous for
	\(
		\ell = 3
	\));

	\item
	\(
		x y * z = x * y z
	\) holds for
	\(
		x \in R_{-i, j}
	\),
	\(
		y \in R_{jk}
	\),
	\(
		z \in R_{ki}
	\);

	\item
	\(
		u \circ (x * y) = 0
	\) holds for
	\(
		u \in \Delta^0_i
	\),
	\(
		x \in R_{-j, k}
	\),
	\(
		y \in R_{kj}
	\);

	\item
	\(
		u \circ (v \cdot x) = (u \circ v)\, x
	\) holds for
	\(
		u \in \Delta^0_i
	\),
	\(
		v \in \Delta^0_j
	\),
	\(
		x \in R_{jk}
	\);

	\item
	\(
		\inv x\, (u \triangleleft y)
		+ (u \cdot x) \circ (u \cdot y)
		+ \inv{(u \triangleleft x)}\, y
		= 0
	\) holds for
	\(
		u \in \Delta^0_i
	\),
	\(
		x \in R_{ij}
	\),
	\(
		y \in R_{ik}
	\);

	\item
	\(
		(x * y) \triangleleft z
		= x\, (y z) - \inv y\, (\inv x z)
	\) and
	\(
		(x * y) \cdot z = \inv z x * y z
	\) hold for
	\(
		x \in R_{-i, j}
	\),
	\(
		y \in R_{ji}
	\),
	\(
		z \in R_{ik}
	\);

	\item
	\(
		(u \triangleleft x)\, y = u \triangleleft x y
	\),
	\(
		(\dotminus u \cdot x) \triangleleft y
		= \inv{(u \triangleleft x)}\, (x y)
	\), and
	\(
		(u \cdot x) \cdot y = u \cdot x y
	\) hold for
	\(
		u \in \Delta^0_i
	\),
	\(
		x \in R_{ij}
	\),
	\(
		y \in R_{jk}
	\).

\end{itemize}

The following lemma is essentially a possible replacement of the axiom
\(
	(\dotminus u \cdot x) \triangleleft y
	= \inv{(u \triangleleft x)}\, (x y)
\).

\begin{lemma} \label{tri-dot}
	The identity
	\(
		(u \cdot x) \triangleleft y
		= \inv x\, (u \triangleleft x y)
	\) holds for
	\(
		u \in \Delta^0_i
	\),
	\(
		x \in R_{ij}
	\),
	\(
		y \in R_{jk}
	\), where \(|i|\), \(|j|\), \(|k|\) are distinct.
\end{lemma}
\begin{proof}
	Indeed,
	\begin{align*}
			(u \cdot x) \triangleleft y
		&=
			\inv{(
				(\dotminus u) \triangleleft x
			)}\, (x y)
		=
			- \inv x\, (
				(\dotminus u) \triangleleft x y
			)
			- (u \cdot x) \circ (u \cdot x y) \\
		&=
			- \inv x\, \bigl(
				u \circ (u \cdot (-x y))
				- u \triangleleft x y
			\bigr)
			- (u \cdot x) \circ (u \cdot x y) \\
		&=
			\inv x\, (u \triangleleft x y)
			- (u \cdot x) \circ (
				u \cdot (-x y)
				\dotplus u \cdot x y
			)
		=
			\inv x\, (u \triangleleft x y). \qedhere
	\end{align*}
\end{proof}

From now on we fix a group \(G\) with
\(
	\mathsf B_\ell
\)-commutator relations. Let
\(
	(R_{ij}, \Delta^0_i)_{ij}
\) be the corresponding partial graded odd form ring. Suppose that \(G\) has Weyl elements
\[
	n_{i, i + 1}
	= t_{i, i + 1}(a_i)\,
	t_{i + 1, i}(b_i)\,
	t_{i, i + 1}(c_i)
\]
for all
\(
	1 \leq i \leq \ell - 1
\), so \(G\) has \(\alpha\)-Weyl elements for all roots \(\alpha\) from the root subsystem
\(
	\mathsf A_{\ell - 1}
\). The next lemma implies that
\(
	R_{ij}
\) are abelian groups.

\begin{lemma} \label{ba-ring-1}
	Let
	\(
		n = t_{ij}(a)\, t_{ji}(b)\, t_{ij}(c)
	\) be a Weyl element. Then for any
	\(
		k \notin \{i, -i, j, -j\}
	\) we have
	\begin{itemize}

		\item
		\(
			\up n{t_{ik}(x)} = t_{jk}(b x)
		\) and
		\(
			a\, (b x) = -x
		\) for
		\(
			x \in R_{ik}
		\);

		\item
		\(
			\up n{t_{jk}(x)} = t_{ik}(c x)
		\) and
		\(
			b\, (c x) = -x
		\) for
		\(
			x \in R_{jk}
		\);

		\item
		\(
			\up n{t_{ki}(x)} = t_{kj}(-x c)
		\) and
		\(
			(x c)\, b = -x
		\) for
		\(
			x \in R_{ki}
		\);

		\item
		\(
			\up n{t_{kj}(x)} = t_{ki}(-x b)
		\) and
		\(
			(x b)\, a = -x
		\) for
		\(
			x \in R_{kj}
		\);

		\item
		\(
			\up n{t_i(u)} = t_j(u \cdot (-c))
		\),
		\(
			\up n{t_j(u)} = t_i(u \cdot (-b))
		\);

		\item
		\(
			\up n{t_{\pm i, j}(x)}
			= t_{\pm j, i}(W(x))
		\) for a unique group isomorphisms
		\(
			R_{\pm i, j} \to R_{\pm j, i},\,
			x \mapsto W(x)
		\);

		\item
		\(
			W(x y) = -(b x)\, (y b)
		\) for
		\(
			x \in R_{ik}
		\),
		\(
			y \in R_{kj}
		\) and
		\(
			W(x y) = (\inv{\,c\,} x)\, (y b)
		\) for
		\(
			x \in R_{-i, k}
		\),
		\(
			y \in R_{kj}
		\);

		\item
		\(
			W(x)\, (c y) = b\, (x y)
		\) for
		\(
			x \in R_{ij}
		\),
		\(
			y \in R_{jk}
		\) and
		\(
			W(x)\, (c y) = -\inv{\,c\,}\, (x y)
		\) for
		\(
			x \in R_{-i, j}
		\),
		\(
			y \in R_{jk}
		\);

		\item
		\(
			(x c)\, W(y) = (x y)\, b
		\) for
		\(
			x \in R_{ki}
		\),
		\(
			y \in R_{ij}
		\) and
		\(
			(x \inv{\,b\,})\, W(y) = -(x y)\, b
		\) for
		\(
			x \in R_{k, -i}
		\),
		\(
			y \in R_{-i, j}
		\);

		\item
		\(
			a = c
		\).

	\end{itemize}
\end{lemma}
\begin{proof}
	The first five properties follow by direct calculations inside various groups
	\(
		G_\Sigma
	\), the six one is clear. Using them, the next three properties follow from conjugating the commutator formula
	\(
		[t_{pq}(x), t_{qr}(y)] = t_{pr}(x y)
	\) by \(n\) for appropriate \(p\), \(q\), \(r\). Now take an index \(k\) such that
	\(
		\sign(k) = \sign(j)
	\) and
	\(
		|i|
	\),
	\(
		|j|
	\),
	\(
		|k|
	\) are distinct. Choose a Weyl element
	\(
		n' = t_{jk}(a')\, t_{kj}(b')\, t_{jk}(c')
	\), then
	\[
			a
		=
			-(a c')\, b'
		=
			\bigl(a\, (b\, (c c'))\bigr)\, b'
		=
			-(c c')\, b'
		=
			c. \qedhere
	\]
\end{proof}

Let
\(
	e_{ij} = a_i \cdots a_{j - 1}
\) for
\(
	1 \leq i < j \leq \ell
\),
\(
	e_{ij} = (-1)^{i - j}\, b_{i - 1} \cdots b_j
\) for \(
	1 \leq j < i \leq \ell
\), and
\(
	e_{ij} = \inv{e_{-j, -i}}
\) for distinct
\(
	i, j < 0
\). In particular,
\(
	a_i = c_i = e_{i, i + 1}
\) and
\(
	b_i = -e_{i + 1, i}
\).

\begin{lemma} \label{ba-ring-2}
	The elements \(
		e_{ij}
	\) satisfy the following.
	\begin{itemize}

		\item
		\(
			e_{ij} e_{jk} = e_{ik}
		\) for distinct \(i\), \(j\), \(k\) of the same sign.

		\item
		\(
				n_{ij}
			=
				t_{ij}(e_{ij})\,
				t_{ji}(-e_{ji})\,
				t_{ij}(e_{ij})
			=
				n_{ji}^{-1}
			=
				n_{-i, -j}
		\) are Weyl elements for distinct \(i\) and \(j\) of the same sign.

		\item
		\(
				\up{n_{ij}}{n_{jk}}
			=
				n_{ik}
			=
				n_{ij}^{n_{jk}}
		\) for distinct \(i\), \(j\), \(k\) of the same sign.

		\item
		\(
			[n_{ij}, n_{kl}] = 1
		\) for distinct \(i\), \(j\), \(k\), \(l\) of the same sign.

	\end{itemize}
\end{lemma}
\begin{proof}
	By definition,
	\(
		e_{ij} e_{jk} = e_{ik}
	\) if
	\(
		i < j < k
	\) or
	\(
		i > j > k
	\). Since
	\(
		n_{i, i + 1}
	\) are Weyl elements,
	\(
		e_{ji} e_{i, i \pm 1} = e_{j, i \pm 1}
	\) and
	\(
		e_{i \pm 1, i} e_{ij} = e_{i \pm 1, j}
	\) for
	\(
		j \notin \{i, i \pm 1\}
	\) by lemma \ref{ba-ring-1}. In particular,
	\(
		n_{ij}
		= \up{n_{i, i + 1}}{n_{i + 1, j}}
	\) for
	\(
		j \notin \{i, i + 1\}
	\). An easy induction shows that
	\(
		n_{ij} = n_{-i, -j}
	\) and
	\(
		n_{ji} = n_{-j, -i} = n_{ij}^{-1}
	\) are Weyl elements for
	\(
		0 < i < j
	\). We have
	\(
		(x e_{ij})\, e_{ji} = x
	\) for
	\(
		x \in R_{ki}
	\) and
	\(
		e_{ji}\, (e_{ij} y) = y
	\) for
	\(
		y \in R_{jk}
	\) by lemma \ref{ba-ring-1}, where
	\(
		|k| \notin \{|i|, |j|\}
	\). In particular,
	\(
		e_{ij} e_{jk} = e_{ik}
	\) and
	\(
			\up{n_{ij}}{n_{jk}}
		=
			n_{ik}
		=
			n_{ij}^{n_{jk}}
	\) for distinct \(i\), \(j\), \(k\) of the same sign. The last property is obvious.
\end{proof}

It turns out that in our case
\(
	(R_{ij}, \Delta^0_i)_{ij}
\) has a matrix-like structure.

\begin{lemma} \label{ba-ring-3}
	If
	\(
		R_{ij}
	\) and
	\(
		R_{kl}
	\) are defined,
	\(
		\sign(i) = \sign(k)
	\), and
	\(
		\sign(j) = \sign(l)
	\), then there exists a unique isomorphism
	\(
		R_{ij} \cong R_{kl}
	\) composed from the left and right multiplications by
	\(
		e_{pq}
	\). If
	\(
		\sign(i) = \sign(j)
	\), then there exists a unique isomorphism
	\(
		\Delta^0_i \cong \Delta^0_j
	\) composed from the maps
	\(
		({-}) \cdot e_{pq}
	\). All operations of partial graded odd form rings are preserved under these isomorphisms.
\end{lemma}
\begin{proof}
	The existence of isomorphisms is clear. In order to prove the uniqueness for
	\(
		R_{ij} \cong R_{kl}
	\) it suffices to check the following identities, where all indices have distinct absolute values.
	\begin{align*}
		(x e_{ij})\, e_{ji} &= x
		\text{ for } x \in R_{ki}.
	\\
		(x e_{ij})\, e_{jk} &= x e_{ik}
		\text{ for } x \in R_{li}.
	\\
		(e_{ij} x)\, e_{kl} &= e_{ij}\, (x e_{kl})
		\text{ for } x \in R_{jk}.
	\\
		(
			e_{\eps j, \eps i}\,
			(x e_{jk})
		) e_{ki}
		&= e_{\eps j, \eps k}\, (
			(e_{\eps k, \eps i} x)\,
			e_{ji}
		)
		\text{ for } x \in R_{\eps i, j}
		\text{ and } \eps = \pm 1.
	\end{align*}
	Indeed, if
	\(
		R_{ij} \to R_{ij}
	\) is a composition of such multiplications, then we may use the first two identities (and their duals) to reduce the number of multiplications. Such identities are not applicable if left and right multiplications alternate. But in this case either we may apply the third identity (and then reduce the number of multiplications), or all indices take only \(3\) absolute values and we may apply the last identity. The first three identities easily follow from lemmas \ref{ba-ring-1}, \ref{ba-ring-2}, and
	\(
		\up{n_{ij}}{t_{kl}(x)} = t_{kl}(x)
	\) for indices with distinct absolute values. The last identity follows from lemma \ref{ba-ring-1} and
	\[
			e_{\eps j, \eps k}\, (
				(e_{\eps k, \eps i} x)\,
				e_{ji}
			)
		=
			-W_{ji}(
				e_{\eps i, \eps k}\,
				(e_{\eps k, \eps i} x)
			)
		=
			-W_{ji}(x)
		=
			-W_{ji}((x e_{jk})\, e_{kj})
		=
			(
				e_{\eps j, \eps i}\,
				(x e_{jk})
			)\, e_{ki}.
	\]

	The uniqueness claim for
	\(
		\Delta^0_i \cong \Delta^0_j
	\) means that
	\(
		(u \cdot e_{ij}) \cdot e_{ji} = u
	\) for
	\(
		u \in \Delta^0_i
	\). This follows from lemma \ref{ba-ring-1} and the identity
	\(
		n_{ij}^{-1} = n_{ji}
	\) from lemma \ref{ba-ring-2}.

	It remains to check that the operations are preserved under these isomorphisms. The only non-trivial case is the multiplication maps
	\(
		R_{ij} \times R_{jk} \to R_{ik}
	\) since for other operations we may just apply corresponding axioms of partial graded odd form rings and lemma \ref{tri-dot}. If
	\(
		|l| \notin \{|i|, |j|, |k|\}
	\), then
	\(
		(x y)\, e_{kl} = x\, (y e_{kl})
	\),
	\(
		e_{li}\, (x y) = (e_{li} x)\, y
	\),
	\(
		x y = (x e_{jl})\, (e_{lj} y)
	\) for
	\(
		x \in R_{ij}
	\),
	\(
		y \in R_{jk}
	\). Also,
	\[
			(e_{\pm k, i} x)\,
			(y e_{k, \pm i})
		=
			-W_{\pm k, i}(x y)
		=
			-W_{\pm k, i}\bigl(
				((x y)\, e_{k, \eps j})\,
				e_{\eps j, k}
			\bigr)
		=
			\bigl(
				e_{\pm k, i}\,
				((x y)\, e_{k, \eps j})
			\bigr)\, e_{\eps j, \pm i}
	\]
	for
	\(
		\eps = \pm 1
	\) and
	\[
		e_{\pm j, i}\, (x y)
		= -W_{\pm i, j}(x)\, (e_{\pm i, j} y)
	\]
	for such \(x\), \(y\) by lemma \ref{ba-ring-1}.
\end{proof}

Using lemma \ref{ba-ring-3} we identify the groups
\(
	R_{ij}
\) and
\(
	\Delta^0_i
\) using the canonical isomorphisms. We call the resulting object a
\(
	(\mathsf B_\ell, \mathsf A_{\ell - 1})
\)-\textit{ring}. Some binary operations may be replaced by unary ones, namely, let
\(
	x * y = \phi(x y)
\), where
\(
	\phi(x) = x * e_{ij}
\), and
\(
	u \triangleleft x = \widehat \rho(u)\, x
\), where
\(
	\widehat \rho(u) = u \triangleleft e_{ij}
\).

Formally, a
\(
	(\mathsf B_\ell, \mathsf A_{\ell - 1})
\)-ring consists of
\begin{itemize}

	\item
	abelian groups
	\(
		R_{i j}
	\) for
	\(
		i, j \in \{{-}, {+}\}
	\);

	\item
	involutions
	\(
		R_{i j} \to R_{-j, -i},\, x \mapsto \inv x
	\), i.e. isomorphisms with the property
	\(
		\inv{\inv x} = x
	\);

	\item
	biadditive multiplication maps
	\(
		R_{i j} \times R_{j k} \to R_{i k}
	\) such that
	\(
		\inv{(x y)} = \inv y \inv x
	\);

	\item
	two-sided units
	\(
		1_i \in R_{i i}
	\) for the multiplication maps (necessarily
	\(
		\inv{1_i} = 1_{-i}
	\));

	\item
	groups
	\(
		\Delta^0_-
	\) and
	\(
		\Delta^0_+
	\) with the group operations \(\dotplus\);

	\item
	group homomorphisms
	\(
		\phi \colon R_{-i, i} \to \Delta^0_i
	\) with central images such that
	\(
		\phi(x + \inv x) = \dot 0
	\);

	\item
	biadditive maps
	\(
		\Delta^0_i \times \Delta^0_j \to R_{-i, j},\,
		(u, v) \mapsto u \circ v
	\) such that
	\(
		\inv{u \circ v} = v \circ u
	\);

	\item
	maps
	\(
		\widehat \rho \colon \Delta^0_i \to R_{-i, i}
	\) such that
	\(
		\widehat \rho(u \dotplus v)
		= \widehat \rho(u)
		- u \circ v
		+ \widehat \rho(v)
	\) and
	\(
		\widehat \rho(u)
		+ u \circ u
		+ \inv{\widehat \rho(u)}
		= 0
	\);

	\item
	maps
	\(
		({-}) \cdot ({=})
		\colon \Delta^0_i \times R_{i j}
		\to \Delta^0_j
	\) such that
	\(
		(u \dotplus v) \cdot x
		= u \cdot x \dotplus v \cdot x
	\) and
	\(
			u \cdot (x + y)
		=
			u \cdot x
			\dotplus \phi\bigl(
				\inv y\, (\widehat \rho(u) x)
			\bigr)
			\dotplus u \cdot y
	\)

\end{itemize}
Moreover, the operations satisfy additional axioms
\begin{align*}
	u \dotplus v
	&= v \dotplus u \dotminus \phi(u \circ v),
&
	\phi(x y) \cdot z
	&= \phi((\inv z x)\, (y z)),
\\
	\phi((x y)\, z) &= \phi(x\, (y z)),
&
	\widehat \rho(u \cdot x)\, y
	&= (\inv x\, \widehat \rho(u))\, (x y),
\\
	u \circ \phi(x) &= 0,
&
	u \circ (v \cdot x) &= (u \circ v)\, x,
\\
	\widehat \rho(\phi(x y))\, z
	&= x\, (y z) - \inv y\, (\inv x z),
&
	(u \cdot x) \cdot y &= u \cdot x y,
\\
	(\widehat \rho(u)\, x)\, y
	&= \widehat \rho(u)\, (x y),
&
	u \cdot 1_i &= u \text{ for } u \in \Delta^0_i.
\end{align*}
Finally, the multiplication is associative if
\(
	\ell \geq 4
\). Other than this requirement, the definition is independent of \(\ell\).

\begin{lemma} \label{weyl-a-long}
	Let \(G\) be a group with
	\(
		\mathsf B_\ell
	\)-commutator relations and
	\(
		\mathsf A_{\ell - 1}
	\)-Weyl elements for
	\(
		\ell \geq 3
	\). We parameterize the root subgroups by components of the corresponding
	\(
		(\mathsf B_\ell, \mathsf A_{\ell - 1})
	\)-ring
	\(
		(R_{ij}, \Delta^0_i)_{
			i, j \in \{{-}, {+}\}
		}
	\). An element
	\(
		n = t_{ij}(a)\, t_{ji}(b)\, t_{ij}(c)
	\) is Weyl if and only if
	\begin{itemize}

		\item
		\(
			a = c
		\) is invertible,
		\(
			b = -a^{-1}
		\) (i.e.
		\(
			a b = -1_i
		\) and
		\(
			b a = -1_j
		\));

		\item
		\(
			a^{-1}\, (a x) = x
		\),
		\(
			a\, (a^{-1} x) = x
		\),
		\(
			(x a)\, a^{-1} = x
		\),
		\(
			(x a^{-1})\, a = x
		\);

		\item
		\(
			(\zeta x)\, \eta = \zeta\, (x \eta)
		\),
		\(
			(\zeta x)\, (y \zeta)
			= \zeta\, (x y)\, \zeta
		\),
		\(
			\zeta\, (x\, (\zeta y))
			= (\zeta x \zeta)\, y
		\),
		\(
			((x \zeta)\, y)\, \zeta
			= x\, (\zeta y \zeta)
		\) for
		\(
			\zeta, \eta \in \{
				a, a^{-1}, \inv a, \inv a^{-1}
			\}
		\).

	\end{itemize}
	If
	\(
		|k| \notin \{|i|, |j|\}
	\), then
	\begin{align*}
		\up n{t_{ik}(x)} &= t_{jk}(-a^{-1} x),
	&
		\up n{t_{-i, j}(x)}
		&= t_{-i, j}(\inv a^{-1} \inv x a),
	&
		\up n{t_i(u)} &= t_j(u \cdot (-a)),
	\\
		\up n{t_{jk}(x)} &= t_{ik}(a x),
	&
		\up n{t_{i, -j}(x)}
		&= t_{i, -j}(a \inv x \inv a^{-1}),
	&
		\up n{t_{-i}(u)}
		&= t_{-j}(u \cdot (-\inv a^{-1})),
	\\
		\up n{t_{ki}(x)} &= t_{kj}(-x a),
	&
		\up n{t_{ij}(x)} &= t_{ji}(-a^{-1} x a^{-1}),
	&
		\up n{t_j(u)} &= t_i(u \cdot a^{-1}),
	\\
		\up n{t_{kj}(x)} &= t_{ki}(x a^{-1}),
	&
		\up n{t_{ji}(x)} &= t_{ij}(-a x a),
	&
		\up n{t_{-j}(u)} &= t_{-i}(u \cdot \inv a).
	\end{align*}
\end{lemma}
\begin{proof}
	The conditions for \(n\) to be a Weyl element and the following group of formulae follow from lemma \ref{ba-ring-1} applied to
	\(
			n
		=
			t_{ij}(a)\, t_{ji}(b)\, t_{ij}(a)
		=
			t_{ji}(b)\, t_{ij}(a)\, t_{ji}(b)
	\) and
	\(
		n^{-1}
	\). Conversely, if the conditions hold, then it may be checked using the axioms of
	\(
		(\mathsf B_\ell, \mathsf A_{\ell - 1})
	\)-rings, lemma \ref{weyl-crit}, and calculations in various
	\(
		G_\Sigma
	\) that \(n\) is indeed a Weyl element.
\end{proof}

\begin{lemma} \label{weyl-a-short}
	Let \(G\) be a group with
	\(
		\mathsf B_\ell
	\) and
	\(
		\mathsf A_{\ell - 1}
	\)-Weyl elements for
	\(
		\ell \geq 3
	\). We parameterize the root subgroups by components of the corresponding
	\(
		(\mathsf B_\ell, \mathsf A_{\ell - 1})
	\)-ring
	\(
		(R_{ij}, \Delta^0_i)_{
			i, j \in \{{-}, {+}\}
		}
	\). An element
	\(
		n = t_i(u)\, t_{-i}(v)\, t_i(w)
	\) is Weyl if and only if
	\(
		\widehat \rho(u) = \widehat \rho(w)
	\) is invertible (i.e. there is
	\(
		a \in R_{i, -i}
	\) such that
	\(
		a\, \widehat \rho(u) = 1_i
	\) and
	\(
		\widehat \rho(u)\, a = 1_{-i}
	\)),
	\(
		\widehat \rho(v)
		= \inv{\widehat \rho(w)}^{-1}
	\),
	\(
		v
		= w \cdot (-\widehat \rho(w)^{-1})
	\), and
	\(
		u
		= w
		\cdot \widehat \rho(w)^{-1}\,
		\inv{\widehat \rho(w)}
	\). If
	\(
		i \neq \pm j
	\), then,
	\begin{align*}
		\up n{t_{ij}(x)}
		&= t_{-i, j}(\widehat \rho(w)\, x),
	\qquad
		\up n{t_{-i, j}(x)}
		= t_{ij}(\inv{\widehat \rho(w)}^{-1}\, x),
	\\
		\up n{t_j(\zeta)}
		&= t_j(
			\zeta
			\dotminus v \cdot (w \circ \zeta)
		),
	\\
		\up n{t_i(\zeta)} &= t_{-i}\bigl(
			(
				\zeta
				\dotminus v \cdot (w \circ \zeta)
			) \cdot \widehat \rho(w)^{-1}
		\bigr),
	\\
		\up n{t_{-i}(\zeta)} &= t_i\bigl(
			(
				\zeta
				\dotminus v \cdot (w \circ \zeta)
			) \cdot \inv{\widehat \rho(w)}
		\bigr).
	\end{align*}
\end{lemma}
\begin{proof}
	The expressions
	\(
		\up n{t_\alpha(\zeta)}
	\) for
	\(
		\alpha \neq \pm \mathrm e_i
	\) may be simplified using calculations in corresponding
	\(
		G_\Sigma
	\) and axioms of
	\(
		(\mathsf B_\ell, \mathsf A_{\ell - 1})
	\)-rings. It turns out that these expressions are of type
	\(
		t_{s_{\mathrm e_i}(\alpha)}(\eta)
	\) if and only if \(u\), \(v\), \(w\) satisfy the conditions from the statement. By lemma \ref{weyl-crit} these conditions are necessary and sufficient for \(n\) to be a Weyl element. The same calculations give explicit values of the arguments \(\eta\). Finally,
	\(
		\up n{t_{\pm i}(\zeta)}
	\) may be calculated using conjugation by \(n\) the identity
	\[
		\bigl[
			t_{\pm j}(\dotminus \zeta),
			t_{\pm j, \pm i}(-1_{\pm \sign(i)})
		\bigr]
		= t_{\mp i, \pm j}(\widehat \rho(\zeta))\,
		t_{\pm i}(\zeta),
	\]
	where
	\(
		j \neq \pm i
	\) has the same sign as \(i\).
\end{proof}

\section{\(\mathsf{ADE}\)-graded groups}

Recall that if \(R\) is an associative unital ring and
\(
	m \geq 3
\), then the \textit{Steinberg group}
\(
	\stlin(m, R)
\) is the abstract group with the generators
\(
	t_{ij}(x)
\) for
\(
	1 \leq i \neq j \leq m
\),
\(
	x \in R
\) and the relations
\begin{align*}
	t_{ij}(x)\, t_{ij}(x') &= t_{ij}(x + x');
\\
	[t_{ij}(x), t_{jk}(y)] &= t_{ik}(x y)
	\text{ for } i \neq k;
\\
	[t_{ij}(x), t_{kl}(y)] &= 1
	\text{ for } j \neq k \text{ and } i \neq l.
\end{align*}
This group is
\(
	\mathsf A_{m - 1}
\)-graded. Similarly, if \(K\) is a commutative unital ring and \(\Phi\) is a crystallographic root system of rank \(\geq 2\), then the Steinberg group
\(
	\stlin(\Phi, K)
\) is generated by
\(
	t_\alpha(x)
\) for
\(
	\alpha \in \Phi
\),
\(
	x \in K
\) with the relations
\[
	t_\alpha(x)\, t_\alpha(x') = t_\alpha(x + x');
\quad
	[t_\alpha(x), t_\beta(y)]
	= \prod_{
		\substack{
			i \alpha + j \beta \in \Phi
		\\
			i, j \in \mathbb N_{> 0}
		}
	} t_{i \alpha + j \beta}(
		N_{\alpha \beta i j} x^i y^j
	).
\]
This group is \(\Phi\)-graded.

\begin{theorem} \label{group-a}
	Let \(G\) be an
	\(
		\mathsf A_\ell
	\)-graded group for
	\(
		\ell \geq 3
	\). Then there exists an associative unital ring \(R\) and a homomorphism
	\(
		Q \colon \stlin(\ell + 1, R) \to G
	\) inducing isomorphisms between the root subgroups (so we may identify
	\(
		P_\alpha
	\) with \(R\)).

	Any other such homomorphism
	\(
		Q' \colon \stlin(\ell + 1, R') \to G
	\) is of the following type. Choose a ring isomorphism
	\(
		F \colon R' \to R
	\)
	and elements
	\(
		a_1, \ldots, a_{\ell + 1} \in R^*
	\) and let
	\(
		Q'(t_{ij}(x))
		= Q(
			t_{ij}(a_i\, F(x)\, a_j^{-1})
		)
	\). The tuple
	\(
		(F, a_1, \ldots, a_{\ell + 1})
	\) is uniquely determined by \(Q'\) up to the change
	\(
		F \mapsto c^{-1} F c
	\),
	\(
		a_i \mapsto a_i c
	\) for
	\(
		c \in R^*
	\).
\end{theorem}
\begin{proof}
	We may consider \(G\) as a group with
	\(
		\mathsf B_{\ell + 1}
	\) commutator relations by taking
	\(
		G_\alpha = 1
	\) for additional roots \(\alpha\). Let
	\(
		(R_{ij}, \Delta^0_i)_{
			i, j \in \{{-}, {+}\}
		}
	\) be the
	\(
		(\mathsf B_{\ell + 1}, \mathsf A_\ell)
	\)-ring constructed by \(G\). Clearly,
	\(
		R_{+-} = R_{-+} = 0
	\),
	\(
		\Delta^0_{-} = \Delta^0_{+} = \dot 0
	\), and
	\(
		R_{--} \cong R_{++}^\op,\,
		x \mapsto \inv x
	\). By definition of
	\(
		(\mathsf B_{\ell + 1}, \mathsf A_\ell)
	\)-rings,
	\(
		R_{++}
	\) is an associative unital ring and the relations between the root elements
	\(
		t_{ij}(x)
	\) are the same as in the Steinberg group.

	Suppose that
	\(
		Q \colon \stlin(\ell + 1, R) \to G
	\) and
	\(
		Q' \colon \stlin(\ell + 1, R') \to G
	\) are group homomorphisms inducing isomorphisms between root subgroups. There are additive isomorphisms
	\(
		F_{ij} \colon R' \to R
	\) for
	\(
		1 \leq i \neq j \leq \ell + 1
	\) such that
	\(
		Q'(t_{ij}(x))
		= Q\bigl( t_{ij}(F_{ij}(x)) \bigr)
	\). It follows that
	\(
		F_{ij}(a)\, F_{jk}(b) = F_{ik}(a b)
	\) for distinct \(i\), \(j\), \(k\). In particular,
	\[
			F_{ij}(1)\, F_{ji}(1)\, F_{ik}(a)
		=
			F_{ij}(1)\, F_{jk}(a)
		=
			F_{ik}(a)
	\]
	for
	\(
		k \notin \{i, j\}
	\), so
	\(
		F_{ij}(1)\, F_{ji}(1) = 1
	\). Let
	\(
		a_i = F_{i, \ell + 1}(1)
	\) for
	\(
		1 \leq i \leq \ell
	\),
	\(
		a_{\ell + 1} = 1
	\), and
	\(
		F(x) = F_{\ell + 1, 1}(x)\, a_1
	\). Then
	\(
		F_{ij}(1) = a_i a_j^{-1}
	\),
	\(
		F_{ij}(x) = a_i\, F(x)\, a_j^{-1}
	\),
	\(
		F(1) = 1
	\), and
	\(
		F(x y) = F(x)\, F(y)
	\). Conversely, for any such tuple
	\(
		(F, a_1, \ldots, a_{\ell + 1})
	\) the resulting homomorphism \(Q'\) is a well-defined homomorphism from the Steinberg group inducing isomorphisms on the root subgroups.
\end{proof}

Recall that the \textit{even special orthogonal group}
\(
	\sorth(2 \ell, K)
\) for
\(
	\ell \geq 2
\) is the Chevalley group
\(
	G(\mathsf D_\ell, K)
\) for a suitable choice of the weight lattice, it is
\(
	\mathsf D_\ell
\)-graded. The corresponding Steinberg group
\(
	\storth(2 \ell, K)
\) is also
\(
	\mathsf D_\ell
\)-graded, it has generators
\(
	t_{ij}(x)
\) for
\(
	x \in K
\),
\(
	i \neq \pm j
\) and the relations
\begin{align*}
	t_{ij}(x)\, t_{ij}(x') &= t_{ij}(x + x'),
&
	[t_{ij}(x), t_{kl}(y)] &= 1
	\text{ for }
	\{-i, j\} \cap \{k, -l\} = \varnothing,
\\
	t_{ij}(x) &= t_{-j, -i}(-x),
&
	[t_{ij}(x), t_{jk}(y)] &= t_{ik}(x y)
	\text{ for } i \neq \pm k,
\\&&
	[t_{-i, j}(x), t_{ji}(y)] &= 1.
\end{align*}

\begin{theorem} \label{group-d}
	Let \(G\) be a
	\(
		\mathsf D_\ell
	\)-graded group for
	\(
		\ell \geq 4
	\). Then there is a commutative unital ring \(K\) and a homomorphism
	\(
		Q \colon \storth(2 \ell, K) \to G
	\) from the even orthogonal Steinberg group inducing isomorphisms between the root subgroups.

	Any other such homomorphism
	\(
		Q' \colon \storth(2 \ell, K') \to G
	\) is of the following type. Choose a ring isomorphism
	\(
		F \colon K' \to K
	\) and elements
	\(
		a_{-\ell}, \ldots, a_{-1},
		a_1, \ldots, a_\ell
		\in K^*
	\) such that
	\(
		a_i a_{-i} = a_j a_{-j}
	\) and let
	\(
		Q'(t_{ij}(x))
		= Q(
			t_{ij}(F(x)\, a_i / a_j)
		)
	\). The tuple
	\(
		(F, a_i)_i
	\) is uniquely determined by \(Q'\) up to the change
	\(
		a_i \mapsto a_i c^{\eps_i}
	\) for
	\(
		c \in K^*
	\).
\end{theorem}
\begin{proof}
	Let
	\(
		(R_{ij}, \Delta^0_i)_{
			i, j \in \{{-}, {+}\}
		}
	\) be the
	\(
		(\mathsf B_\ell, \mathsf A_{\ell - 1})
	\)-ring constructed by \(G\),
	\(
		\Delta^0_{-} = \Delta^0_{+} = \dot 0
	\). Choose a Weyl element
	\[
			n_{-1, 2}
		=
			t_{-1, 2}(1_{-+})\,
			t_{2, -1}(-1_{+-})\,
			t_{-1, 2}(1_{-+})
	\]
	with
	\(
		1_{-+} 1_{+-} = 1_{-}
	\) and
	\(
		1_{+-} 1_{-+} = 1_{+}
	\) by lemma \ref{weyl-a-long}. Then
	\(
		\sMat{R_{--}}{R_{-+}}{R_{+-}}{R_{++}}
		\cong \mat(2, K)
	\) for the associative unital ring
	\(
		K = R_{++}
	\), where
	\(
		1_{-}
	\),
	\(
		1_{-+}
	\),
	\(
		1_{+-}
	\), and
	\(
		1_{+}
	\) are matrix units. Let
	\(
		x^* = 1_{+-} \inv a 1_{-+}
	\) for
	\(
		x \in K
	\) and
	\(
		\lambda = 1_{+-} \inv{1_{-+}} \in K^*
	\). By axioms of
	\(
		(\mathsf B_\ell, \mathsf A_{\ell - 1})
	\)-rings,
	\(
		x \mapsto x^*
	\) is a ring anti-endomorphism,
	\(
		\lambda^* = \lambda^{-1}
	\),
	\(
		x^{**} = \lambda x \lambda^*
	\),
	\(
		\lambda^{**} = \lambda
	\), and
	\(
		x y z = y^* x^* \lambda z
	\) for all
	\(
		x, y, z \in K
	\) (the last identity follows from the axiom for
	\(
		\widehat \rho(\phi(1_{-+} x y))\, z
	\) since
	\(
		\Delta^0_{+} = \dot 0
	\)). It follows that
	\(
		\lambda = 1
	\),
	\(
		x^* = x
	\), and \(K\) is commutative. Thus we have the required homomorphism
	\(
		Q \colon \storth(2 \ell, K) \to G
	\).

	Now let
	\(
		Q' \colon \storth(2 \ell, K') \to G
	\) be another such homomorphism. There are additive isomorphisms
	\(
		F_{ij} \colon K' \to K
	\) such that
	\(
		Q'(t_{ij}(x))
		= Q\bigl( t_{ij}(F_{ij}(x)) \bigr)
	\). It follows that
	\(
		F_{ij}(x)\, F_{jk}(y) = F_{ik}(x y)
	\) and
	\(
		F_{ij}(x) = F_{-j, -i}(x)
	\) for distinct
	\(
		|i|
	\),
	\(
		|j|
	\),
	\(
		|k|
	\). As in the proof of theorem \ref{group-a},
	\(
		F_{ij}(1) \in K^*
	\) and there are
	\(
		a_i \in K^*
	\) and a ring isomorphism
	\(
		F \colon K' \to K
	\) such that
	\(
		F_{ij}(x) = F(x)\, a_i / a_j
	\), e.g.
	\(
		a_i = F_{i \ell}(1)
	\) for
	\(
		|i| < \ell
	\),
	\(
		a_\ell = 1
	\),
	\(
		a_{-\ell} = F_{-\ell, 1}(1)\, a_1
	\), and
	\(
		F(x) = F_{\ell 1}(x)\, a_1
	\). The second identity implies that
	\(
		a_{-i} a_i = a_{-j} a_j
	\). Conversely, for any such tuple
	\(
		(F, a_i)_i
	\) the resulting homomorphism \(Q'\) is well-defined.
\end{proof}

Now let \(G\) be an
\(
	\mathsf E_\ell
\)-graded group for
\(
	\ell \in \{6, 7, 8\}
\). We are going to construct its coordinatisation
\(
	\stlin(\mathsf E_\ell, K) \to G
\) for some commutative unital ring \(K\). Let us fix parametrizations
\(
	t_\alpha
\) of root subgroups of the ``standard'' group
\(
	G_{\mathrm{std}}
	= G^{\mathrm{ad}}(\mathsf E_\ell, \mathbb C)
\) and let
\(
	N_{\alpha \beta i j}
\),
\(
	d_{\alpha \beta}
\) be the corresponding structure constants. We also need the torus
\[
	T^{\mathsf{ad}}(\mathsf E_\ell, \mathbb Z)
	= \bigl\{
		(\sigma_\alpha)_{\alpha \in \mathsf E_\ell}
		\mid
		\sigma_\alpha \in \{-1, 1\},\,
		\sigma_\alpha \sigma_\beta = \sigma_\gamma
		\text{ for } \alpha + \beta = \gamma
	\bigr\}
	\cong \{-1, 1\}^\ell.
\]
This group acts on root subgroups of both \(G\) and
\(
	G_{\mathrm{std}}
\) by
\(
	\up{(\sigma_\beta)_\beta}{t_\alpha(x)}
	= t_\alpha(\sigma_\alpha x)
\) preserving the commutator relations. If
\(
	g = (\sigma_\gamma)_\gamma
	\in T^{\mathrm{ad}}(\mathsf E_\ell, \mathbb Z)
\) is arbitrary and
\(
	n_\alpha \in G
\) is \(\alpha\)-Weyl, then clearly
\[
	\up{n_\alpha g n_\alpha^{-1}}{t_\beta(x)}
	= t_\beta(\sigma_{s_\alpha(\beta)} x)
\]
in \(G\) or
\(
	G_{\mathrm{std}}
\).

\begin{lemma} \label{e-braid}
	For any root
	\(
		\alpha \in \mathsf E_\ell
	\) there is
	\(
		g_\alpha
		\in T^{\mathrm{ad}}(\mathsf E_\ell, \mathbb Z)
	\) such that in every
	\(
		\mathsf E_\ell
	\)-graded group \(G\) every \(\alpha\)-Weyl element
	\(
		n_\alpha
	\) satisfies
	\(
		[n_\alpha^2 g_\alpha^{-1}, G_\beta] = 1
	\). For any
	\(
		\alpha \perp \beta
	\) all \(\alpha\)-Weyl elements commute with all \(\beta\)-Weyl elements. For all roots
	\(
		\alpha, \beta \in \mathsf E_\ell
	\) such that
	\(
		\alpha + \beta \in \mathsf E_\ell
	\) there is
	\(
		g_{\alpha \beta}
		\in T^{\mathrm{ad}}(\mathsf E_\ell, \mathbb Z)
	\) such that in any
	\(
		\mathsf E_\ell
	\)-graded group \(G\) all \(\alpha\)-Weyl and \(\beta\)-Weyl elements
	\(
		n_\alpha
	\),
	\(
		n_\beta
	\) satisfy
	\(
		[
			(n_\alpha n_\beta)^3\,
			g_{\alpha \beta}^{-1},
			G_\gamma
		]
		= 1
	\).
\end{lemma}
\begin{proof}
	Such result is well-known for all simply-laced Chevalley groups. For arbitrary \(G\) we may apply theorem \ref{group-d} to a root subsystem of type
	\(
		\mathsf D_m
	\) containing all roots from an identity. Indeed, if
	\(
		\Psi \subseteq \mathsf E_\ell
	\) is a root subsystem of rank \(3\) (not of the type
	\(
		3 \mathsf A_1
	\) since the second claim is obvious), then any its base may be continued to a base of
	\(
		\mathsf E_\ell
	\), so without loss of generality it corresponds to a triple of roots in the Dynkin diagram of
	\(
		\mathsf E_\ell
	\) and at least two of them are neighbors. If these roots do not lie in a common sub-diagram of type
	\(
		\mathsf D_m
	\), then we may apply a suitable Weyl element permuting one of the chosen roots with its neighbor.
\end{proof}

Recall that the set of orthogonal roots
\(
	\mu^\perp \cap \mathsf E_\ell
\) to a root
\(
	\mu \in \mathsf E_\ell
\) is a root system itself of type
\(
	\mathsf A_5
\) for
\(
	\ell = 6
\),
\(
	\mathsf D_6
\) for
\(
	\ell = 7
\), and
\(
	\mathsf E_7
\) for
\(
	\ell = 8
\). If \(\mu\) is the highest root with respect to a base
\(
	\Delta \subseteq \mathsf E_\ell
\), then
\(
	\mu^\perp \cap \Delta
\) is a base of
\(
	\mu^\perp \cap \mathsf E_\ell
\). The stabilizer of \(\mu\) in the Weyl group
\(
	\mathrm W(\mathsf E_\ell)
\) is precisely the Weyl group
\(
	\mathrm W(\mu^\perp \cap \mathsf E_\ell)
\).

\begin{theorem} \label{group-e}
	Let \(G\) be an
	\(
		\mathsf E_\ell
	\)-graded group for
	\(
		\ell \in \{6, 7, 8\}
	\). Then there is a commutative unital ring \(K\) and a homomorphism
	\(
		Q \colon \stlin(\mathsf E_\ell, K) \to G
	\) inducing isomorphisms between the root subgroups.

	Any other such homomorphism
	\(
		Q' \colon \stlin(\mathsf E_\ell, K') \to G
	\) is of the following type. Choose a ring isomorphism
	\(
		F \colon K' \to K
	\) and element
	\(
		(a_\alpha)_{\alpha \in \Phi}
		\in T^{\mathrm{ad}}(\mathsf E_\ell, K)
	\) and let
	\(
		Q'(t_\alpha(x))
		= Q\bigl(
			t_\alpha(a_\alpha\, F(x))
		\bigr)
	\). Both \(F\) and
	\(
		(a_\alpha)_\alpha
	\) are uniquely determined by \(Q'\).
\end{theorem}
\begin{proof}
	Fix a base
	\(
		\{\alpha_1, \ldots, \alpha_\ell\}
		\subseteq \mathsf E_\ell
	\) and choose
	\(
		\alpha_i
	\)-Weyl elements
	\(
		n_i
	\) for
	\(
		1 \leq i \leq \ell
	\). Let
	\(
		K = G_\mu
	\), where \(\mu\) is the highest root, for now it is just an abelian group (commutativity follows from theorem \ref{group-d} applied to a root subsystem of type
	\(
		\mathsf D_5
	\)). We are going to define isomorphisms
	\(
		t_\alpha \colon K \to G_\alpha
	\) for all
	\(
		\alpha \in \mathsf E_\ell
	\) in such a way that
	\(
		t_\mu(x) = x
	\) and
	\[
		t_{s_i(\beta)}(d_{\alpha_i \beta} x)
		= \up{n_i}{t_\beta(x)},
	\]
	where
	\(
		s_i = s_{\alpha_i}
	\).

	This definition allows us to evaluate all
	\(
		t_\alpha
	\) by choosing tuples
	\(
		(k_1, \ldots, k_m)
	\) such that
	\(
		\alpha
		= s_{k_1}(\ldots s_{k_m}(\mu) \ldots)
	\). We only have to check that
	\(
		t_\alpha
	\) is independent on such choice. Indeed, if
	\(
		s_{k_1} \cdots s_{k_m}
		= s_{k'_1} \cdots s_{k'_{m'}}
	\), then the two formulae for
	\(
		t_\alpha(x)
	\) coincide by lemma \ref{e-braid} and explicit description of relations between generators of reflection groups since the required identities between various
	\(
		d_{\beta \gamma}
	\),
	\(
		g_\beta
	\), and
	\(
		g_{\beta \gamma}
	\) hold in
	\(
		G_{\mathrm{std}}
	\). Moreover, if
	\(
		s_{k_1} \cdots s_{k_n}
	\) stabilizes \(\mu\), then without loss of generality
	\(
		\alpha_{k_i} \perp \mu
	\), so
	\(
		t_{\alpha}(x) = x
	\).

	Now let us introduce a multiplication on \(K\) in such a way that
	\(
		[t_\alpha(x), t_\beta(y)]
		= t_{\alpha + \beta}(N_{\alpha \beta 1 1} x y)
	\) for
	\(
		\alpha,
		\beta,
		\alpha + \beta
		\in \mathsf E_\ell
	\). Recall that in this case
	\(
		N_{\alpha \beta 1 1} \in \{-1, 1\}
	\). Such multiplication is independent on the choice of
	\(
		(\alpha, \beta)
	\) since the Weyl group acts transitively on the set of such pairs and the corresponding identities between various
	\(
		d_{\gamma \delta}
	\) and
	\(
		N_{\gamma \delta 1 1}
	\) hold in
	\(
		G_{\mathrm{std}}
	\). Here we use that the Dynkin diagram is not a chain, for
	\(
		\mathsf A_\ell
	\) there are two orbits of such pairs
	\(
		(\alpha, \beta)
	\) and the resulting ring may be non-commutative. The constructed multiplication makes \(K\) a commutative unital ring by theorem \ref{group-d} applied to a root subsystem of type
	\(
		\mathsf D_5
	\).

	If
	\(
		Q' \colon \stlin(\mathsf E_\ell, K') \to G
	\) is another such homomorphism given by isomorphisms
	\(
		t'_\alpha \colon K \to G
	\), then there are additive isomorphisms
	\(
		F_\alpha \colon K' \to K
	\) such that
	\(
		t'_\alpha(x) = t_\alpha(F_\alpha(x))
	\). Clearly,
	\(
		F_\alpha(x)\, F_\beta(y)
		= F_{\alpha + \beta}(x y)
	\) for
	\(
		\alpha,
		\beta,
		\alpha + \beta
		\in \mathsf E_\ell
	\). Let
	\(
		a_\alpha = F_\alpha(1)
	\), they form an element of
	\(
		T^{\mathrm{ad}}(\mathsf E_\ell, K)
	\). It is easy to see that
	\(
		F(x) = F_\alpha(x) / a_\alpha
	\) is independent of \(\alpha\).
\end{proof}

\section{Alternative rings}

In this section we collect basic facts about alternative rings needed below. Let \(R\) be a non-associative ring. The \textit{associator} of
\(
	x, y, z \in R
\) is the element
\(
	[x, y, z] = (x y)\, z - x\, (y z)
\). The \textit{nucleus} of \(R\) is the set
\[
	\mathrm N(R)
	= \{
		\nu \in R
		\mid
		[\nu, R, R] = [R, \nu, R] = [R, R, \nu] = 0
	\}.
\]
If \(R\) is unital, then
\(
	1 \in \mathrm N(R)
\). Every element
\(
	\nu \in \mathrm N(R)
\) satisfies the identities \cite[lemma 7.1.1]{zhevlakov}
\begin{align*}
	\nu\, [x, y, z] &= [\nu x, y, z],
&
	[x \nu, y, z] &= [x, \nu y, z],
\\
	[x, y \nu, z] &= [x, y, \nu z],
&
	[x, y, z \nu] &= [x, y, z]\, \nu.
\end{align*}
Moreover,
\(
	\mathrm N(R) \subseteq R
\) is an associative subring \cite[corollary 7.1.1]{zhevlakov}. The \textit{center} of \(R\) is
\[
	\mathrm C(R)
	= \{
		x \in \mathrm N(R)
		\mid
		x y = y x \text{ for all } y \in R
	\}.
\]

A ring \(R\) is called \textit{alternative} if the identities
\(
	[x, x, y] = [x, y, x] = [y, x, x] = 0
\) hold, i.e.
\(
	x^2 y = x\, (x y)
\),
\(
	(x y)\, x = x\, (y x)
\), and
\(
	(y x)\, x = y x^2
\). For such rings the associator is skew-symmetric,
\(
	[x_{\sigma(1)}, x_{\sigma(2)}, x_{\sigma(3)}]
	= (-1)^{\sign(\sigma)}\, [x_1, x_2, x_3]
\) for every permutation \(\sigma\). Moreover, the following identities hold (see \cite[lemma 2.3.7 and its corollary]{zhevlakov} and \cite{mccrimmon}).
\begin{align*}
	x\, (y\, (x z)) &= (x y x)\, z,
&
	[x, y, z x] &= x\, [y, z, x],
&
	x\, (y\, (z x)\, y) &= (x y)\, z\, (x y),
\\
	((y x)\, z)\, x &= y\, (x z x),
&
	[x y, z, x] &= [x, y, z]\, x,
&
	(x\, (y z)\, x)\, y &= (x y)\, z\, (x y),
\\
	(x y)\, (z x) &= x\, (y z)\, x,
&
	[y, x^2, z] &= x\, [y, x, z] + [y, x, z]\, x,
&
	x\, (y\, (x z x)\, y)\, x &= (x y x)\, z\, (x y x).
\end{align*}

Let \(R\) be a non-associative unital ring. An element
\(
	a \in R
\) is \textit{invertible} if there is
\(
	b \in R
\) such that
\(
	a b = b a = 1
\) (an \textit{inverse} of \(a\)). In an alternative ring \(a\) is invertible if and only if it is left- and right-invertible if and only if
\(
	x \mapsto a x
\) is bijective if and only if
\(
	x \mapsto x a
\) is bijective \cite[proposition 2]{mccrimmon}. If \(a\) and \(b\) are invertible in an alternative ring, then
\(
	a b
\) is invertible with
\(
	(a b)^{-1} = b^{-1} a^{-1}
\) \cite[proposition 2]{mccrimmon}. If
\(
	a \in R
\) is invertible in an alternative ring, then
\(
	[a, a^{-1}, R] = 0
\) \cite[lemma 10.3.7]{zhevlakov}.

Every \(2\)-generated unital subring of an alternative unital ring \(R\) is associative by Artin's theorem \cite[theorem 2.3.2]{zhevlakov}. A \(3\)-generated unital subring
\(
	\langle x, y, z \rangle
	\subseteq R
\) is associative if and only if
\(
	[x, y, z] = 0
\) \cite[theorem I.2]{bruck-kleinfeld}. If
\(
	S \subseteq R
\) is an associative unital subring and
\(
	a \in S
\) is invertible in \(R\), then the subring
\(
	\langle S, a^{-1} \rangle \subseteq R
\) is also associative by \cite[theorem I.3]{bruck-kleinfeld} and \cite[lemma 10.3.8]{zhevlakov}.

The next lemma is used in the existence theorem.

\begin{lemma} \label{inv-alter}
	Let \(R\) be an alternative unital ring and
	\(
		x, y \in R
	\) be such that
	\(
		1 + x y
	\) is invertible. Then
	\(
		1 + y x
	\) is invertible with
	\(
		(1 + y x)^{-1} = 1 - y\, (1 + x y)^{-1}\, x
	\), and
	\[
		x\, (\widehat y z) = \widehat x\, (y z),
	\quad
		y\, (\widehat x z) = \widehat y\, (x z),
	\quad
		(z \widehat x)\, y = (z x)\, \widehat y,
	\quad
		(z \widehat y)\, x = (z y)\, \widehat x
	\]
	for any
	\(
		z \in R
	\), where
	\[
		\widehat y
		= y\, (1 + x y)^{-1}
		= (1 + y x)^{-1}\, y,
	\quad
		\widehat x
		= x\, (1 + y x)^{-1}
		= (1 + x y)^{-1}\, x.
	\]
\end{lemma}
\begin{proof}
	The formula for the inverse of
	\(
		1 + y x
	\) is well-known and it may be easily checked in the associative subring generated by \(x\), \(y\), and
	\(
		(1 + x y)^{-1}
	\). We check only the identity
	\(
		x\, (\widehat y z) = \widehat x\, (y z)
	\) since the remaining ones are similar. Multiply both sides from the left and from the right by
	\(
		1 + x y
	\) and replace \(z\) by
	\(
		(1 + x y)\, z\, (1 + x y)
	\), so the identity becomes
	\[
		(1 + x y x)\, \bigl((y z)\, (1 + x y)^2\bigr)
		= x\, \bigl(
			((y + y x y)\, z)\, (1 + x y)^2
		\bigr).
	\]
	Now expand both sides. The resulting identity easily follows from the standard identities for monomials in alternative rings.
\end{proof}

\section{\(\mathsf B_\ell\)-graded groups}

Let \(G\) be a
\(
	\mathsf B_\ell
\)-graded group and
\(
	(R_{ij}, \Delta^0_i)_{
		i, j \in \{{-}, {+}\}
	}
\) be the corresponding
\(
	(\mathsf B_\ell, \mathsf A_{\ell - 1})
\)-ring. By lemma \ref{weyl-a-short} there is a Weyl element
\[
		n_1
	=
		t_1\bigl(
			\iota
			\cdot \widehat \rho(\iota)^{-1}\,
			\inv{\widehat \rho(\iota)}
		\bigr)\,
		t_{-1}\bigl(
			\iota
			\cdot (-\widehat \rho(\iota)^{-1})
		\bigr)\,
		t_1(\iota).
\]
Let
\(
	1_{++} = 1_{+}
\),
\(
	1_{--} = 1_{-}
\),
\(
	1_{-+} = \widehat \rho(\iota) \in R_{-+}
\), and
\(
	1_{+-} = \widehat \rho(\iota)^{-1} \in R_{+-}
\), i.e.
\(
	1_{ij} 1_{jk} = 1_{ik}
\). Also, let
\(
	\lambda = 1_{+-} \inv{1_{-+}} \in R_{++}
\) and
\(
	x^* = 1_{+-}\, (\inv x 1_{-+})
\) for any
\(
	x \in R_{++}
\). We use the notation
\(
	\eps_i = (-1)^{\sign(i)}
\) for non-zero integer \(i\), i.e.
\(
	\eps_i = 1
\) for
\(
	i > 0
\) and
\(
	\eps_i = -1
\) for
\(
	i < 0
\).

\begin{lemma} \label{bb-ring-1}
	The elements
	\(
		1_{+-}
	\),
	\(
		1_{-+}
	\),
	\(
		\lambda
	\),
	\(
		\lambda^*
	\) lie in the nucleus of \(
		R_{**}
		= \sMat{R_{--}}{R_{-+}}{R_{+-}}{R_{++}}
	\), so
	\(
		R_{**} \cong \mat(2, R_{++})
	\). Moreover,
	\(
		(x y)^* = y^* x^*
	\),
	\(
		1_{+}^* = 1_{+}
	\),
	\(
		\lambda^* \lambda = \lambda \lambda^* = 1_{+}
	\),
	\(
		x^{**} = \lambda x \lambda^*
	\), and
	\(
		\lambda^{**} = \lambda
	\). The involution on
	\(
		R_{**}
	\) may be expressed via
	\(
		({-})^*
	\) as
	\[
			\inv{1_{p+} x 1_{+q}}
		=
			1_{-q, {+}}
			\lambda^{(\eps_q - 1) / 2}
			x^*
			\lambda^{(1 - \eps_p) / 2}
			1_{{+}, -p}
	\]
	for
	\(
		x \in R_{++}
	\).
\end{lemma}
\begin{proof}
	Note that
	\[
			\up{n_1}{n_{12}}
		=
			t_{-1, 2}(1_{-+})\,
			t_{2, -1}(-1_{+-})\,
			t_{-1, 2}(1_{-+})
	\]
	is also a Weyl element by lemma \ref{weyl-a-short}. From lemma \ref{weyl-a-long} we obtain
	\(
			[\zeta, \eta, R_{**}]
		=
			[\zeta, R_{**}, \eta]
		=
			[R_{**}, \zeta, \eta]
		= 0
	\),
	\(
		(\zeta x)\, (y \zeta) = \zeta\, (x y)\, \zeta
	\),
	\(
		\zeta\, (x\, (\zeta y)) = (\zeta x \zeta)\, y
	\),
	\(
		((x \zeta)\, y)\, \zeta = x\, (\zeta y \zeta)
	\) for
	\(
		\zeta, \eta \in \{1_{+-}, 1_{-+}\}
	\).

	On the other hand, by axioms of
	\(
		(\mathsf B_\ell, \mathsf A_{\ell - 1})
	\)-rings the elements
	\(
		1_{-+} = \widehat \rho(\iota)
	\) and
	\(
		\inv{1_{+-}}
		= \widehat \rho(\iota \cdot 1_{+-})
	\) lie in the left nucleus of
	\(
		R_{**}
	\), i.e.
	\(
			[1_{-+}, R_{**}, R_{**}]
		=
			[\inv{1_{+-}}, R_{**}, R_{**}]
		= 0
	\). By duality,
	\(
			[R_{**}, R_{**}, \inv{1_{-+}}]
		=
			[R_{**}, R_{**}, 1_{+-}]
		= 0
	\). It easily follows that
	\(
		1_{+-}
	\) and
	\(
		1_{-+}
	\) lie in the nucleus of
	\(
		R_{**}
	\), as well as
	\(
		\inv{1_{+-}}
	\),
	\(
		\inv{1_{-+}}
	\),
	\(
		\lambda = 1_{+-} \inv{1_{-+}}
	\), and
	\(
		\lambda^* = \inv{1_{+-}} 1_{-+}
	\). The formula for the involution may be checked by cases. The remaining identities easily follow from
	\(
		\inv{(x y)} = \inv y \inv x
	\) and
	\(
		\inv{\inv x} = x
	\) for
	\(
		x, y \in R_{++}
	\).
\end{proof}

Now let
\(
	\phi(x) = \phi(1_{-+} x)
\) for
\(
	x \in R_{++}
\),
\(
	\langle u, v \rangle
	= 1_{+-}\, (u \circ v)
\) for
\(
	u, v \in \Delta^0_+
\), and
\(
	\rho(u) = 1_{+-}\, \widehat \rho(u)
\) for
\(
	u \in \Delta^0_+
\). Then
\begin{align*}
		\phi(1_{-p, +} x 1_{+p})
	&=
		\phi\bigl(
			\lambda^{(1 - \eps_p) / 2} x
		\bigr) \cdot 1_{+p},
\\
		(u \cdot 1_{+p}) \circ (v \cdot 1_{+q})
	&=
		1_{-p, +}
		\lambda^{(\eps_p - 1) / 2}\,
		\langle u, v \rangle\,
		1_{+q},
\\
		\widehat \rho(u \cdot 1_{+p})
	&=
		1_{-p, +}
		\lambda^{(\eps_p - 1) / 2}\,
		\rho(u)\,
		1_{+p}
\end{align*}
for
\(
	x \in R_{++}
\) and
\(
	u, v \in \Delta^0_+
\).

\begin{lemma} \label{bb-ring-2}
The unital ring \(R_{++}\) is alternative. The associator map satisfies the identities
\[
	[x^*, x, y] = 0,
\quad
		\zeta\, [x, y, z]
	=
		[\zeta x, y, z]
	=
		[x \zeta, y, z]
	=
		[x, y, z]\, \zeta
	=
		-[x, y, z],
\quad
	[x^*, y, z] = [x, y, z]^* = -[x, y, z]
\]
for
\(
	\zeta \in \{\lambda, \lambda^*\}
\).
\end{lemma}
\begin{proof}
	Note that
	\(
		1_{+-}\, \widehat \rho(\phi(1_{-+} x))
		= \rho(\phi(x))
		= x - x^* \lambda
	\) lie in the left nucleus of
	\(
		R_{++}
	\), i.e.
	\(
		[x, y, z] = [x^* \lambda, y, z]
	\). On the other hand,
	\(
			(x^* x)\, y
		=
			1_{+-}\,
			(\inv x\, \widehat \rho(\iota)\, x)\,
			y
		=
			1_{+-}\,
			(\inv x\, \widehat \rho(\iota))\,
			(x y)
		=
			x^*\, (x y)
	\). It follows that
	\(
		[x^*, x, y] = 0
	\) and
	\[
		[x, x, y]
		= \lambda^*\, [\lambda x, x, y]
		= \lambda^*\, [x^*, x, y]
		= 0.
	\]
	Since
	\(
		[x, y, z]^* = -[z^*, y^*, x^*]
	\) in any ring with an anti-automorphism, the ring
	\(
		R_{++}
	\) is alternative.

	It remains to check that
	\(
		[\lambda x, y, z] = -[x, y, z]
	\). But this follows from
	\(
		1 + \lambda = -\langle \iota, \iota \rangle
	\) and
	\[
			(
				(\langle \iota, \iota \rangle\, x)\,
				y
			)\,
			z
		=
			\bigl\langle
				\iota,
				((\iota \cdot x) \cdot y) \cdot z
			\bigr\rangle
		=
			\bigl\langle
				\iota,
				(\iota \cdot x) \cdot y z
			\bigr\rangle
		=
			(\langle \iota, \iota \rangle\, x)\,
			(y z).\qedhere
	\]
\end{proof}

By lemma \ref{bb-ring-1} we may identify all
\(
	R_{ij}
\) using multiplications by
\(
	1_{pq}
\) from both sides, as well as identify
\(
	\Delta^0_{-}
\) with
\(
	\Delta^0_{+}
\) using the maps
\(
	u \mapsto u \cdot 1_{+-}
\) and
\(
	u \mapsto u \cdot 1_{-+}
\). We call the resulting objects
\(
	(\mathsf B_\ell, \mathsf B_\ell)
\)-\textit{rings}. Namely, a
\(
	(\mathsf B_\ell, \mathsf B_\ell)
\)-ring consists of
\begin{itemize}

	\item
	an alternative unital ring
	\(
		R_{++}
	\) with an anti-automorphism
	\(
		({-})^* \colon R_{++} \to R_{++}
	\) and an element
	\(
		\lambda \in R_{++}^*
	\) satisfying
	\(
		\lambda^{-1} = \lambda^*
	\),
	\(
		\lambda, \lambda^* \in \mathrm N(R_{++})
	\),
	\(
		x^{**} = \lambda x \lambda^*
	\) (in particular,
	\(
		\lambda^{**} = \lambda
	\));

	\item
	a group
	\(
		\Delta^0_+
	\) with the group operation \(\dotplus\) and a distinguished element
	\(
		\iota
	\);

	\item
	a group homomorphism
	\(
		\phi \colon R_{++} \to \Delta^0_+
	\) with central image and such that
	\(
		\phi(x + x^* \lambda) = \dot 0
	\);

	\item
	a biadditive map
	\(
		\langle {-}, {=} \rangle
		\colon \Delta^0_+ \times \Delta^0_+
		\to \mathrm N( R_{++} )
	\) such that
	\(
		\langle v, u \rangle
		= \langle u, v \rangle^*\, \lambda
	\);

	\item
	a map
	\(
		\rho \colon \Delta^0_+ \to \mathrm N(R_{++})
	\) such that
	\(
		\rho(u \dotplus v)
		= \rho(u) - \langle u, v \rangle + \rho(v)
	\),
	\(
		\rho(u)
		+ \langle u, u \rangle
		+ \rho(u)^*\, \lambda
		= 0
	\),
	\(
		\rho(\iota) = 1
	\) (in particular,
	\(
		\rho(-u) = \rho(u)^*\, \lambda
	\) and
	\(
		\langle \iota, \iota \rangle = -1 - \lambda
	\));

	\item
	A map
	\(
		({-}) \cdot ({=})
		\colon \Delta^0_+ \times R_{++}
		\to \Delta^0_+
	\) such that
	\(
		(u \dotplus v) \cdot x
		= u \cdot x \dotplus v \cdot x
	\) and
	\(
			u \cdot (x + y)
		=
			u \cdot x
			\dotplus \phi(y^*\, \rho(u)\, x)
			\dotplus u \cdot y
	\);

\end{itemize}
Moreover, the operations satisfy the identities from lemma \ref{bb-ring-2},
\begin{align*}
		u
		\dotplus v
	&=
		v
		\dotplus u
		\dotminus \phi(\langle u, v \rangle),
&
	u \cdot 1 &= u,
\\
	\phi((x y)\, z) &= \phi(x\, (y z)),
&
	(u \cdot x) \cdot y &= u \cdot x y,
\\
		\langle u, \phi(x) \rangle
	&=
		\langle \phi(x), u \rangle
	= 0,
&
	u \cdot (x y)\, z &= u \cdot x\, (y z),
\\
	\rho(\phi(x)) &= x - x^* \lambda,
&
	\langle u, v \cdot x \rangle
	&= \langle u, v \rangle\, x,
\\
	x\, (y^*\, \rho(u)\, y) &= (x y^*)\, \rho(u)\, y,
&
	\langle u \cdot x, v \rangle
	&= x^*\, \langle u, v \rangle,
\\
	(x^*\, \rho(u)\, x)\, y &= x^*\, \rho(u)\, (x y),
&
	\phi(x) \cdot y &= \phi(y^* x y),
\\&&
	\rho(u \cdot x) &= x^*\, \rho(u)\, x.
\end{align*}
In the case
\(
	\ell \geq 4
\) we again require that
\(
	R_{++}
\) is associative.

If
\(
	(R, \Delta) = (R_{++}, \Delta^0_{+})
\) is a
\(
	(\mathsf B_\ell, \mathsf B_\ell)
\)-ring, then its \textit{Steinberg group}
\(
	\stlin(\mathsf B_\ell, R, \Delta)
\) is the abstract group with generators
\(
	t_{ij}(x)
\) for
\(
	i \neq \pm j
\),
\(
	x \in R
\) and
\(
	t_i(u)
\) for
\(
	u \in \Delta
\) satisfying the relations
\begin{align*}
	t_{ij}(x)\, t_{ij}(y) &= t_{ij}(x + y),
&
	t_{ij}(x)
	&= t_{-j, -i}\bigl(
		-\lambda^{(\eps_j - 1) / 2}
		x^*
		\lambda^{(1 - \eps_i) / 2}
	\bigr),
\\
	t_i(u)\, t_i(v) &= t_i(u \dotplus v),
&
	[t_{ij}(x), t_{jk}(y)] &= t_{ik}(x y)
	\text{ for } i \neq k,
\\
	[t_i(u), t_{kl}(x)] &= 1
	\text{ for } i \notin \{k, -l\},
&
	[t_{-i, j}(x), t_{ji}(y)]
	&= t_i\bigl(
		\phi\bigl(
			\lambda^{(1 - \eps_i) / 2} x y
		\bigr)
	\bigr),
\\
	[t_{ij}(x), t_{kl}(y)] &= 1
	\text{ for }
	\{-i, j\} \cap \{k, -l\} = \varnothing,
&
	[t_i(u), t_j(v)]
	&= t_{-i, j}\bigl(
		-\lambda^{(\eps_i - 1) / 2}\,
		\langle u, v \rangle
	\bigr)
	\text{ for } i \neq \pm j,
\\&&
	[t_i(u), t_{ij}(a)]
	&= t_{-i, j}\bigl(
		\lambda^{(\eps_i - 1) / 2}\,
		\rho(u)\,
		x
	\bigr)\,
	t_j(\dotminus u \cdot (-x)).
\end{align*}
This group is
\(
	\mathsf B_\ell
\)-graded if
\(
	(R, \Delta)
\) is constructed by a
\(
	\mathsf B_\ell
\)-graded group. Later we prove that it is always
\(
	\mathsf B_\ell
\)-graded.

\begin{theorem} \label{group-b}
	Let \(G\) be a
	\(
		\mathsf B_\ell
	\)-graded group for
	\(
		\ell \geq 3
	\). Then there is a
	\(
		(\mathsf B_\ell, \mathsf B_\ell)
	\)-ring
	\(
		(R, \Delta)
	\) and a homomorphism
	\(
		Q
		\colon \stlin(\mathsf B_\ell, R, \Delta)
		\to G
	\) from the Steinberg group inducing isomorphisms between the root subgroups.

	In the case
	\(
		\ell \geq 4
	\) any other such homomorphism
	\(
		Q'
		\colon \stlin(\mathsf B_\ell, R', \Delta')
		\to G
	\) is of the following type. Choose a ring isomorphism
	\(
		F \colon R' \to R
	\), a group isomorphism
	\(
		H \colon \Delta' \to \Delta
	\), and elements
	\(
		a_{-\ell}, \ldots, a_{-1},
		a_1, \ldots, a_\ell,
		b
		\in R^*
	\) such that
	\begin{align*}
		a_i^* a_{-i} &= b \text{ for } i > 0,
	&
		H(\phi(x)) &= \phi(b\, F(x)),
	&
		H(u \cdot x) &= H(u) \cdot F(x),
	\\
		F(\lambda) &= b^{-1} b^* \lambda,
	&
		F(\langle u, v \rangle)
		&= b^{-1}\,
		\bigl\langle H(u), H(v) \bigr\rangle,
	&
	\\
		F(x^*) &= b^{-1}\, F(x)^*\, b,
	&
		F(\rho(u)) &= b^{-1}\, \rho(H(u)),
	&
	\end{align*}
	and let
	\[
		Q'(t_{ij}(x))
		= Q\bigl(
			t_{ij}\bigl(a_i\, F(x)\, a_j^{-1}\bigr)
		\bigr),
	\quad
		Q'(t_i(u))
		= Q\bigl(
			t_i\bigl(H(u) \cdot a_i^{-1}\bigr)
		\bigr).
	\]
	The tuple
	\(
		(F, H, a_i, b)_i
	\) is uniquely determined by \(Q'\) up to the change
	\(
		F \mapsto c^{-1} F c
	\),
	\(
		H \mapsto H \cdot c
	\),
	\(
		a_i \mapsto a_i c
	\),
	\(
		b \mapsto c^* b c
	\) for
	\(
		c \in R^*
	\).

	In the case
	\(
		\ell = 3
	\)
	for any other such homomorphism
	\(
		Q' \colon (R', \Delta') \to G
	\) there are unique group isomorphism
	\(
		F \colon R' \to R
	\), group isomorphism
	\(
		H \colon \Delta' \to \Delta
	\), and elements
	\(
		a, b \in R^*
	\) such that
	\begin{align*}
		F(1) &= (a b)^{-1},
	&
		Q'(t_{\sigma 1, \tau 2}(x))
		&= Q\bigl(
			t_{\sigma 1, \tau 2}\bigl(
				\nu_1^{(1 - \sigma) / 2}\,
				F(x)\, a
				\nu_2^{(\tau - 1) / 2}
			\bigr)
		\bigr),
	\\
		F(x y) &= (F(x)\, a)\, (b\, F(y)),
	&
		Q'(t_{\sigma 1, \tau 3}(x))
		&= Q\bigl(
			t_{\sigma 1, \tau 3}\bigl(
				\nu_1^{(1 - \sigma) / 2}\,
				F(x)\,
				\nu_3^{(\tau - 1) / 2}
			\bigr)
		\bigr),
	\\
		F(\lambda)
		&= (\nu_3 a b)^{-1}\, \nu_3^* \lambda,
	&
		Q'(t_{\sigma 2, \tau 1}(x))
		&= Q\bigl(
			t_{\sigma 2, \tau 1}\bigl(
				\nu_2^{(1 - \sigma) / 2}
				b\, (F(x)\, a)\, b
				\nu_1^{(\tau - 1) / 2}
			\bigr)
		\bigr),
	\\
		F(\lambda^*)
		&= (\nu_3^* \lambda a b)^{-1}\, \nu_3,
	&
		Q'(t_{\sigma 2, \tau 3}(x))
		&= Q\bigl(
			t_{\sigma 2, \tau 3}\bigl(
				\nu_2^{(1 - \sigma) / 2}
				b\, F(x)\,
				\nu_3^{(\tau - 1) / 2}
			\bigr)
		\bigr),
	\\
		F(x^*)
		&= (\nu_3 a b)^{-1}\,
		F(x)^*\,
		\nu_1\,
		(a b)^{-1},
	&
		Q'(t_{\sigma 3, \tau 1}(x))
		&= Q\bigl(
			t_{\sigma 3, \tau 1}\bigl(
				\nu_3^{(1 - \sigma) / 2}\,
				(a b)\, F(x)\, (a b)\,
				\nu_1^{(\tau - 1) / 2}
			\bigr)
		\bigr),
	\\
		H(\phi(x)) &= \phi(\nu_3 a b\, F(x)),
	&
		Q'(t_{\sigma 3, \tau 2}(x))
		&= Q\bigl(
			t_{\sigma 3, \tau 2}\bigl(
				\nu_3^{(1 - \sigma) / 2}
				a\, (b\, F(x))\, a
				\nu_2^{(\tau - 1) / 2}
			\bigr)
		\bigr),
	\\
		F(\rho(u)) &= (\nu_3 a b)^{-1}\, \rho(H(u)),
	&
		Q'(t_{\sigma 1}(u))
		&= Q\bigl(
			t_{\sigma 1}\bigl(
				H(u) \cdot a b \nu_1^{(\sigma - 1) / 2}
			\bigr)
		\bigr),
	\\
		F(\langle u, v \rangle)
		&= (\nu_3 a b)^{-1}\,
		\bigl\langle H(u), H(v) \bigr\rangle,
	&
		Q'(t_{\sigma 2}(u))
		&= Q\bigl(
			t_{\sigma 2}\bigl(
				H(u) \cdot a \nu_2^{(\sigma - 1) / 2}
			\bigr)
		\bigr),
	\\
		H(u \cdot x) &= H(u) \cdot a b\, F(x),
	&
		Q'(t_{\sigma 3}(u))
		&= Q\bigl(
			t_{\sigma 3}\bigl(
				H(u) \cdot \nu_3^{(\sigma - 1) / 2}
			\bigr)
		\bigr)
	\end{align*}
	for
	\(
		\sigma, \tau \in \{-1, 1\}
	\), where
	\[
		\nu_1 = \rho(H(\iota) \cdot a b),
		\nu_2 = \rho(H(\iota) \cdot a),
		\nu_3 = \rho(H(\iota))
		\in \mathrm N(R).
	\]
	Conversely, for any tuple
	\(
		(F, H, a, b)
	\) satisfying the identities from the first column the corresponding pair
	\(
		(R', \Delta')
	\) is a
	\(
		(\mathsf B_\ell, \mathsf B_\ell)
	\)-ring and the group homomorphism \(Q'\) defined by the second column is well-defined.
\end{theorem}
\begin{proof}
	The claim about existence of \(Q\) follows from our constructions and the definition of
	\(
		(\mathsf B_\ell, \mathsf B_\ell)
	\)-rings, see theorem \ref{phi-0-ring} and lemmas \ref{ba-ring-3}, \ref{bb-ring-1}, \ref{bb-ring-2}. Now suppose that
	\(
		Q
		\colon \stlin(\mathsf B_\ell, R, \Delta)
		\to G
	\),
	\(
		Q'
		\colon \stlin(\mathsf B_\ell, R', \Delta')
		\to G
	\) are group homomorphisms inducing isomorphisms between the root subgroups. There are additive isomorphisms
	\(
		F_{ij} \colon R' \to R
	\) for
	\(
		|i| \neq |j|
	\) and
	\(
		H_i \colon \Delta' \to \Delta
	\) such that
	\(
		Q'(t_{ij}(x))
		= Q\bigl( t_{ij}(F_{ij}(x)) \bigr)
	\) and
	\(
		Q'(t_i(u))
		= Q'\bigl( t_i(H_i(u)) \bigr)
	\). It follows that
	\begin{align*}
		F_{ij}(x)\, F_{jk}(y) &= F_{ik}(x y),
	&
			\lambda^{(\eps_i - 1) / 2}\,
			\bigl\langle H_i(u), H_j(v) \bigr\rangle
		&=
			F_{-i, j}\bigl(
				\lambda^{(\eps_i - 1) / 2}\,
				\langle u, v \rangle
			\bigr),
	\\
			\lambda^{(\eps_j - 1) / 2}\,
			F_{ij}(x)^*\,
			\lambda^{(1 - \eps_i) / 2}
		&=
			F_{-j, -i}\bigl(
				\lambda^{(\eps_j - 1) / 2}
				x^*
				\lambda^{(1 - \eps_i) / 2}
			\bigr),\!\!\!\!\!\!\!\!\!\!\!\!\!\!
	&
		H_i(u) \cdot F_{ij}(x) &= H_j(u \cdot x),
	\\
			\phi\bigl(
				\lambda^{(1 - \eps_i) / 2}\,
				F_{-i, j}(x)\,
				F_{ji}(y)
			\bigr)
		&=
			H_i\bigl(
				\phi\bigl(
					\lambda^{(1 - \eps_i) / 2}
					x
					y
				\bigr)
			\bigr),
	&
			\lambda^{(\eps_i - 1) / 2}\,
			\rho(H_i(u))\,
			F_{ij}(x)
		&=
			F_{-i, j}\bigl(
				\lambda^{(\eps_i - 1) / 2}\,
				\rho(u)\,
				x
			\bigr)
	\end{align*}
	if \(i\), \(j\), \(k\) have distinct absolute values.

	In the case
	\(
		\ell \geq 4
	\) the rings \(R\) and \(R'\) are associative. It is easy to see using the first identity that
	\(
		F_{ij}(1) \in R^*
	\) and there are
	\(
		a_i \in R^*
	\) and a ring isomorphism
	\(
		F \colon R' \to R
	\) such that
	\(
		F_{ij}(x) = a_i\, F(x)\, a_j^{-1}
	\) (e.g.
	\(
		a_i = F_{i \ell}(1)
	\) for
	\(
		|i| < \ell
	\),
	\(
		a_\ell = 1
	\),
	\(
		a_{-\ell} = F_{-\ell, 1}(1)\, a_1
	\), and
	\(
		F(x) = F_{\ell 1}(x)\, a_1
	\)), see the proof of theorem \ref{group-a}. The identity for involution means that
		\[
			F(x)^*\,
			a_i^*
			\lambda^{(1 - \eps_i) / 2}
			a_{-i}\,
			F\bigl( \lambda^{(\eps_i - 1) / 2} \bigr)
		=
			a_j^*
			\lambda^{(1 - \eps_j) / 2}
			a_{-j}\,
			F\bigl( \lambda^{(\eps_j - 1) / 2} \bigr)\,
			F(x^*).
	\]
	In other words, there exists
	\(
		b \in R^*
	\) such that
	\(
		a_i^* a_{-i} = b
	\) for
	\(
		i > 0
	\),
	\(
		F(\lambda) = b^{-1} b^* \lambda
	\), and
	\(
		F(x^*) = b^{-1}\, f(x)^*\, b
	\). In particular,
	\(
			a_i^*
			\lambda^{(1 - \eps_i) / 2}
			a_{-i}\,
			F\bigl( \lambda^{(\eps_i - 1) / 2} \bigr)
		=
			b
	\) for all \(i\). Now let
	\(
		H(u) = H_i(u) \cdot a_i
	\) for some \(i\), this is a group isomorphism. It is easy to see using the remaining identities between
	\(
		F_{ij}
	\) and
	\(
		H_i
	\) that \(H\) is independent of \(i\) and satisfies the required identities. Conversely, for any such tuple
	\(
		(F, H, a_i, b)_i
	\) the resulting object
	\(
		(R', \Delta')
	\) is a
	\(
		(\mathsf B_\ell, \mathsf B_\ell)
	\)-ring and \(Q'\) is a well-defined homomorphism.

	Finally, suppose that
	\(
		\ell = 3
	\). The triple
	\(
		(F_{12}, F_{23}, F = F_{13})
	\) satisfying
	\(
		F_{13}(xy) = F_{12}(x)\, F_{23}(y)
	\) is called an \textit{isotopy} between \(R'\) and \(R\) \cite{mccrimmon}. It is known \cite[theorem 2]{mccrimmon} that in this case there are unique
	\(
		a, b \in R^*
	\) such that
	\(
		F(1) = (a b)^{-1}
	\),
	\(
		F_{12}(x) = F(x)\, a
	\), and
	\(
		F_{23}(x) = b\, F(x)
	\) (namely,
	\(
		a = F_{23}(1)^{-1}
	\) and
	\(
		b = F_{12}(1)^{-1}
	\)). Using the isotopy identities between
	\(
		F_{ij}
	\) we may express
	\(
		F_{ij}
	\) in terms of \(F\), \(a\), \(b\) for
	\(
		i, j > 0
	\). Let
	\(
		H = H_3
	\), then
	\(
		H_2(u) = H(u) \cdot a
	\) and
	\(
		H_1(u) = H(u) \cdot a b
	\). From the identity
	\(
		\rho(H_i(u))\, F_{ij}(x)
		= F_{-i, j}(\rho(u)\, x)
	\) for
	\(
		i > 0
	\) we get
	\(
		F_{-i, j}(x) = \nu_i\, F_{ij}(x)
	\) for
	\(
		i > 0
	\), where
	\[
		\nu_1
		= \rho(H_1(\iota))
		= b^* \nu_2 b
		= (a b)^*\, \nu_3\, (a b),
	\quad
		\nu_2 = \rho(H_2(\iota)) = a^* \nu_3 a,
	\quad
		\nu_3 = \rho(H_3(\iota)).
	\]
	Note that
	\(
		a^* = \nu_2 a^{-1} \nu_3^{-1}
	\) and
	\(
		b^* = \nu_1 b^{-1} \nu_2^{-1}
	\), where all factors except
	\(
		a^{-1}
	\) and
	\(
		b^{-1}
	\) lie in the nucleus, so the \(*\)-subring of \(R\) generated by
	\(
		a^{\pm 1}
	\),
	\(
		b^{\pm 1}
	\), \(\lambda\), and all
	\(
		\nu_i
	\) is associative. Also, there is the dual identity
	\[
			F_{i, -j}(x)\,
			\lambda^{(\eps_j - 1) / 2}\,
			\rho(H_j(u))
		=
			F_{ij}\bigl(
				x
				\lambda^{(\eps_j - 1) / 2}\,
				\rho(u)
			\bigr)
	\]
	obtained from the identities for
	\(
		\rho(H_i(u))\, F_{ij}(x)
	\) and
	\(
		F_{ij}(x)^*
	\). This dual identity easily implies
	\(
		F_{i, -j}(x) = F_{ij}(x)\, \nu_j^{-1}
	\) for
	\(
		j > 0
	\). The identities
	\(
		H_{-i}(u) = H_i(u) \cdot \nu_i^{-1}
	\) for
	\(
		i > 0
	\) follow from
	\(
		H_i(u) \cdot F_{ij}(1) = H_j(u)
	\) since
	\(
		F_{ij}(1)\, F_{j, -i}(1) = \nu_i^{-1}
	\) for
	\(
		i > 0
	\).

	Now we have
	\[
			F_{ij}(x)^*\,
			(\nu_i^* \lambda)^{(1 - \sigma) / 2}\,
			\nu_i^{(1 + \sigma) / 2}
		=
			(\lambda \nu_j^*)^{(1 - \tau) / 2}\,
			\nu_j^{(1 + \tau) / 2}\,
			F_{ji}\bigl(
				\lambda^{(\tau - 1) / 2}
				x^*
				\lambda^{(1 - \sigma) / 2}
			\bigr)
	\]
	for
	\(
		i, j > 0
	\) and
	\(
		\sigma, \tau \in \{-1, 1\}
	\). Take, for example,
	\(
		i = -2
	\) and
	\(
		j = 3
	\). This gives us the formulae for
	\(
		F(\lambda)
	\) (if
	\(
		x = 1
	\),
	\(
		\sigma = -1
	\), and
	\(
		\tau = 1
	\)),
	\(
		F(\lambda^*)
	\) (if
	\(
		x = 1
	\),
	\(
		\sigma = 1
	\), and
	\(
		\tau = -1
	\)), and
	\(
		F(x^*)
	\) (if
	\(
		\sigma = \tau = 1
	\)). The remaining identities are straightforward.

	To prove the converse statement note that for all elements
	\(
		a, b \in R^*
	\) and
	\(
		H(\iota) \in \Delta
	\) with invertible
	\(
		\rho(H(\iota)) = \nu_3
	\) there exists a corresponding homomorphism
	\(
		Q'
		\colon \stlin(\mathsf B_\ell, R', \Delta')
		\to G
	\). Indeed, the
	\(
		(\mathsf B_3, \mathsf B_3)
	\)-ring
	\(
		(R', \Delta')
	\) is constructed from \(G\) in the usual way using the Weyl elements
	\[
		t_1\bigl(H(\iota) \cdot \lambda\bigr)\,
		t_{-1}\bigl(
			H(\iota) \cdot (-\nu_3^{-1})
		\bigr)\,
		t_1\bigl(H(\iota)\bigr),
	\quad
		t_{12}(a)\, t_{21}(-a^{-1})\, t_{12}(a),
	\quad
		t_{23}(b)\, t_{32}(-b^{-1})\, t_{23}(b)
	\]
	instead of the standard ones, see lemmas \ref{weyl-a-long} and \ref{weyl-a-short}. So without loss of generality
	\(
		a = b = 1
	\) and
	\(
		H(\iota) = \iota
	\). This means that \(F\) and \(H\) preserve all operations (in particular,
	\(
		(R', \Delta')
	\) is an
	\(
		(\mathsf B_3, \mathsf B_3)
	\)-ring) and \(Q'\) is clearly well-defined.
\end{proof}

\begin{example} \label{bb-ring-b-c}
	There exist non-trivial
	\(
		(\mathsf B_\ell, \mathsf B_\ell)
	\)-rings with
	\(
		\lambda = \pm 1
	\) for all \(\ell\). For example, let \(R\) be a commutative unital ring with trivial involution and \(\Delta = R\). Then
	\[
		\lambda = 1,
	\quad
		\phi(x) = \dot 0,
	\quad
		\iota = 1,
	\quad
		\langle u, v \rangle = -2 u v,
	\quad
		\rho(u) = u^2,
	\quad
		u \cdot x = u x
	\]
	and
	\[
		\lambda = -1,
	\quad
		\phi(x) = 2 x,
	\quad
		\iota = 1,
	\quad
		\langle u, v \rangle = 0,
	\quad
		\rho(u) = u,
	\quad
		u \cdot x = u x^2
	\]
	are both
	\(
		(\mathsf B_\ell, \mathsf B_\ell)
	\)-rings corresponding to the groups
	\(
		\orth(2 \ell + 1, R)
	\) and
	\(
		\symp(2 \ell, R)
	\).
\end{example}

\begin{example} \label{bb-ring-init}
	We may construct the free
	\(
		(\mathsf B_\ell, \mathsf B_\ell)
	\)-ring with empty set of generators (it is independent of \(\ell\)). Namely, let
	\(
		R = \mathbb Z[\lambda^{\pm 1}]
	\) and
	\(
			\Delta
		=
			\phi(\mathbb Z[\lambda^{-1}])
			\dotoplus \bigoplus_{m \in \mathbb Z}^\cdot
				\iota \cdot \mathbb Z \lambda^m
	\). As a set \(\Delta\) is isomorphic to
	\(
		\mathbb Z[\lambda^{-1}]
		\times \mathbb Z^{(\mathbb Z)}
	\), where the second factor consists of all integer sequences with finite support. The operations on
	\(
		(R, \Delta)
	\) are uniquely determined by the axioms and it is straightforward to check that this is indeed a well-defined
	\(
		(\mathsf B_\ell, \mathsf B_\ell)
	\)-ring satisfying the universal property. This also follows from the description of the odd form ring \cite[lemma 6 and proposition 1]{thesis} generated by
	\(
		1_{ij}
	\) for
	\(
		i, j \in \{-, +\}
	\) in the ring part and
	\(
		\iota
	\) in the form part with the relations
	\(
		1_{ij} 1_{jk} = 1_{ik}
	\),
	\(
		1_{ij} 1_{kl} = 0
	\) for
	\(
		k \neq l
	\),
	\(
		\inv{1_{++}} = 1_{--}
	\),
	\(
		\iota = \iota \cdot 1_{++}
	\), and
	\(
		\widehat \rho(\iota) = 1_{-+}
	\). The Steinberg group
	\(
		\stlin(\mathsf B_\ell, R, \Delta)
	\) is
	\(
		\mathsf B_\ell
	\)-graded for any
	\(
		\ell \geq 3
	\) by general facts about odd unitary groups \cite[lemma 9]{thesis} or by our existence theorem \ref{exist}. The element
	\(
		t_1(\iota \cdot \lambda)\,
		t_{-1}(\iota \cdot (-1))\,
		t_1(\iota)
	\) is Weyl by lemma \ref{weyl-a-short}, but clearly
	\(
		\iota \cdot \lambda \neq \iota
	\) unlike Weyl elements in Chevalley groups.
\end{example}

\begin{example} \label{bb-ring-e7}
	Consider the surjective map
	\[
		\xi
		\colon \mathsf E_7
		\to \mathsf C_3 \cup \{0\},
	\quad
		\mathrm e_i \mapsto 0
		\text{ for } 1 \leq i \leq 4,
	\quad
		\mathrm e_5 \mapsto \mathrm e_1 + \mathrm e_2,
	\quad
		\mathrm e_6 \mapsto \mathrm e_1 - \mathrm e_2,
	\quad
		\mathrm e_7 \mapsto \mathrm e_3,
	\quad
		\mathrm e_8 \mapsto -\mathrm e_3
	\]
	induced by a linear operator. Its kernel (i.e. the preimage of \(0\)) is a root subsystem of type
	\(
		\mathsf D_4
	\). Preimages of roots are as follows.
	\begin{align*}
		\xi^{-1}(\pm (\mathrm e_1 + \mathrm e_2))
		&= \{
			\sigma \mathrm e_i \pm \mathrm e_5
			\mid
			\sigma \in \{-1, 1\},\,
			1 \leq i \leq 4
		\},
	\\
		\xi^{-1}(\pm (\mathrm e_1 - \mathrm e_2))
		&= \{
			\sigma \mathrm e_i \pm \mathrm e_6
			\mid
			\sigma \in \{-1, 1\},\,
			1 \leq i \leq 4
		\},
	\\
		\xi^{-1}(
			\sigma \mathrm e_1
			+ \tau \mathrm e_3
		)
		&\textstyle= \bigl\{
			\frac 1 2 \sum_{p = 1}^8
				\upsilon_p \mathrm e_p
			\mid
			\upsilon_p \in \{-1, 1\},\,
			\upsilon_5 = \upsilon_6 = \sigma,\,
			\upsilon_7 = -\upsilon_8 = \tau,\,
			\sum_{p = 1}^6 \upsilon_p
			\equiv 0 \!\!\!\pmod 4
		\bigr\},
	\\
		\xi^{-1}(
			\sigma \mathrm e_2
			+ \tau \mathrm e_3
		)
		&\textstyle= \bigl\{
			\frac 1 2 \sum_{p = 1}^8
				\upsilon_p \mathrm e_p
			\mid
			\upsilon_p \in \{-1, 1\},\,
			\upsilon_5 = -\upsilon_6 = \sigma,\,
			\upsilon_7 = -\upsilon_8 = \tau,\,
			\sum_{p = 1}^6 \upsilon_p
			\equiv 0 \!\!\!\pmod 4
		\bigr\},
	\\
		\xi^{-1}(\pm 2 \mathrm e_1)
		&= \{\pm (\mathrm e_5 + \mathrm e_6)\},
	\quad
		\xi^{-1}(\pm 2 \mathrm e_2)
		= \{\pm (\mathrm e_5 - \mathrm e_6)\},
	\quad
		\xi^{-1}(\pm 2 \mathrm e_3)
		= \{\pm (\mathrm e_7 - \mathrm e_8)\}.
	\end{align*}
	It follows that
	\(
		\xi^{-1}(\mathsf A_1 \cup \{0\})
	\) is a root subsystem of type
	\(
		\mathsf D_5
	\) in the short root case and
	\(
		\mathsf A_1 + \mathsf D_4
	\) in the long root case. Any
	\(
		\mathsf E_7
	\)-graded group \(G\) is
	\(
		\mathsf B_3
	\)-graded via \(\xi\) since a new \(\alpha\)-Weyl element for
	\(
		\alpha \in \mathsf C_3
	\) may be constructed as product of \(\beta\)-Weyl elements for a maximal family of pairwise orthogonal
	\(
		\beta \in \xi^{-1}(\alpha)
	\) (such a family contains either \(2\) or \(1\) roots depending on the length of \(\alpha\)). In the case
	\(
		G = G^{\mathrm{ad}}(\mathsf E_7, K)
	\) the corresponding
	\(
		(\mathsf B_3, \mathsf B_3)
	\)-ring
	\(
		(R, \Delta)
	\) has the following structure. The group \(\Delta\) is a free \(K\)-module generated by \(\iota\), \(R\) is an \(8\)-dimensional unital \(K\)-algebra,
	\(
		\rho \colon \Delta \to R,\, \iota \mapsto 1
	\) is a \(K\)-linear injection,
	\(
		\langle u, v \rangle = 0
	\),
	\(
		\phi \colon R \to \Delta
	\) is linear and surjective,
	\(
		(u, x) \mapsto u \cdot x
	\) is linear on \(u\) and quadratic on \(x\), and the quadratic form
	\(
		x \mapsto \iota \cdot x
	\) splits. It follows that we may identify \(\Delta\) with
	\(
		K \subseteq R
	\) using \(\rho\),
	\(
		u \cdot x = u x^* x
	\), \(R\) is a composition algebra,
	\(
		\lambda = -1
	\),
	\(
		\phi(x) = x + x^*
	\). Actually, \(R\) splits (i.e. isomorphic to the Cayley--Dickson algebra from \cite[\S 2.4]{zhevlakov}) since there is a tower of Cayley--Dickson extensions
	\(
		K \times K \subseteq R_0 \subseteq R
	\) corresponding to the root subsystems
	\(
		\{
			\mathrm e_1 - \mathrm e_2,
			\mathrm e_2 - \mathrm e_3
		\}^\perp
		\subseteq (\mathrm e_1 - \mathrm e_2)^{\perp}
		\subseteq \mathsf E_7
	\).
\end{example}

Examples \ref{bb-ring-b-c}--\ref{bb-ring-e7} show that the free
\(
	(\mathsf B_3, \mathsf B_3)
\)-ring
\(
	(R, \Delta)
\) with \(3\) generators
\(
	x, y, z \in R
\) is non-associative and has zero divisors,
\(
	(1 + \lambda)\, [x, y, z] = 0
\).

\section{\(\mathsf F_4\)-graded groups}

In this section \(G\) is an
\(
	\mathsf F_4
\)-graded group. We fix root subsystems
\(
	\mathsf B_3 \subseteq \mathsf F_4
\) and
\(
	\mathsf C_3 \subseteq \mathsf F_4
\) corresponding to the cover of the Dynkin diagram of
\(
	\mathsf F_4
\) by two Dynkin diagrams of
\(
	\mathsf B_3
\). There are
\(
	(\mathsf B_3, \mathsf B_3)
\)-rings
\(
	(R, \Delta)
\) and
\(
	(S, \Theta)
\) parameterizing corresponding subgroups of \(G\) by maps \(t_{ij}\), \(t_i\) and \(t'_{ij}\), \(t'_i\) respectively. Notice that these two root subsystems intersect by a root subsystem of type
\(
	\mathsf B_2
\). By a suitable choice of indices, there are group isomorphisms
\[
	H_1,
	H_{-2},
	H_{-1},
	H_2
	\colon \Delta
	\to S,
\quad
	F_{-1, -2},
	F_{1, -2},
	F_{12},
	F_{-1, 2}
	\colon R
	\to \Theta
\]
such that
\begin{align*}
	t_{\pm 1}(u)
	&= t'_{\pm 1, \pm 2}(H_{\pm 1}(u)),
&
	t_{\pm 1, \pm 2}(x)
	&= t'_{\pm 1}(F_{\pm 1, \pm 2}(x)),
\\
	t_{\pm 2}(u)
	&= t'_{\mp 1, \pm 2}(H_{\pm 2}(u)),
&
	t_{\mp 1, \pm 2}(x)
	&= t'_{\pm 2}(F_{\mp 1, \pm 2}(x)).
\end{align*}
Moreover, since Weyl elements from the constructions of
\(
	(R, \Delta)
\) and
\(
	(S, \Theta)
\) may be chosen arbitrarily, we may further assume that
\(
	H_1(\iota_{\Delta}) = 1_S
\) and
\(
	F_{1 2}(1_R) = \iota_\Theta
\). Comparing commutator formulae, we get
\begin{align*}
	H_1(u \cdot \lambda^* x^* \lambda)
	&= \lambda^*\, \rho(F_{-1, -2}(x))\, H_2(u),
&
	F_{-1, 2}(x)
	\cdot \lambda^*\,
	H_{-1}(u)^*\,
	\lambda
	&= F_{1 2}(\lambda^*\, \rho(u)\, x),
\\
	H_1(u \cdot x^* \lambda)
	&= H_{-2}(u)\, \rho(F_{-1, 2}(x))^*\, \lambda,
&
	F_{1, -2}(x) \cdot H_2(u)^*\, \lambda
	&= F_{1 2}(x\, \rho(u)^*\, \lambda),
\\
	H_2(u \cdot x)
	&= H_{-1}(u)\, \rho(F_{-1, 2}(x)),
&
	F_{-1, -2}(x) \cdot H_2(u)
	&= F_{-1, 2}(x\, \rho(u)),
\\
	H_2(u \cdot x)
	&= \rho(F_{1 2}(x))^*\, \lambda\, H_1(u),
&
	F_{1 2}(x) \cdot H_1(u)
	&= F_{-1, 2}(\rho(u)^*\, \lambda x),
\\
	H_{-1}(u \cdot x^*)
	&= \rho(F_{1 2}(x))\, H_{-2}(u),
&
	F_{1, -2}(x) \cdot H_1(u)^*
	&= F_{-1, -2}(\rho(u)\, x),
\\
	H_{-1}(u \cdot \lambda^* x^*)
	&= H_2(u)\,
	\lambda^*\,
	\rho(F_{1, -2}(x))^*\,
	\lambda,
&
	F_{-1, 2}(x) \cdot \lambda^*\, H_{-2}(u)^*
	&= F_{-1, -2}(x \lambda^*\, \rho(u)^*\, \lambda),
\\
	H_{-2}(u \cdot x)
	&= H_1(u)\, \lambda^* \rho(F_{1, -2}(x)),
&
	F_{1 2}(x) \cdot H_{-2}(u)
	&= F_{1, -2}(x \lambda^*\, \rho(u)),
\\
	H_{-2}(u \cdot x)
	&= \lambda^*\,
	\rho(F_{-1, -2}(x))^*\,
	\lambda\,
	H_{-1}(u),
&
	F_{-1, -2}(x) \cdot H_{-1}(u)
	&= F_{1, -2}(\lambda^*\, \rho(u)^*\, \lambda x),
\\
	\lambda^*\,
	\bigl\langle
		F_{-1, -2}(x),
		F_{-1, 2}(y)
	\bigr\rangle
	&= H_1(\phi(x y^* \lambda)),
&
	F_{1 2}(\lambda^*\, \langle u, v \rangle)
	&= \phi(H_{-1}(u)\, H_2(v)^*\, \lambda),
\\
	\bigl\langle F_{1 2}(x), F_{-1, 2}(y) \bigr\rangle
	&= H_2(\phi(x^* y)),
&
	F_{-1, 2}(\langle u, v \rangle)
	&= \phi(H_1(u)^*\, H_2(v)),
\\
	\bigl\langle F_{1 2}(x), F_{1, -2}(y) \bigr\rangle
	&= H_{-1}(\phi(\lambda x \lambda^* y^*)),
&
	F_{-1, -2}(\langle u, v \rangle)
	&= \phi(
		\lambda\,
		H_1(u)\,
		\lambda^*\,
		H_{-2}(v)^*
	),
\\
	\lambda^*\,
	\bigl\langle
		F_{-1, -2}(x),
		F_{1, -2}(y)
	\bigr\rangle
	&= H_{-2}(\phi(x^* \lambda y)),
&
	F_{1, -2}(\lambda^*\, \langle u, v \rangle)
	&= \phi(H_{-1}(u)^*\, \lambda\, H_{-2}(v)).
\end{align*}
for all
\(
	x, y \in R
\) and
\(
	u, v \in \Delta
\).

\begin{lemma} \label{f4-ring-1}
	The maps
	\begin{align*}
		H(u)
		&= H_1(u)
		= \lambda^*\, H_2(u)
		= H_{-1}(u)\, \lambda
		= H_{-2}(u)\, \lambda,
	\\
		F(x)
		&= F_{1 2}(x)
		= F_{-1, 2}(\lambda x)
		= F_{1, -2}(x \lambda^*)
		= F_{-1, -2}(x \lambda^*)
	\end{align*}
	satisfy the identities
	\begin{align*}
		\lambda_S^2 &= 1_S,
	&
		\lambda_R^2 &= 1_R,
	\\
		\lambda_S\, H(u) &= H(u)\, \lambda_S,
	&
		\lambda_R x &= x \lambda_R,
	\\
		\rho(F(x))^* &= \rho(F(x)),
	&
		\rho(u)^* &= \rho(u),
	\\
		H(u)\, \rho(F(x)) &= \rho(F(x))\, H(u),
	&
		x\, \rho(u) &= \rho(u)\, x,
	\\
		H(u \cdot x) &= \rho(F(x))\, H(u),
	&
		F(x) \cdot H(u) &= F(\rho(u)\, x),
	\\
		u \cdot x^* &= u \cdot x,
	&
		F(x) \cdot H(u)^* &= F(x) \cdot H(u),
	\\
		u \cdot \lambda &= u,
	&
		F(x) \cdot \lambda &= F(x),
	\\
		\phi(x^*) &= \phi(x) = -\phi(x \lambda_R),
	&
		\phi(H(u)^*)
		&= \phi(H(u))
		= -\phi(H(u)\, \lambda_S),
	\\
		\phi(x y) &= \phi(y x),
	&
		\phi(H(u)\, H(v)) &= \phi(H(v)\, H(u)),
	\\
		\bigl\langle F(x), F(y) \bigr\rangle
		&= H(\phi(-x^* y)),
	&
		F(\langle u, v \rangle)
		&= \phi(-H(u)^*\, H(v)).
	\end{align*}
	These identities are equivalent to the list above modulo axioms of
	\(
		(\mathsf B_3, \mathsf B_3)
	\)-rings.
\end{lemma}
\begin{proof}
	Let us simplify the top \(16\) identities from the list above. Take
	\(
		x = 1
	\) in
	\(
		F_{1 2}(x)
	\) and
	\(
		u = \iota
	\) in
	\(
		H_1(u)
	\). We get the identities
	\[
		H_2(u) = \lambda H_1(u),
	\quad
		F_{1 2}(x) = F_{-1, 2}(\lambda x),
	\quad
		H_{-1}(u) = H_{-2}(u),
	\quad
		F_{1, -2}(x) = F_{-1, -2}(x).
	\]
	Now let
	\(
		x = \lambda
	\) in
	\(
		F_{-1, 2}(x)
	\) and
	\(
		u = \iota
	\) in
	\(
		H_2(u)
	\), we obtain
	\[
		H_1(u) = H_{-2}(u)\, \lambda,
	\quad
		F_{1, -2}(x) = F_{1 2}(x \lambda),
	\quad
		H_2(u \cdot \lambda) = H_{-1}(u),
	\quad
		F_{-1, -2}(x) \cdot \lambda = F_{-1, 2}(x).
	\]
	This implies that \(F\) and \(H\) are well-defined and
	\[
		H(u \cdot \lambda^{\pm 1})
		= \lambda^{\mp 1}\, H(u)\, \lambda^{\mp 1},
	\quad
		F(x) \cdot \lambda^{\pm 1}
		= F(\lambda^{\mp 1} x \lambda^{\mp 1}).
	\]
	Now express all identities using \(F\) and \(H\). After simplification we get the first \(14\) identities from the statement. The remaining \(6\) identities are equivalent to the bottom \(8\) identities in the list above.
\end{proof}

Now we may identify \(\Delta\) with \(S\) and \(\Theta\) with \(R\) using the isomorphisms \(F\) and \(H\). By lemma \ref{f4-ring-1} the operations
\(
	({-}) \cdot ({=})
\) and
\(
	\langle {-}, {=} \rangle
\) may be expressed in terms of the ring multiplications, \(\rho\), and \(\phi\). Also,
\(
	\lambda = \rho(-1)
\) both in \(R\) and \(S\). We call the resulting object an \textit{\((\mathsf F_4, \mathsf F_4)\)-ring}. Formally, an abstract
\(
	(\mathsf F_4, \mathsf F_4)
\)-ring \((R, S)\) consists of two alternative unital ring \(R\) and \(S\) with involutions
\(
	x \mapsto x^*
\), additive maps
\(
	\phi \colon R \to \mathrm C(S)
\) and
\(
	\phi \colon S \to \mathrm C(R)
\), and with maps
\(
	\rho \colon R \to \mathrm C(S)
\) and
\(
	\rho \colon S \to \mathrm C(R)
\) such that
\begin{align*}
	[x^*, y, z] &= -[x, y, z],
&
	u\, \phi(x) &= \phi(\rho(u)\, x),
\\
	[x, y, z]^* &= -[x, y, z],
&
	\rho(1) &= 1,
\\
	\phi(x y) &= \phi(y x),
&
	\rho(x)\, \rho(y) &= \rho(x y),
\\
	\phi((x y)\, z) &= \phi(x\, (y z)),
&
	\rho(x + y) &= \rho(x) + \phi(x^* y) + \rho(y),
\\
	\phi(x^*) &= \phi(x)^* = \phi(x),
&
	\rho(x^*) &= \rho(x)^* = \rho(x),
\\
	\phi(\phi(x)) &= 0,
&
	\rho(\rho(x)) &= x^* x,
\\&&
	\rho(\phi(x)) + \phi(\rho(x)) &= x + x^*,
\end{align*}
for
\(
	x, y, z \in R
\) and
\(
	u \in S
\) or vice versa. Note that an element \(a\) of an
\(
	(\mathsf F_4, \mathsf F_4)
\)-ring is invertible if and only if
\(
	\rho(a)
\) is invertible (and the inverse of
\(
	\rho(a)
\) is necessarily central).

\begin{example} \label{f4-ring-chev}
	For any commutative unital ring \(K\) consider the
	\(
		(\mathsf F_4, \mathsf F_4)
	\)-ring
	\(
		(R, S)
	\) constructed by
	\(
		G^{\mathrm{ad}}(\mathsf F_4, K)
	\), where \(R\) parameterize long roots and \(S\) parameterize short roots. It is easy to see using example \ref{bb-ring-b-c} that
	\(
		R \cong S \cong K
	\) and under these isomorphisms
	\(
		\zeta^* = \zeta
	\) for
	\(
		\zeta \in R \cup S
	\),
	\(
		\phi(x) = 0
	\) for
	\(
		x \in R
	\),
	\(
		\phi(u) = 2 u
	\) for
	\(
		u \in S
	\),
	\(
		\rho(x) = x^2
	\) for
	\(
		x \in R
	\),
	\(
		\rho(u) = u
	\) for
	\(
		u \in S
	\).
\end{example}

\begin{example} \label{f4-ring-init}
	Let us construct the free
	\(
		(\mathsf F_4, \mathsf F_4)
	\)-ring
	\(
		(R, S)
	\) with empty set of generators. Since \(R\) and \(S\) as involution rings are generated by the sets \(\phi(S) \cup \rho(S)\) and \(\phi(R) \cup \rho(R)\) respectively, they are commutative with trivial involutions. It is easy to check that
	\(
		R = \mathbb Z \oplus \mathbb Z \lambda
	\) and
	\(
		S = \mathbb Z \oplus \mathbb Z \lambda
	\) with
	\begin{align*}
		\lambda
		&= \rho(-1) = \phi(1) - 1,
	&
		\phi(a + b \lambda)
		&= (a - b) + (a - b) \lambda,
	\\
		(a + b \lambda) (c + d \lambda)
		&= (a c + b d) + (a d + b c) \lambda,
	&
		\rho(a + b \lambda)
		&= \textstyle \frac{(a - b)^2 + a + b}2
		+ \frac{(a - b)^2 - a - b}2\, \lambda.
	\end{align*}
	Clearly,
	\(
		R^* = \{-1, 1, -\lambda, \lambda\}
	\) and
	\(
		S^* = \{-1, 1, -\lambda, \lambda\}
	\). A corresponding
	\(
		\mathsf F_4
	\)-graded group exists by our existence theorem \ref{exist}.
\end{example}

\begin{example} \label{f4-ring-e}
	Consider the surjective map
	\[
		\xi
		\colon \mathsf E_8
		\to \mathsf F_4 \cup \{0\},
	\quad
		\mathrm e_i \mapsto \mathrm e_i
		\text{ for } 1 \leq i \leq 4,
	\quad
		\mathrm e_i \mapsto 0
		\text{ for } 5 \leq i \leq 8
	\]
	induced by a linear operator. It also restricts to surjections
	\(
		\mathsf E_7 \to \mathsf F_4 \cup \{0\}
	\) and
	\(
		\mathsf E_6 \to \mathsf F_4
	\). It is easy to see that
	\(
		\Ker(\xi) \cong \mathsf D_4
	\),
	\(
		\Ker(\xi) \cap \mathsf E_7 \cong 3 \mathsf A_1
	\), and
	\(
		\Ker(\xi) \cap \mathsf E_6 = \varnothing
	\). Also,
	\begin{align*}
		\xi^{-1}(\pm \mathrm e_i)
		&= \{
			\pm \mathrm e_i + \sigma \mathrm e_j
			\mid
			\sigma \in \{-1, 1\},\,
			5 \leq j \leq 8
		\},
	\\
		\xi^{-1}(\sigma \mathrm e_i + \tau \mathrm e_j)
		&= \{\sigma \mathrm e_i + \tau \mathrm e_j\},
	\\
		\xi^{-1}\bigl(
			\textstyle \frac 1 2
			\sum_{i = 1}^4 \sigma_i \mathrm e_i
		\bigr)
		&= \bigl\{
			\textstyle \frac 1 2 \sum_{i = 1}^8
				\sigma_i \mathrm e_i
			\mid
			\sigma_5, \sigma_6, \sigma_7, \sigma_8
			\in \{-1, 1\};\,
			\sum_{i = 1}^8 \sigma_p
			\equiv 0 \!\!\!\pmod 4
		\bigr\}.
	\end{align*}
	It follows that
	\[
		\xi^{-1}\bigl(\mathsf A_1 \cup \{0\}\bigr)
		\cong \mathsf D_5,
	\quad
		\xi^{-1}\bigl(\mathsf A_1 \cup \{0\}\bigr)
		\cap \mathsf E_7
		\cong \mathsf A_1 + \mathsf A_3,
	\quad
		\xi^{-1}\bigl(\mathsf A_1 \cup \{0\}\bigr)
		\cap \mathsf E_6
		\cong \mathsf A_1 + \mathsf A_1
	\]
	in the short root case and
	\[
		\xi^{-1}\bigl(\mathsf A_1 \cup \{0\}\bigr)
		\cong \mathsf A_1 + \mathsf D_4,
	\quad
		\xi^{-1}\bigl(\mathsf A_1 \cup \{0\}\bigr)
		\cap \mathsf E_7
		\cong 4 \mathsf A_1,
	\quad
		\xi^{-1}\bigl(\mathsf A_1 \cup \{0\}\bigr)
		\cap \mathsf E_6
		\cong \mathsf A_1
	\]
	in the long root case. Any
	\(
		\mathsf E_\ell
	\)-graded group \(G\) is
	\(
		\mathsf F_4
	\)-graded via \(\xi\) since a new \(\alpha\)-Weyl element for
	\(
		\alpha \in \mathsf F_4
	\) may be constructed as product of \(\beta\)-Weyl elements for a maximal family of pairwise orthogonal
	\(
		\beta \in \xi^{-1}(\alpha) \cap \mathsf E_\ell
	\) (such a family contains either \(2\) or \(1\) roots depending on the length of \(\alpha\)). In the case of
	\(
		G = G^{\mathrm{ad}}(\mathsf E_\ell, K)
	\) the corresponding
	\(
		(\mathsf F_4, \mathsf F_4)
	\)-ring
	\(
		(R, S)
	\) has the following structure. The ring \(R\) is canonically isomorphic to \(K\) with the trivial involution, \(S\) is a
	\(
		2^{\ell - 5}
	\)-dimensional unital \(K\)-algebra with involution,
	\(
		\phi \colon R \to S
	\) is zero,
	\(
		\phi \colon S \to R
	\) and
	\(
		\rho \colon R \to S
	\) are \(K\)-linear,
	\(
		\rho \colon S \to R
	\) is a quadratic map. Moreover, the quadratic form
	\(
		\rho \colon S \to R
	\) splits. It follows that \(R\) may be identified with
	\(
		K \subseteq S
	\) via
	\(
		\rho \colon R \to S
	\),
	\(
		\phi(u) = u^* + u
	\) and
	\(
		\rho(u) = u^* u
	\) for
	\(
		u \in S
	\), \(S\) is a composition algebra,
	\(
		\lambda_S = -1
	\),
	\(
		\lambda_R = 1
	\). Since the rings \(S\) for
	\(
		\mathsf E_6
		\subseteq \mathsf E_7
		\subseteq \mathsf E_8
	\) form a tower of Cayley--Dickson extensions and the smallest one of the splits, the remaining two also split. Groups from example \ref{bb-ring-e7} correspond to
	\(
		\mathsf C_3 \subseteq \mathsf F_4
	\) in \(G\) since
	\(
		\xi^{-1}(\mathsf C_3) \cong \mathsf E_7
	\). Indeed,
	\(
		\mathsf C_3
	\) is the orthogonal complement to a long root in
	\(
		\mathsf F_4
	\), so its preimage is the orthogonal complement to a root in
	\(
		\mathsf E_8
	\).
\end{example}

For any
\(
	\mathsf F_4
\)-graded group \(G\) we have a corresponding
\(
	(\mathsf F_4, \mathsf F_4)
\)-ring
\(
	(R, S)
\). Fix a basis
\(
	\{\alpha_1, \alpha_2, \alpha_3, \alpha_4\}
	\subseteq \mathsf F_4
\) such that the root subsystems
\(
	\mathsf B_3
\) and
\(
	\mathsf C_3
\) from the construction of
\(
	(R, S)
\) are spanned by
\(
	\{\alpha_1, \alpha_2, \alpha_3\}
\) and
\(
	\{\alpha_2, \alpha_3, \alpha_4\}
\), namely,
\[
	G_{\alpha_1} = t_{23}(R),
\quad
	G_{\alpha_2} = t_{12}(R) = t'_1(R),
\quad
	G_{\alpha_3} = t_1(S) = t'_{12}(S),
\quad
	G_{\alpha_4} = t'_{23}(S),
\]
see also \cite[\S VI.4.9]{bourbaki}. For root subgroups
\(
	G_{\pm \alpha_2}
\) and
\(
	G_{\pm \alpha_3}
\) we use the parametrizations
\[
	t_{\alpha_2} = t_{12} \colon R \to G_{\alpha_2},
\quad
	t_{-\alpha_2} = t_{21} \colon R \to G_{-\alpha_2},
\quad
	t_{\alpha_3} = t'_{12} \colon S \to G_{\alpha_3},
\quad
	t_{-\alpha_3} = t'_{21} \colon S \to G_{-\alpha_3},
\]
also let us fix the standard Weyl elements
\begin{align*}
	n_1 &= t_{23}(1)\, t_{32}(-1)\, t_{23}(1),
	&
	n_3 &= t'_{12}(1)\, t'_{21}(-1)\, t'_{12}(1),
	\\
	n_2 &= t_{12}(1)\, t_{21}(-1)\, t_{12}(1),
	&
	n_4 &= t'_{23}(1)\, t'_{32}(-1)\, t'_{23}(1).
\end{align*}

We are going to construct parametrizations of remaining root subgroups, in particular, to choose parametrizations for
\(
	(\mathsf B_3 \cup \mathsf C_3)
	\setminus \{\pm \alpha_2, \pm \alpha_3\}
\) (recall that
\(
	t_{ij}
\) and
\(
	t_{-j, -i}
\) parameterize the same subgroups). Consider the following sets of terms in the language of
\(
	(\mathsf F_4, \mathsf F_4)
\)-rings, where we identify terms obtained by permutation of \(R\) and \(S\).
\begin{align*}
	\Upsilon
	&= \{-1, 1, -\lambda, \lambda\},
\\
	O_{\mathrm{sym}}
	&= \{
		\upsilon x,
		\upsilon x^*
		\mid
		\upsilon \in \Upsilon
	\},
\\
	O_{\pi / 2}
	&= \{
		\sigma\, \phi(x y),
		\sigma\, \phi(x^* y)
		\mid
		\sigma \in \{-1, 1\}
	\},
\\
	O_{2 \pi / 3}
	&= \{
		\upsilon x y,
		\upsilon x y^*,
		\upsilon x^* y,
		\upsilon x^* y^*,
		\upsilon y x,
		\upsilon y x^*,
		\upsilon y^* x,
		\upsilon y^* x^*
		\mid
		\upsilon \in \Upsilon
	\},
\\
	O_{3 \pi / 4}
	&= \{
		\upsilon x\, \rho(u),
		\upsilon x^*\, \rho(u)
		\mid
		\upsilon \in \Upsilon
	\}.
\end{align*}
These terms are indeed distinct. Namely, let
\[
	G_{\mathrm{std}}
	= G^{\mathrm{ad}}(\mathsf E_7, \mathbb C)
	\times G^{\mathrm{ad}}(\mathsf E_7, \mathbb C).
\]
where the
\(
	\mathsf F_4
\) grading on the first factor is given as in example \ref{f4-ring-e}, and on the second factor it is given as in the same example followed by an outer automorphism of
\(
	\mathsf F_4
\) turning over the Dynkin diagram. The corresponding
\(
	(\mathsf F_4, \mathsf F_4)
\)-ring is
\(
	(
		\mathbb C \times \mathbb O_{\mathbb C},
		\mathbb O_{\mathbb C} \times \mathbb C
	)
\), where
\(
	\mathbb O_{\mathbb C}
\) is the (split) octonion algebra over
\(
	\mathbb C
\). For this
\(
	(\mathsf F_4, \mathsf F_4)
\)-ring the terms in \(\Upsilon\),
\(
	O_{\mathrm{sym}}
\),
\(
	O_{\pi / 2}
\),
\(
	O_{2 \pi / 3}
\),
\(
	O_{3 \pi / 4}
\) represent distinct elements and functions. Clearly, \(\Upsilon\) is an elementary abelian group under multiplication. The set
\(
	O_{\mathrm{sym}}
\) is an elementary abelian group under composition.

Unlike the case of
\(
	\mathsf E_\ell
\)-groups, we cannot right now prove the full analogue of lemma \ref{e-braid} since the
\(
	\mathsf F_4
\)-analogue of
\(
	T^{\mathrm{ad}}(\mathsf E_\ell, \mathbb Z)
\) acts on the root subgroups by elements of \(\Upsilon\) and multiplication by \(\lambda\) cannot be introduced before coordinatisation.
\begin{lemma} \label{f4-braid}
	Let
	\(
		\alpha,
		\beta,
		\alpha + \beta,
		\gamma
		\in \mathsf F_4
	\) be roots of the same length. Then there exists a sign
	\(
		\sigma \in \{-1, 1\}
	\) such that for every
	\(
		\mathsf F_4
	\)-graded group \(G\) and every \(\alpha\)- and \(\beta\)-Weyl elements
	\(
		n_\alpha
	\) and
	\(
		n_\beta
	\) the identity
	\(
		\up{(n_\alpha n_\beta)^3}g = g^\sigma
	\) holds for every
	\(
		g \in G_\gamma
	\).
\end{lemma}
\begin{proof}
	This follows from lemma \ref{weyl-a-long} applied to a root subsystem of non-crystallographic type
	\(
		\mathsf B_3
	\) containing all these roots. Note that the root subsystem spanned by \(\alpha\), \(\beta\), and \(\gamma\) is necessarily of type
	\(
		\mathsf A_2
	\) or
	\(
		\mathsf B_3
	\).
\end{proof}

\begin{theorem} \label{f4-consts}
	There exist terms
	\begin{align*}
		d_{i, \alpha}(x)
		&\in O_{\mathrm{sym}}
		\text{ for }
			1 \leq i \leq 4
		\text{ and }
			\alpha \in \mathsf F_4;
	\\
		C_{\alpha \beta}^{\alpha + \beta}(x, y)
		&\in O_{\pi / 2}
		\text{ for }
			\alpha \perp \beta
			\text{ and } |\alpha| = |\beta| = 1;
	\\
		C_{\alpha \beta}^{(\alpha + \beta) / 2}(x, y)
		&\in O_{\pi / 2}
		\text{ for }
			\alpha \perp \beta
			\text{ and } |\alpha| = |\beta| = \sqrt 2;
	\\
		C_{\alpha \beta}^{\alpha + \beta}(x, y)
		&\in O_{2 \pi / 3}
		\text{ for }
		\widehat{\alpha \beta} = 2 \pi / 3;
	\\
		C_{\alpha \beta}^{2 \alpha + \beta}(u, x),
		C_{\alpha \beta}^{\alpha + \beta}(x, u)
		&\in O_{3 \pi / 4}
		\text{ for }
			\widehat{\alpha \beta} = 3 \pi / 4,\,
			|\alpha| = 1,
			\text{ and } |\beta| = \sqrt 2;
	\\
		C_{\alpha \beta}^{\alpha + \beta}(u, x),
		C_{\alpha \beta}^{\alpha + 2 \beta}(x, u)
		&\in O_{3 \pi / 3}
		\text{ for }
			\widehat{\alpha \beta} = 3 \pi / 4,\,
			|\alpha| = \sqrt 2,
			\text{ and } |\beta| = 1
	\end{align*}
	with the following property. For every
	\(
		\mathsf F_4
	\)-graded group \(G\) there are (necessarily unique) parametrizations
	\(
		t_\alpha \colon P_\alpha \to G_\alpha
	\) for
	\(
		\alpha \neq \pm \alpha_2, \pm \alpha_3
	\) and
	\(
		P_\alpha \in \{R, S\}
	\) such that
	\[
		\up{n_i}{t_\alpha(x)}
		= t_{s_i(\alpha)}(d_{i, \alpha}(x))
	\]
	for
	\(
		1 \leq i \leq 4
	\),
	\(
		\alpha \in \mathsf F_4
	\) (where
	\(
		s_i = s_{\alpha_i}
	\)), the commutator formula
	\[
		[t_\alpha(\zeta), t_\beta(\eta)]
		= \prod_{\gamma \in \interval \alpha \beta}
			t_\gamma\bigl(
				C_{\alpha \beta}^\gamma(\zeta, \eta)
			\bigr)
	\]
	holds for
	\(
		\alpha \neq -\beta
	\), and
	\(
		t_\alpha(a)\,
		t_{-\alpha}(b)\,
		t_\alpha(c)
	\) is a Weyl element if and only if
	\(
		a = c
	\) is invertible and
	\(
		b = -a^{-1}
	\). Such a family of terms is unique up to the change
	\(
		d_{i, \alpha}(\zeta)
		\mapsto \upsilon_{s_i(\alpha)}^{-1}\bigl(
			d_{i, \alpha}\bigl(
				\upsilon_\alpha(\zeta)
			\bigr)
		\bigr)
	\) and
	\(
		C_{\alpha \beta}^\gamma(\zeta, \eta)
		\mapsto \upsilon_{\gamma}^{-1}\bigl(
			C_{\alpha \beta}^\gamma\bigl(
				\upsilon_\alpha(\zeta),
				\upsilon_\beta(\eta)
			\bigr)
		\bigr)
	\), where
	\(
		\upsilon_\alpha
		= \upsilon_{-\alpha}
		\in \Upsilon
	\) and
	\(
		\upsilon_{\alpha_2}
		= \upsilon_{\alpha_3}
		= x
	\).
\end{theorem}
\begin{proof}
	We cannot proceed as in lemma \ref{e-braid} and theorem \ref{group-e} since multiplication by
	\(
		\lambda \in \Upsilon
	\) is not expressible in the group-theoretic language. At first we find the required terms and construct
	\(
		t_\alpha
	\) not necessarily satisfying the characterization of Weyl elements, and at the end we change
	\(
		t_\alpha
	\) for negative \(\alpha\) by composing them with suitable elements of
	\(
		O_{\mathrm{sym}}
	\). For
	\(
		\alpha \in \mathsf B_3 \cup \mathsf C_3
	\) we already have some homomorphisms
	\(
		t_\alpha
	\) uniquely defined up to the action of
	\(
		O_{\mathrm{sym}}
	\), see lemma \ref{f4-ring-1} for
	\(
		\mathsf B_3 \cap \mathsf C_3
	\). Now let
	\(
		\height
		\colon \mathsf F_4 \cup \{0\}
		\to \mathbb Z
	\) be the \textit{height function}, i.e.
	\(
		\height(\alpha + \beta)
		= \height(\alpha) + \height(\beta)
	\),
	\(
		\height(0) = 0
	\), and
	\(
		\height(\alpha_i) = 1
	\). Let us construct
	\(
		t_\alpha
	\) for
	\(
		\height(\alpha) > 0
	\) and
	\(
		\alpha \notin \mathsf B_3 \cup \mathsf C_3
	\) by induction on
	\(
		\height(\alpha)
	\) as follows. There exists
	\(
		1 \leq i \leq 4
	\) such that
	\(
		\height(s_i(\alpha)) < \height(\alpha)
	\) (but necessarily still
	\(
		\height(s_i(\alpha)) > 0
	\)). Let
	\(
		t_\alpha(x) = \up{n_i}{t_{s_i(\alpha)}(o(x))}
	\) for some
	\(
		o \in O_{\mathrm{sym}}
	\). The same definition also works for
	\(
		\height(\alpha) > 1
	\) and
	\(
		\alpha \in \mathsf B_3 \cup \mathsf C_3
	\) for some explicit \(o\) by lemmas \ref{weyl-a-long} and \ref{weyl-a-short}. Note that our definition is ``independent'' on the choice of \(i\) in the following sense. If
	\(
		\height(s_i(\alpha)),
		\height(s_j(\alpha))
		< \height(\alpha)
	\) for
	\(
		i \neq j
	\), then
	\(
		|i - j| \geq 2
	\) (by inspecting the root diagram of
	\(
		\mathsf F_4
	\) \cite[figure 25]{atlas}),
	\(
		n_i n_j = n_j n_i
	\), and
	\[
		t_\alpha(x)
		= \up{n_i}{t_{s_i(\alpha)}(o(x))}
		= \up{n_i n_j}
			{t_{s_j(s_i(\alpha))}\bigl(o'(o(x))\bigr)}
		= \up{n_j n_i}
			{t_{s_i(s_j(\alpha))}\bigl(o'(o(x))\bigr)}
		= \up{n_j}
			{t_{s_i(\alpha)}\bigl(o''(o'(o(x)))\bigr)}
	\]
	for some
	\(
		o, o', o'' \in O_{\mathrm{sym}}
	\). It follows that
	\(
		\up{n_i}{t_\alpha(x)}
		= t_{s_i(\alpha)}(d_{i, \alpha}(x))
	\) holds for all \(i\) and \(\alpha\) with positive height for some universal terms
	\(
		d_{i, \alpha}
	\) unless
	\(
		s_i(\alpha) = \alpha
	\).

	In order to find
	\(
		d_{i, \alpha}
	\) satisfying this identity consider the following cases.
	\begin{itemize}

		\item
		If
		\(
			|\alpha| \neq |\alpha_i|
		\), then we clearly have to take
		\(
			d_{i, \alpha}(x) = x
		\).

		\item
		If
		\(
			|\alpha| = |\alpha_i|
		\) and
		\(
			\alpha \in \mathsf B_3 \cup \mathsf C_3
		\), then we may apply lemmas \ref{weyl-a-long} and \ref{weyl-a-short} since both \(\alpha\) and
		\(
			\alpha_i
		\) lie in
		\(
			\mathsf B_3
		\) or
		\(
			\mathsf C_3
		\).

		\item
		Suppose that
		\(
			|\alpha| = |\alpha_i|
		\) and there is \(j\) such that
		\(
			|i - j| \geq 2
		\) and
		\(
			\height(s_j(\alpha)) < \height(\alpha)
		\). If there exists an appropriate
		\(
			d_{i, s_j(\alpha)}
		\), then
		\(
			d_{i, \alpha} = d_{i, s_j(\alpha)}
		\) also exists since
		\[
			\up{n_i}{t_{\alpha}(x)}
			= \up{n_i n_j}{
				t_{s_j(\alpha)}(d_{j, \alpha}^{-1}(x))
			}
			= \up{n_j}{
				t_{s_j(\alpha)}\bigl(
					d_{i, s_j(\alpha)}(
						d_{j, \alpha}^{-1}(x)
					)
				\bigr)
			}
			= t_\alpha(d_{i, s_j(\alpha)}(x)).
		\]

		\item
		By considering the root diagram \cite[figure 25]{atlas}, the terms not covered by other cases are
		\(
			d_{2, \mu_{\mathrm{long}}}
		\) and
		\(
			d_{3, \mu_{\mathrm{short}}}
		\), where
		\(
			\mu_{\mathrm{long}}
		\) and
		\(
			\mu_{\mathrm{short}}
		\) are the highest roots of the corresponding lengths. We may find them using lemma \ref{f4-braid}.
	\end{itemize}

	By symmetry, the same holds for roots with negative height. Since
	\(
		d_{i, \alpha_{\pm i}}
	\) clearly exist by lemma \ref{weyl-a-long}, we have all
	\(
		d_{i, \alpha}
	\). In order to find
	\(
		C_{\alpha \beta}^\gamma
	\) we may apply a product of
	\(
		n_i
	\) to
	\(
		(\alpha, \beta)
	\) in order to move the roots into
	\(
		\mathsf B_3
	\) or
	\(
		\mathsf C_3
	\). Inside these root subgroups the required term clearly exists. Similarly, by lemmas \ref{weyl-a-long} and \ref{weyl-a-short} for every root \(\alpha\) with positive height there exists
	\(
		\upsilon \in O_{\mathrm{sym}}
	\) such that
	\(
		t_\alpha(a)\, t_{-\alpha}(b)\, t_\alpha(c)
	\) is Weyl if and only if
	\(
		a = c
	\) is invertible and
	\(
		b = -\upsilon(a)^{-1}
	\). We replace
	\(
		t_{-\alpha}
	\) by
	\(
		t_{-\alpha} \circ \upsilon
	\). The characterization of
	\(
		(-\alpha)
	\)-Weyl elements became valid since
	\(
		t_\alpha(a)\, t_{-\alpha}(b)\, t_\alpha(a)
		= t_{-\alpha}(b)\, t_\alpha(a)\, t_{-\alpha}(b)
	\) for such elements.
\end{proof}

For any abstract
\(
	(\mathsf F_4, \mathsf F_4)
\)-ring
\(
	(R, S)
\) we now may construct the Steinberg group
\(
	\stlin(\mathsf F_4, R, S)
\) with generators
\(
	t_\alpha(x)
\) for long \(\alpha\) and
\(
	x \in R
\),
\(
	t_\beta(\zeta)
\) for short \(\beta\) and
\(
	\zeta \in S
\). The relations are, as usual,
\(
	t_\alpha(x)\, t_\alpha(y) = t_\alpha(x + y)
\) for all \(\alpha\) and
\(
	[t_\alpha(x), t_\beta(y)]
	= \prod_{\gamma \in \interval \alpha \beta}
		t_\gamma\bigl(
			C_{\alpha \beta}^\gamma(x, y)
		\bigr)
\) for
\(
	\alpha \neq -\beta
\).

\begin{theorem} \label{group-f}
	Let \(G\) be a
	\(
		\mathsf F_4
	\)-graded group. Then there is an
	\(
		(\mathsf F_4, \mathsf F_4)
	\)-ring
	\(
		(R, S)
	\) and a homomorphism
	\(
		Q \colon \stlin(\mathsf F_4, R, S) \to G
	\) inducing isomorphisms between root subgroups.

	Any other such homomorphism
	\(
		Q' \colon \stlin(\mathsf F_4, R', S') \to G
	\) is of the following type. Choose additive isomorphisms \(F \colon R' \to R\), \(H \colon S' \to S\) and elements \(a, b \in R^*\), \(\zeta, \eta \in S^*\) such that
	\begin{align*}
		F(1) &= (a b)^{-1},
	&
		H(1) &= (\zeta \eta)^{-1},
	\\
		F(x y) &= (F(x)\, a)\, (b\, F(y)),
	&
		H(u v) &= (H(u)\, \zeta)\, (\eta\, H(v)),
	\\
		F(\lambda x) &= \lambda\, F(x),
	&
		H(\lambda u) &= \lambda\, H(u),
	\\
		F(x^*)\, (a b)
		&= \bigl(F(x)\, (a b)\bigr)^*,
	&
		H(u^*)\, (\zeta \eta)
		&= \bigl(H(u)\, (\zeta \eta)\bigr)^*,
	\\
		F(\phi(u))\, (a b)
		&= \phi\bigl(H(u)\, (\zeta \eta)\bigr),
	&
		H(\phi(x))\, (\zeta \eta)
		&= \phi\bigl(F(x)\, (a b)\bigr),
	\\
		F(\rho(u))\, (a b)
		&= \rho\bigl(H(u)\, (\zeta \eta)\bigr),
	&
		H(\rho(x))\, (\zeta \eta)
		&= \rho\bigl(F(x)\, (a b)\bigr).
	\end{align*}
	The root homomorphisms
	\(
		t_\alpha \colon P_\alpha \to G
	\) for
	\(
		P_\alpha \in \{R, S\}
	\) and
	\(
		t'_\alpha \colon P'_\alpha \to G
	\) for
	\(
		P'_\alpha \in \{R', S'\}
	\) are related via explicit formulae from theorem \ref{group-b} for
	\(
		\alpha \in \mathsf B_3 \cup \mathsf C_3
	\). The tuple
	\(
		(F, H, a, b, \zeta, \eta)
	\) is uniquely determined by \(Q'\).
\end{theorem}
\begin{proof}
	The first claim follows from theorem \ref{f4-consts}. To prove the second claim, we just rewrite theorem \ref{group-b} for
	\(
		\mathsf B_3
	\) and
	\(
		\mathsf C_3
	\) in terms of
	\(
		(\mathsf F_4, \mathsf F_4)
	\)-rings and impose conditions that both resulting identifications
	\(
		P'_\alpha \cong P_\alpha
	\) coincide for
	\(
		\alpha \in \mathsf B_3 \cap \mathsf C_3
	\). The relation
	\(
		F(\phi(u^* v))
		= (\nu_3 a b)^{-1}\,
		\phi\bigl(H(u)^*\, H(v)\bigr)
	\) turns out to be redundant since after applying the identities for
	\(
		F(\phi(x))
	\),
	\(
		H(x y)
	\), and
	\(
		H(x^*)
	\) in the left hand side we get
	\begin{align*}
			\phi\Bigl(
				\Bigl(
					\bigl(
						\bigl(
							(\zeta \eta)^*\,
							&H(u)^*\,
							(\zeta \eta)^{-1}
						\bigr)\, \zeta
					\bigr)\, (\eta\, H(v))
				\Bigr)\, (\zeta \eta)
			\Bigr)\, (a b)^{-1} \\
		&=
			\phi\Bigl(
				\bigl(
					(\zeta \eta)^*\,
					H(u)^*\,
					(\zeta \eta)^{-1}
				\bigr)\,
				\bigl(
					(\zeta \eta)\,
					H(v)\,
					(\zeta\, \eta)
				\bigr)
			\Bigr)\, (a b)^{-1} \\
		&=
			(a b)^{-1}\, \phi\bigl(
				\rho\bigl(\rho(\zeta \eta)\bigr)\,
				H(u)^*\,
				H(v)
			\bigr)
		=
			(\nu_3 a b)^{-1}\,
			\phi\bigl(H(u)^*\, H(v)\bigr).
	\end{align*}
	It remains to prove that for every such tuple
	\(
		(F, H, a, b, \zeta, \eta)
	\) the pair
	\(
		(R', S')
	\) is an
	\(
		(\mathsf F_4, \mathsf F_4)
	\)-ring and there exists a corresponding homomorphism \(Q'\). This may be done in the same way as in the proof of theorem \ref{group-b}.
\end{proof}

\section{Existence theorem}

We begin with a complementary result to theorem \ref{f4-consts}.
\begin{lemma} \label{f4-empty}
	For any abstract
	\(
		(\mathsf F_4, \mathsf F_4)
	\)-ring
	\(
		(R, S)
	\) there is an
	\(
		(\mathsf F_4, \varnothing)
	\)-ring
	\(
		(P_\alpha)_\alpha
	\) such that
	\(
		P_\alpha = R
	\) for long \(\alpha\),
	\(
		P_\alpha = S
	\) for short \(\alpha\), and
	\(
		C_{\alpha \beta}^\gamma
	\) are given by the terms from theorem \ref{f4-consts}.
\end{lemma}
\begin{proof}
	We have to check that the commutator terms
	\(
		C_{\alpha \beta}^\gamma
	\) satisfy all identities from theorem \ref{phi-0-ring}. Every such identity involves only roots from a root subsystem \(\Psi\) of the non-crystallographic type
	\(
		\mathsf B_3
	\) since the last axiom from theorem \ref{phi-0-ring} is vacuous otherwise. Now we may transform \(\Psi\) to the distinguished root subsystem
	\(
		\mathsf B_3
	\) or
	\(
		\mathsf C_3
	\). Notice that
	\[
		d_{i, \gamma}\bigl(
			C_{\alpha \beta}^\gamma(\zeta, \eta)
		\bigr)
		= C_{s_i(\alpha), s_i(\beta)}
			^{s_i(\gamma)}
			\bigl(
				d_{i, \alpha}(\zeta),
				d_{i, \beta}(\eta)
			\bigr)
	\]
	holds since this is true in
	\(
		G_{\mathrm{std}}
		= G^{\mathrm{ad}}(\mathsf E_7, \mathbb C)
		\times G^{\mathrm{ad}}(\mathsf E_7, \mathbb C)
	\).
\end{proof}

We prove the existence theorem of a \(\Phi\)-graded group with given
\(
	(\Phi, \Phi)
\)-ring in the following way. There exists a certain nilpotent group
\(
	G_0
\) analogous to the Lie algebra
\(
	\mathfrak g^{\mathrm{ad}}(\Phi, K)
	= \mathfrak t^{\mathrm{ad}}
	\oplus \bigoplus_{\alpha \in \Phi}
		\mathfrak g_\alpha
\) of a Chevalley group
\(
	G^{\mathrm{ad}}(\Phi, K)
\). Instead of the torus we take the group of all formal diagonal elements, i.e. the automorphism group of the
\(
	(\Phi, \varnothing)
\)-ring. Also, first order infinitesimals are insufficient for our purposes, so we use root elements parameterized by specifically truncated power series. The required \(\Phi\)-graded group is constructed as a subgroup of
\(
	\Aut(G_0)
\) generated by certain ``root automorphisms''.

We say that an element \(g\) of a group \(G\) with \(\Phi\)-commutator relations is \textit{diagonal} if
\(
	\up g{G_\alpha} = G_\alpha
\) for all \(\alpha\). For example, diagonal matrices are diagonal in
\(
	\glin(m, R)
\), as well as elements of
\(
	T^{\mathrm{ad}}(\Phi, K)
	\leq G^{\mathrm{ad}}(\Phi, K)
\). The next two lemmas complement lemmas \ref{weyl-a-long} and \ref{weyl-a-short} and we formulate them for
\(
	(\mathsf B_\ell, \mathsf A_{\ell - 1})
\)-rings for simplicity.

\begin{lemma} \label{a-swap-long}
	Let \(G\) be a group with
	\(
		\mathsf B_\ell
	\)-commutator relations and
	\(
		\mathsf A_{\ell - 1}
	\)-Weyl elements for
	\(
		\ell \geq 4
	\) or a
	\(
		\mathsf B_3
	\)-graded group,
	\(
		(R_{ij}, \Delta^0_i)_{ij}
	\) be the corresponding
	\(
		(\mathsf B_\ell, \mathsf A_{\ell - 1})
	\)-ring,
	\(
		i \neq \pm j
	\) be indices,
	\(
		x \in R_{ij}
	\),
	\(
		y \in R_{ji}
	\). Then the element
	\(
		t_{ij}(x)\, t_{ji}(y)
	\) may be written as
	\(
		t_{ji}(y')\, t_{ij}(x')\, h
	\) for some \(x'\), \(y'\) and for diagonal \(h\) if and only if
	\(
		1_i + x y
	\) is invertible in
	\(
		R_{ii}
	\) (so
	\(
		1_j + y x
	\) is invertible in
	\(
		R_{jj}
	\) by lemma \ref{inv-alter}). In this case necessarily
	\begin{align*}
		x' &= x + x y x,
	&
		\up h{t_i(u)}
		&= t_i(u \cdot (1 + x y)^{-1}),
	\\
		y' &= (1 + y x)^{-1}\, y = y\, (1 + x y)^{-1},
	&
		\up h{t_{-i}(u)}
		&= t_{-i}(u \cdot \inv{(1 + x y)}),
	\\
		\up h{t_{kl}(z)} &= t_{kl}(z),
	&
		\up h{t_j(u)}
		&= t_j(u \cdot (1 + y x)),
	\\
		\up h{t_k(u)} &= t_k(u),
	&
		\up h{t_{-j}(u)}
		&= t_{-j}(u \cdot \inv{(1 + y x)}^{-1}),
	\\
		\up h{t_{ik}(z)} &= t_{ik}(z + x\, (y z)),
	&
		\up h{t_{-i, j}(z)}
		&= t_{-i, j}\bigl(
			\inv{(1 + x y)}^{-1}\,
			(z + (z y)\, x)
		\bigr),
	\\
		\up h{t_{jk}(z)} &= t_{jk}(z - y'\, (x z)),
	&
		\up h{t_{i, -j}(z)}
		&= t_{i, -j}\bigl(
			(z + x\, (y z))\,
			\inv{(1 + y x)}^{-1}
		\bigr),
	\\
		\up h{t_{ki}(z)} &= t_{ki}(z - (z x)\, y'),
	&
		\up h{t_{ij}(z)}
		&= t_{ij}\bigl(
			(1 + x y)\,
			(z + (z y)\, x)
		\bigr),
	\\
		\up h{t_{kj}(z)} &= t_{kj}(z + (z y)\, x),
	&
		\up h{t_{ji}(z)}
		&= t_{ji}\bigl(
			(1 + y x)^{-1}\,
			(z - (z x)\, y')
		\bigr).
	\end{align*}
	for
	\(
		|k|, |l| \notin \{|i|, |j|\}
	\).
\end{lemma}
\begin{proof}
	We use that the ring
	\(
		\sMat{R_{--}}{R_{-+}}{R_{+-}}{R_{++}}
	\) is associative or a matrix algebra over an alternative ring. Conjugating
	\(
		t_{ik}(z)
	\),
	\(
		t_{jk}(z)
	\),
	\(
		t_{ki}(z)
	\),
	\(
		t_{kj}(z)
	\) by both sides of
	\(
		t_{ij}(x)\, t_{ji}(y)
		= t_{ji}(y')\, h\, t_{ij}(x')
	\), we get the corresponding formulae for
	\(
		\up h{t_{pq}(z)}
	\) and
	\begin{align*}
		x z &= x' z + x\, (y\, (x' z))
		\text{ for } z \in R_{jk},
	&
		y z &= y' z + y'\, (x\, (y z))
		\text{ for } z \in R_{ik},
	\\
		z x &= z x' + ((z x')\, y)\, x
		\text{ for } z \in R_{ki},
	&
		z y &= z y' + ((z y)\, x)\, y'
		\text{ for } z \in R_{kj}.
	\end{align*}
	Since the map
	\(
		R_{ik} \to R_{ik},\, z \mapsto z + x\, (y z)
	\) is invertible, there exists
	\(
		a \in R_{ii}
	\) such that
	\(
		a + x\, (y a) = 1_i
	\). We have
	\[
		[x, y, a]
		= [x, y, -x\, (y a)]
		= -([x, y, a]\, y)\, x,
	\]
	so again using the bijectivity
	\(
		[x, y, a] = 0
	\) and
	\(
		(1 + x y)\, a = 1_i
	\). Similarly, there exists
	\(
		b \in R_{jj}
	\) such that
	\(
		[x, y, b] = 0
	\) and
	\(
		b\, (1 + y x) = 1_j
	\). On the other hand,
	\(
		(1 + y x)\, (1 - y a x) = 1_j
	\), so
	\(
		1 + y x
	\) is invertible and
	\(
		b = 1 - y a x
	\). Similarly,
	\(
		1 + x y
	\) is invertible and
	\(
		a = 1 - x b y
	\). From this we easily obtain the formulae for \(x'\) and \(y'\).

	Now suppose that
	\(
		1 + x y
	\) is invertible and define \(x'\) and \(y'\) using the explicit formulae. The four identities above may be checked using symmetry and lemma \ref{inv-alter} (or the associativity law) as follows.
	\begin{align*}
			x z
		&=
			(x + x y x)\, z
			- (x y' x + x y x y' x)\, z
		=
			x' z
			- x\, (y'\, (x z))
			- x\, \Bigl(
				y\,
				\bigl(
					\bigl(x\, (1 + y x)^{-1}\bigr)\,
					(y\, (x z))
				\bigr)
			\Bigr) \\
		&=
			x' z
			- x\, (y'\, (x z))
			- x\, \bigl(
				y\,
				\bigl(x\, (y'\, (x z))\bigr)
			\bigr)
		=
			x' z
			- x'\, (y'\, (x z));
	\\
			y z
		&=
			y' z
			+ (y' x y)\, z
		=
			y' z
			+ y\, \bigl(
				\bigl((1 + x y)^{-1}\, x\bigr)\,
				(y z)
			\bigr)
		=
			y' z
			+ y'\, (x\, (y z)).
	\end{align*}
	The remaining formulae for
	\(
		\up h{t_\alpha(\zeta)}
	\) easily follow by decomposing
	\(
		t_\alpha(\zeta)
	\) into products of commutators of root elements with known conjugates.
\end{proof}

\begin{lemma} \label{a-swap-short}
	Let \(G\) be a group with
	\(
		\mathsf B_\ell
	\)-commutator relations and
	\(
		\mathsf A_{\ell - 1}
	\)-Weyl elements for
	\(
		\ell \geq 3
	\),
	\(
		(R_{ij}, \Delta^0_i)_{ij}
	\) be the corresponding
	\(
		(\mathsf B_\ell, \mathsf A_{\ell - 1})
	\)-ring, \(i\) be an index,
	\(
		u \in \Delta^0_i
	\),
	\(
		v \in \Delta^0_{-i}
	\). Then the element
	\(
		t_i(u)\, t_{-i}(v)
	\) may be written as
	\(
		t_{-i}(v')\, t_i(u')\, h
	\) for some \(u'\), \(v'\) and for diagonal \(h\) if and only if
	\(
		z
		= \widehat \rho(u)\, \widehat \rho(v)
		- u \circ v
		+ 1_{-i}
	\) is invertible in
	\(
		R_{-i, -i}
	\). In this case necessarily
	\begin{align*}
		u'
		&= u \cdot \inv z
		\dotplus v'
			\cdot (-\widehat \rho(u)\, \inv z),
	&
		\widehat \rho(u')
		&= \widehat \rho(u)\, \inv z,
	\\
		v'
		&= v \cdot z^{-1}
		\dotminus u
			\cdot (-\widehat \rho(v)\, z^{-1}),
	&
		\widehat \rho(v')
		&= \widehat \rho(v)\, z^{-1},
	\\
		\up h{t_j(w)}
		&= t_j\bigl(
			\dotminus v' \cdot (-u \circ w)
			\dotplus w
			\dotplus u'
				\cdot \inv z^{-1}\,
				(-v \circ w)
		\bigr),
	&
		\up h{t_{jk}(x)} &= t_{jk}(x),
	\\
		\up h{t_i(w)}
		&= t_i\bigl(
			\dotminus v' \cdot (-u \circ w)\, \inv z
			\dotplus w \cdot \inv z
			\dotplus u'
				\cdot \inv z^{-1}\,
				(-v \circ w)\,
				\inv z
		\bigr),
	&
		\up h{t_{ij}(x)} &= t_{ij}(\inv z^{-1} x),
	\\
		\up h{t_{-i}(w)} &= t_{-i}\bigl(
			\dotminus v' \cdot (-u \circ w)\, z^{-1}
			\dotplus w \cdot z^{-1}
			\dotplus u'
				\cdot \inv z^{-1}\,
				(-v \circ w)\,
				z^{-1}
		\bigr),
	&
		\up h{t_{-i, j}(x)} &= t_{-i, j}(z x),
	\end{align*}
	for
	\(
		j \neq \pm i
	\).
\end{lemma}
\begin{proof}
	Conjugating
	\(
		t_{-i, j}(x)
	\),
	\(
		t_j(w)
	\), and
	\(
		t_{ij}(x)
	\) by both sides of
	\(
		t_i(u)\, t_{-i}(v) = t_{-i}(v')\, t_i(u')\, h
	\) we obtain the identities
	\begin{align*}
		F_{-i, j}(x) &= z x,
	\\
		\rho(v')\, F_{-i, j}(x) &= \rho(v)\, x,
	\\
		v' \cdot F_{-i, j}(x)
		&= v \cdot x \dotminus u \cdot (-\rho(v)\, x),
	\\
		u' \circ F_j(w)
		&= \rho(u)\, (v \circ w) + u \circ w,
	\\
		F_j(w)
		\dotminus v' \cdot \bigl(u' \circ F_j(w)\bigr)
		&= u \cdot (-v \circ w) \dotplus w,
	\\
		\rho(v')\, \bigl(u' \circ F_j(w)\bigr)
		+ v' \circ F_j(w)
		&= v \circ w,
	\\
		\rho(u')\, F_{ij}(x) &= \rho(u)\, x,
	\\
		u' \cdot F_{ij}(x)
		\dotminus v'
			\cdot \bigl(-\rho(u')\, F_{ij}(x)\bigr)
		&= u \cdot x,
	\\
		\bigl(
			\rho(v')\, \rho(u')
			- v' \circ u'
			+ 1
		\bigr)\, F_{ij}(x)
		&= x,
	\end{align*}
	where
	\(
		\up h{t_\alpha(\zeta)}
		= t_\alpha(F_\alpha(\zeta))
	\) for currently unknown function
	\(
		F_\alpha
	\). The first one implies that \(z\) is invertible. Conversely, if \(z\) is invertible, then these identities hold precisely when \(u'\), \(v'\),
	\(
		F_{-i, j}
	\),
	\(
		F_j
	\),
	\(
		F_{ij}
	\) are defined using the formulae from the statement. The identities for
	\(
		\up h{t_i(w)}
	\) and
	\(
		\up h{t_{-i}(w)}
	\) follow by expanding
	\(
		t_{\pm i}(w)
	\) in terms of root elements with known conjugates.
\end{proof}

We also need a partial converse to these two lemmas for abstract
\(
	(\mathsf B_\ell, \mathsf B_\ell)
\)-rings.

\begin{lemma} \label{b-diag}
	Let
	\(
		(R, \Delta)
	\) be an
	\(
		(\mathsf B_\ell, \mathsf B_\ell)
	\)-ring for
	\(
		\ell \geq 3
	\),
	\(
		(R_{pq}, \Delta^0_p)_{pq}
	\) be the corresponding
	\(
		(\mathsf B_\ell, \mathsf A_{\ell - 1})
	\)-ring, and
	\(
		\alpha \in \mathsf B_\ell
	\) be a root. Suppose that
	\(
		\zeta \in P_\alpha
	\) and
	\(
		\eta \in P_{-\alpha}
	\) are such that the criterion from lemma \ref{a-swap-long} or \ref{a-swap-short} holds, i.e.
	\(
		1_i + \zeta \eta \in R_{ii}^*
	\) for long
	\(
		\alpha = \mathrm e_j - \mathrm e_i
	\) (with respect to the parametrizations
	\(
		t_{ij}
	\) and
	\(
		t_{ji}
	\)) and
	\(
		\widehat \rho(\zeta)\, \widehat \rho(\eta)
		- u \circ v
		+ 1_{-i}
		\in R_{-i}^*
	\) for short
	\(
		\alpha = \mathrm e_i
	\). Then there exists an automorphism
	\(
		h = (F_{pq}, H_p)_{pq}
	\) of the corresponding partial graded odd form ring given by the explicit formulae from these lemmas such that the products of conjugations by
	\(
		t_\alpha(\zeta)
	\),
	\(
		t_{-\alpha}(\eta)
	\) and
	\(
		t_{-\alpha}(\eta')
	\),
	\(
		t_\alpha(\zeta')
	\), \(h\) on
	\(
		G_{
			(\mathbb R \alpha + \mathbb R_{> 0} \beta)
			\cap \Phi
		}
	\) coincide for every
	\(
		\beta \nparallel \alpha
	\).
\end{lemma}
\begin{proof}
	Note that
	\(
		F_{pq}
	\) and
	\(
		H_p
	\) are homomorphisms. For
	\(
		H_p
	\) in the case of short \(\alpha\) this follows from an easy calculation using axioms of
	\(
		(\mathsf B_\ell, \mathsf A_{\ell - 1})
	\)-rings. The same calculations as in the proofs of lemmas \ref{a-swap-long} and \ref{a-swap-short} show the last claim for all cases except
	\(
		F_{\pm i, \mp j}
	\),
	\(
		H_{\pm i}
	\), and
	\(
		H_{\pm j}
	\) for long
	\(
		\alpha = \mathrm e_j - \mathrm e_i
	\). In the remaining cases we just apply the conjugations to
	\begin{align*}
		t_{\pm i}(u)
		&= [t_k(\dotminus u)\, t_{k, \pm i}(-1_k)]\,
		t_{-k, \pm i}(-\inv{\widehat \rho(u)}),
	\\
		t_{\pm j}(u)
		&= [t_k(\dotminus u)\, t_{k, \pm j}(-1_k)]\,
		t_{-k, \pm j}(-\inv{\widehat \rho(u)}),
	\\
		t_{\pm i, \mp j}(z)
		&= [t_{\pm i, k}(z), t_{k, \mp j}(1_k)]
	\end{align*}
	for
	\(
		k \notin \{-i, i, -j, j\}
	\) with corresponding sign. It also follows that all
	\(
		F_{pq}
	\) and
	\(
		H_p
	\) are bijective.

	It remains to check that \(h\) commutes with all operations
	\(
		C_{\beta \gamma}^\delta
	\). This follows from the last claim if
	\(
		\pm \alpha
	\) do not lie in
	\(
		\{\beta, \gamma\}
		\cup \interval \beta \gamma
	\). In the case of short \(\alpha\) the necessarily identities follow from the known ones by inserting the factors
	\(
		z^{\pm 1}
	\) and
	\(
		\inv z^{\pm 1}
	\). Finally, suppose that
	\(
		\alpha = \mathrm e_j - \mathrm e_i
	\) is long. We do calculations in the language of
	\(
		(\mathsf B_\ell, \mathsf A_{\ell - 1})
	\)-rings for simplicity. By definition,
	\(
		F_{ik}(1_{ik})\, F_{kj}(z)
		= F_{ij}(1_{ik} z)
	\) and
	\(
		F_{jk}(1_{jk})\, F_{ki}(z)
		= F_{ji}(1_{jk} z)
	\) for
	\(
		k \notin \{-i, i, -j, -j\}
	\). It follows that
	\begin{align*}
			H_{-i}(u) \circ H_j(v)
		&=
			F_{ik}(1_{ik})
			\bigl(
				H_{-k}(u\, \inv{1_{ki}})
				\circ H_j(v)
			\bigr)
		=
			F_{ij}(u \circ v);
	\\
			\widehat \rho\bigl(H_{-i}(u)\bigr)\,
			F_{-i, j}(z)
		&=
			F_{i, -k}(1_{i, -k})\,
			\widehat \rho\bigl(
				H_k(u \cdot \inv{1_{-k, i}})
			\bigr)\,
			F_{k, -i}(\inv{1_{i, -k}})\,
			F_{-i, j}(z)
		=
			F_{ij}(\widehat \rho(u)\, z);
	\\
			F_{i, -j}(z)\,
			\widehat \rho\bigl(H_j(u)\bigr)
		&=
			F_{ik}(1_{ik})\,
			F_{k, -j}(1_{ki} z)\,
			\widehat \rho\bigl(
				H_j(u)
			\bigr)
		=
			F_{ij}(z\, \widehat \rho(u));
	\\
			H_{-i}(u) \cdot F_{-i, j}(z)
		&=
			H_{-i}(u)
			\cdot F_{-i, k}(1_{-i, k})\,
			F_{kj}(1_{k, -i} z)
		=
			H_j(u \cdot z);
	\\
			H_j(u) \cdot F_{j, -i}(z)
		&=
			H_j(u)
			\cdot F_{jk}(1_{jk})\,
			F_{k, -i}(1_{kj} z)
		=
			H_{-i}(u \cdot z);
	\\
			F_{ik}(z)\, F_{kj}(z')
		&=
			\phi\bigl(
				\widehat \rho\bigr(
					F_{ik}(z)\,
					F_{k, -i}(\inv{1_{i, -k}})
				\bigr)
			\bigr)\,
			F_{-i, j}(\inv{1_{-k, i}} z') \\
			&+ F_{i, -k}(1_{i, -k})\,
			\bigl(
				F_{-k, -i}(\inv z)\,
				F_{-i, j}(\inv{1_{-k, i}} z')
			\bigr)
		=
			F_{ij}(z z');
	\\
			\widehat \rho\bigl(H_i(u)\bigr)\,
			F_{ij}(z)
		&=
			\widehat \rho\bigl(H_i(u)\bigr)\,
			F_{ik}(1_{ik})\,
			F_{kj}(1_{ki} z)
		=
			F_{-i, j}(\widehat \rho(u)\, z);
	\\
			F_{ij}(z)\,
			\widehat \rho\bigl(H_{-j}(u)\bigr)
		&=
			F_{ik}(1_{ik})\,
			F_{kj}(1_{ki} z)\,
			\widehat \rho\bigl(H_{-j}(u)\bigr)
		=
			F_{i, -j}(z\, \widehat \rho(u));
	\\
			H_i(u) \cdot F_{ij}(z)
		&=
			H_i(u)
			\cdot F_{ik}(1_{ik})\,
			F_{kj}(1_{ki} z)
		=
			H_j(u \cdot z);
	\\
			H_{-j}(u) \cdot F_{-j, -i}(z)
		&=
			H_{-j}(u)
			\cdot F_{-j, -k}(z \inv{1_{ki}})\,
			F_{-k, -i}(\inv{1_{ik}})
		=
			H_{-i}(u \cdot z);
	\\
			\phi\bigl(
				F_{ij}(z)\,
				F_{j, -i}(z')
			\bigr)
		&=
			\phi\bigl(
				F_{ik}(1_{ik})\,
				F_{kj}(1_{ki} z)\,
				F_{j, -i}(z')
			\bigr)
		=
			H_{-i}(\phi(z z'));
	\\
			\phi\bigl(
				F_{-j, i}(z)\,
				F_{ij}(z')
			\bigr)
		&=
			\phi\bigl(
				F_{-j, i}(z)\,
				F_{ik}(1_{ik})\,
				F_{kj}(1_{ki} z')
			\bigr)
		=
			H_j(\phi(z z'));
	\\
			F_{ij}(z)\, F_{jk}(z')
		&=
			\inv{
				\phi\bigl(
					\widehat \rho\bigl(
						F_{-k, -j}(\inv{z'})\,
						F_{-j, k}(\inv z \inv{1_{-k, i}})
					\bigr)
				\bigr)\,
				F_{k, -i}(\inv{1_{i, -k}})
			} \\
			&+ \inv{
				F_{-k, j}(1_{-k, i} z)\,
				\bigl(\,
					F_{jk}(z')
					F_{k, -i}(\inv{1_{i, -k}})
				\bigr)
			}
		=
			F_{ik}(z z')
	\\
			F_{ki}(z)\, F_{ij}(z')
		&=
			\phi\bigr(
				\widehat \rho\bigl(
					F_{ki}(z)\,
					F_{i, -k}(1_{i, -k})
				\bigr)
			\bigr)\,
			F_{-k, j}(1_{-k, i} z) \\
			&+ F_{k, -i}(\inv{1_{i, -k}})\,
			\bigl(
				F_{-i, k}(\inv z)\,
				F_{-k, j}(1_{-k, i} z)
			\bigr)
		=
			F_{kj}(z z').
	\end{align*}
	The same works for the identities involving
	\(
		F_{ji}
	\).
\end{proof}

Now let \(\Phi\) be a root system of type
\(
	\mathsf B_\ell
\) for
\(
	\ell \geq 3
\) or
\(
	\mathsf F_4
\) and \((R, \Delta)\) be a
\(
	(\Phi, \Phi)
\)-ring. We define formal power series over
\(
	(R, \Delta)
\) as the pair
\(
	(R[[\eps]], \Delta[[\eps]])
\), where
\[
	R[[\eps]] = \bigoplus_{m \in \mathbb N_0} R \eps^m,
\quad
	\Delta[[\eps]]
	= \bigoplus_{m \in \mathbb N_0}^\cdot
		(\Delta \cdot \eps^m)
	\dotoplus \bigoplus_{m \in \mathbb N_0}^\cdot
		\phi(R \eps^{2 m + 1}).
\]
In other words, \(\Delta[[\eps]]\) is the set of formal sums
\(
	\sum_m^\cdot u_m \cdot \eps^m
	\dotplus \sum_m^\cdot \phi(x_m \eps^{2m + 1})
\), where almost all values
\(
	u_m
\) and
\(
	x_m
\) are trivial. The operations are given by
\begin{align*}
	(x \eps^m)\, (y \eps^{m'}) &= x y \eps^{m + m'},
&
	u \cdot \eps^m \dotplus v \cdot \eps^m
	&= (u \dotplus v) \cdot \eps^m,
\\
	(x \eps^m)^* &= x^* \eps^m,
&
	\bigl\langle
		u \cdot \eps^m,
		v \cdot \eps^{m'}
	\bigr\rangle
	&= \langle u, v \rangle\, \eps^{m + m'},
\\
	\phi(x \eps^{2m}) &= \phi(x)\, \eps^m,
&
	\rho(u \cdot \eps^m) &= \rho(u)\, \eps^{2m},
\\&&
	(u \cdot \eps^m) \cdot (x \cdot \eps^{m'})
	&= (u \cdot x) \cdot \eps^{m + m'}.
\end{align*}
For
\(
	\Phi = \mathsf F_4
\) the additional operations are
\[
	(u \cdot \eps^m)\, (v \cdot \eps^{m'})
	= u v \cdot \eps^{m + m'},
\quad
	(u \cdot \eps^m)^* = u^* \cdot \eps^m,
\quad
	\phi(u \cdot \eps^m) = \phi(u)\, \eps^m,
\quad
	\rho(x \eps^m) = \rho(x) \cdot \eps^m.
\]
Note that this definition is not symmetric under the outer automorphism of
\(
	\Phi = \mathsf F_4
\). It is easy to check that
\(
	(R[[\eps]], \Delta[[\eps]])
\) indeed satisfies all axioms. Actually, we need its truncation
\(
	(\widehat R, \widehat \Delta)
\), where
\[
		\widehat R
	=
		R[[\eps]] / (\eps^3)
	=
		R
		\oplus R \eps
		\oplus R \eps^2,
\quad
		\widehat \Delta
	=
		\Delta[[\eps]] / (\eps^3)
	=
		\Delta
		\dotoplus \phi(R \eps)
		\dotoplus (\Delta \cdot \eps)
		\dotoplus (\Delta / \phi(R) \cdot \eps^2).
\]
Let
\(
	(P_\alpha)_{\alpha \in \Phi}
\) and
\(
	(\widehat P_\alpha)_{\alpha \in \Phi}
\) be the
\(
	(\Phi, \varnothing)
\)-rings associated with
\(
	(R, \Delta)
\) and
\(
	(\widehat R, \widehat \Delta)
\), see lemma \ref{f4-empty} for the case
\(
	\Phi = \mathsf F_4
\). Let also
\[
	\widehat T
	= \bigl\{
		(
			F_\alpha
			\in \Aut(\widehat P_\alpha)
		)_{\alpha \in \Phi}
		\mid
		C_{\alpha \beta}^\gamma\bigl(
			F_\alpha(\zeta),
			F_\beta(\eta)
		\bigr)
		= F_\gamma\bigl(
			C_{\alpha \beta}^\gamma(\zeta, \eta)
		\bigr)
	\bigr\}
\]
be the group of formal diagonal elements (the ``adjoint torus''). Using the commutator formulae, for any special closed
\(
	\Sigma \subseteq \Phi
\) there is a group
\(
	\widehat G_\Sigma
	= \prod_{\alpha \in \Sigma}
		\widehat P_\alpha
\) (with respect to a right extreme order on \(\Sigma\)) and
\(
	\widehat T
\) acts on
\(
	\widehat G_\Sigma
\) by automorphisms. The following lemma actually holds for any
\(
	(\Phi, \Phi)
\)-ring instead of
\(
	(\widehat R, \widehat \Delta)
\).

\begin{lemma} \label{eq-check}
	Let
	\(
		\alpha, \beta \in \Phi
	\) be non-orthogonal linearly independent roots and
	\(
		\Sigma
		= (
			\mathbb R \alpha + \mathbb R_{> 0} \beta
		)
		\cap \Phi
	\). If two triples
	\[
		(\zeta_1, \eta_1, h_1),
		(\zeta_2, \eta_2, h_2)
		\in \widehat P_\alpha
		\times \widehat P_{-\alpha}
		\times \widehat T
	\]
	act on
	\(
		\widehat G_\Sigma
	\) in the same way (as the compositions of conjugations by
	\(
		t_\alpha(\zeta)
	\),
	\(
		t_{-\alpha}(\eta)
	\), \(h\) and by
	\(
		t_\alpha(\zeta')
	\),
	\(
		t_{-\alpha}(\eta')
	\), \(h\) respectively), then
	\(
		\zeta = \zeta'
	\),
	\(
		\eta = \eta'
	\), and the \(\Sigma\)-components of \(h\) and \(h'\) coincide.

	Moreover, let
	\(
		\Psi \subseteq \Phi
	\) be a root subsystem of type
	\(
		\mathsf A_2
	\) or
	\(
		\mathsf B_2
	\),
	\(
		\beta \in \Phi \setminus (\Psi \cup \Psi^\perp)
	\), and
	\(
		\Sigma
		= (
			\mathbb R \Psi + \mathbb R_{> 0} \beta
		)
		\cap \Phi
	\). We order \(\Psi\) by polar angle with respect to some orientation and starting angle in the plane
	\(
		\mathbb R \Psi
	\). If two tuples
	\[
		(\zeta_\alpha, h)_\alpha,
		(\zeta'_\alpha, h')_\alpha
		\in \bigl(
			\prod_{\alpha \in \Psi} \widehat P_\alpha
		\bigr)
		\times \widehat T
	\]
	act in the same way on
	\(
		\widehat G_\Sigma
	\) and
	\(
		\widehat G_{-\Sigma}
	\), then
	\(
		\zeta_\alpha = \zeta'_\alpha
	\) and the \(\Sigma\)-components of \(h\) and \(h'\) coincide.
\end{lemma}
\begin{proof}
	The first claim easily follows from the commutator formulae by considering the three possible cases. Namely, either \(\alpha\) and \(\beta\) span
	\(
		\mathsf A_2
	\), or they span
	\(
		\mathsf B_2
	\) and \(\alpha\) is long, or they span
	\(
		\mathsf B_2
	\) and \(\alpha\) is short.

	To prove the second claim we are going to check by induction on \(m\) that for every sorted sequence
	\(
		\alpha_1, \ldots, \alpha_m \in \Psi
	\) (with respect to the chosen polar angle) the composition of conjugations by
	\(
		t_{\alpha_1}(\zeta_1)
	\), \ldots,
	\(
		t_{\alpha_m}(\zeta_m)
	\), and \(h\) on
	\(
		\widehat G_\Sigma
	\) uniquely determines
	\(
		\zeta_i
	\). The case
	\(
		m \leq 1
	\) is clear.
	\begin{itemize}

		\item
		Suppose that \(\Psi\) is of type
		\(
			\mathsf A_2
		\) and
		\(
			\Psi \cup \Sigma \cup (-\Sigma)
		\) is of type
		\(
			\mathsf A_3
		\). Without loss of generality,
		\(
			\alpha_1 = \mathrm e_2 - \mathrm e_1
		\),
		\(
			\alpha_2 = \mathrm e_3 - \mathrm e_1
		\), and
		\(
			\beta = \mathrm e_4 - \mathrm e_2
		\) (by choosing
		\(
			\beta \in \Sigma
		\) or
		\(
			\beta \in -\Sigma
		\)). Then
		\(
			\zeta_1
		\) may be determined by the
		\(
			(\alpha_1 + \beta)
		\)-component of the conjugation of
		\(
			t_{24}(1)
		\).

		\item
		Now suppose that \(\Psi\) is of type
		\(
			\mathsf A_2
		\) and
		\(
			\Psi \cup \Sigma \cup (-\Sigma)
		\) is of type
		\(
			\mathsf B_3
		\). Without loss of generality,
		\(
			\alpha_1 = \mathrm e_2 - \mathrm e_1
		\),
		\(
			\alpha_2 = \mathrm e_2 - \mathrm e_3
		\), and
		\(
			\beta = \mathrm e_1
		\) (possibly,
		\(
			\beta \in -\Sigma
		\)). Then
		\(
			\zeta_1
		\) may be determined by the
		\(
			(\alpha_1 + \beta)
		\)-component of the conjugation of
		\(
			t_1(\iota)
		\). If \(m \geq 4\), then
		\(
			\alpha_4 = \mathrm e_1 - \mathrm e_2
		\) and we also have to use that
		\(
			\zeta_4
		\) is determined by the \(\beta\)-component of the conjugation of
		\(
			t_2(\iota)
		\).

		\item
		If \(\Psi\) is of type
		\(
			\mathsf B_2
		\) and
		\(
			\alpha_1
		\) is long, we may assume that
		\(
			\alpha_1 = \mathrm e_2 - \mathrm e_1
		\),
		\(
			\alpha_2 = \mathrm e_2
		\), and
		\(
			\beta = \mathrm e_1 + \mathrm e_3
		\). Then
		\(
			\zeta_1
		\) may be determined by the
		\(
			(\alpha_1 + \beta)
		\)-component of the conjugation of
		\(
			t_{-1, 3}(1)
		\).

		\item
		Finally, if \(\Psi\) is of type
		\(
			\mathsf B_2
		\) and
		\(
			\alpha_1
		\) is short, we may assume that
		\(
			\alpha_1 = \mathrm e_1
		\),
		\(
			\alpha_2 = \mathrm e_1 - \mathrm e_2
		\), and
		\(
			\beta = \mathrm e_3 - \mathrm e_1
		\). Moreover, since we may invert the order on \(\Psi\) and all
		\(
			\zeta_i
		\), by the previous case we may also assume that
		\(
			m \leq 7
		\). Then
		\(
			\zeta_1
		\),
		 may be determined by the
		\(
			(2 \alpha_1 + \beta)
		\)-component of the conjugation of
		\(
			t_{13}(1)
		\). \qedhere

	\end{itemize}
\end{proof}

For any special closed subset
\(
	\Sigma \subseteq \Phi
\) consider the group
\(
	\widehat G(\Phi, \Sigma)
\) generated by
\(
	\widehat G_\Sigma
\),
\(
	\widehat T
\), and elements
\(
	t_\alpha(\zeta)
\) for
\(
	\zeta \in \Ker(\widehat P_\alpha \to P_\alpha)
\),
\(
	\alpha \in \Psi \setminus \Sigma
\). The relations are the following.
\begin{itemize}

	\item
	The maps from
	\(
		\widehat G_\Sigma
	\),
	\(
		\widehat T
	\), and
	\(
		\Ker(\widehat P_\alpha \to P_\alpha)
	\) to
	\(
		\widehat G(\Phi, \Sigma)
	\) are homomorphisms.

	\item
	The commutator formula holds for all pairs of non-parallel roots.

	\item
	The conjugacy identities
	\(
		\up h{t_\alpha(\zeta)}
		= t_\alpha(H_\alpha(\zeta))
	\) for
	\(
		h = (H_\alpha)_\alpha \in \widehat T
	\).

	\item
	Finally,
	\(
		t_\alpha(\zeta)\, t_{-\alpha}(\eta)
		= t_{-\alpha}(\eta')\, t_\alpha(\zeta')\, h
	\), where
	\(
		\zeta'
	\),
	\(
		\eta'
	\), and \(h\) are given by the formulae from lemmas \ref{a-swap-long} and \ref{a-swap-short} since all elements from
	\(
		1 + \Ker(\widehat R \to R)
	\) are invertible. In the case
	\(
		\Phi = \mathsf F_4
	\) they may be computed using various root subsystems of type
	\(
		\mathsf B_3
	\), see the proof of lemma \ref{lie-alg} below for the correctness. The element \(h\) indeed lies in
	\(
		\widehat T
	\) by lemma \ref{b-diag}.

\end{itemize}

\begin{lemma} \label{lie-alg}
	Let
	\(
		\Sigma \subseteq \Phi
	\) be a special closed subset and \(\leq\) be a linear order on \(\Phi\) such that its restriction to \(\Sigma\) is right extreme. Then the product map
	\(
		\bigl(
			\prod_{\alpha \in \Phi} P_{\Sigma, \alpha}
		\bigr)
		\times \widehat T
		\to \widehat G(\Phi, \Sigma)
	\) is one-to-one, where
	\(
		P_{\Sigma, \alpha} = \widehat P_\alpha
	\) for
	\(
		\alpha \in \Sigma
	\) and
	\(
		P_{\Sigma, \alpha}
		= \Ker(\widehat P_\alpha \to P_\alpha)
	\) otherwise.
\end{lemma}
\begin{proof}
	Consider the following string rewrite system. Its objects are formal strings of root generators and elements of
	\(
		\widehat T
	\) enclosed in square brackets. The rules are the following.
	\begin{align*}
		t_\alpha(\dot 0) &\to \eps;
	\quad
		[1] \to 1;
	\quad
		t_\alpha(\zeta)\, t_\alpha(\eta)
		\to t_\alpha(\zeta \dotplus \eta);
	\quad
		[h]\, [h'] \to [h h'];
	\\
		[h]\, t_\alpha(\zeta)
		&\to t_\alpha(H_\alpha(\zeta))\, [h]
		\text{ for } h = (H_\alpha)_\alpha;
	\\
		t_\alpha(\zeta)\, t_\beta(\eta)
		&\to \bigl(
			\prod_{\gamma \in \interval \alpha \beta}
				t_\gamma\bigl(
					C_{\alpha \beta}
						^\gamma(\zeta, \eta)
				\bigr)
		\bigr)\,
		t_\beta(\eta)\, t_\alpha(\zeta)
		\text{ for } \alpha \nparallel \beta
		\text{ and } \alpha > \beta;
	\\
		t_\alpha(\zeta)\, t_{-\alpha}(\eta)
		&\to t_{-\alpha}(\eta')\, t_\alpha(\zeta')\, h
		\text{ for } \alpha > -\alpha,
		\text{ where \(\eta'\), \(\zeta'\), \(h\) are taken from lemmas \ref{a-swap-long} and \ref{a-swap-short}}.
	\end{align*}
	This rewrite system is terminating. We have to check that it is confluent and the last rule also holds as an identity for
	\(
		\alpha < -\alpha
	\) (also, that the rule is well-defined for
	\(
		\Phi = \mathsf F_4
	\)). Consider the case
	\(
		\Phi = \mathsf B_\ell
	\). Using theorem \ref{phi-0-ring} and lemma \ref{b-diag}, it suffices to check the confluence for the following start strings.
	\begin{align*}
		&t_\alpha(\zeta)\,
		t_\alpha(\zeta')\,
		t_{-\alpha}(\eta);
	\quad
		t_\alpha(\zeta)\,
		t_{-\alpha}(\eta)\,
		t_{-\alpha}(\eta);
	\quad
		t_\alpha(\dot 0)\, t_{-\alpha}(\eta);
	\quad
		t_\alpha(\zeta)\, t_{-\alpha}(\dot 0);
	\\
		&t_\alpha(\zeta)\,
		t_\beta(\eta)\,
		t_\gamma(\theta)
		\text{ for coplanar } \alpha > \beta > \gamma.
	\end{align*}
	For the first four strings we may apply the first claim of lemma \ref{eq-check}. For the last one we have to check the equalities
	\(
		\psi_\delta = \psi'_\delta
	\) and
	\(
		h = h'
	\) for some
	\(
		\psi_\delta, \psi'_\delta \in \widehat P_\delta
	\) and
	\(
		h, h' \in \widehat T
	\) depending on \(\zeta\), \(\eta\), \(\theta\), where \(\delta\) runs over some root subsystem
	\(
		\Psi \subseteq \Phi
	\) of type
	\(
		\mathsf A_2
	\) or
	\(
		\mathsf B_2
	\) (the case
	\(
		\mathsf A_1 \times \mathsf A_1
	\) is trivial). We may further apply rewrite rules backwards (i.e. for the ``wrong'' orderings of the involved roots) to change the order in
	\(
		\bigl(
			\prod_\delta t_\delta(\psi_\delta)
		\bigr)\, [h]
	\) and
	\(
		\bigl(
			\prod_\delta t_\delta(\psi'_\delta)
		\bigr)\, [h']
	\) from \(\leq\) to some order by a polar angle and then apply the second claim of lemma \ref{eq-check}. Since the applied ``wrong'' rules are the same for both strings, the equality of the resulting strings implies the claim. Also, by lemmas \ref{b-diag} and \ref{eq-check} the formal application of the last rewrite rule to its right hand side
	\(
	\) (for the ``wrong'' root
	\(
		-\alpha < -\alpha
	\)) gives the left hand side after cancellation of the two factors of
	\(
		\widehat T
	\).

	Finally, let
	\(
		\Phi = \mathsf F_4
	\). The elements \(\zeta'\) and \(\eta'\) from the last rule are independent on the choice of a root subsystem
	\(
		\alpha \in \Psi \subseteq \Phi
	\) of type
	\(
		\mathsf B_3
	\) by lemma \ref{eq-check}. By the same argument the components of \(h\) are independent on such a choice and
	\(
		h \in \widehat T
	\). Finally, the rewrite system is confluent and the last rule holds for
	\(
		\alpha < -\alpha
	\) as an identity since all necessary relations may be checked inside various root subsystems of type
	\(
		\mathsf B_3
	\) using the already known case of the lemma.
\end{proof}

\begin{theorem} \label{exist}
	Let \(\Phi\) be a root system of type
	\(
		\mathsf B_\ell
	\) for
	\(
		\ell \geq 3
	\) or
	\(
		\mathsf F_4
	\) and
	\(
		(R, \Delta)
	\) be an
	\(
		(\Phi, \Phi)
	\)-ring. Then the Steinberg group
	\(
		\stlin(\Phi, R, \Delta)
	\) is \(\Phi\)-graded. Moreover, for any system of positive roots
	\(
		\Pi \subseteq \Sigma
	\) the subgroups
	\(
		\stlin(\Phi, R, \Delta)_{-\Pi}
	\) and
	\(
		\stlin(\Phi, R, \Delta)_\Pi
	\) have trivial intersection.
\end{theorem}
\begin{proof}
	The cases
	\(
		\Phi
		\in \{
			\mathsf A_\ell,
			\mathsf D_\ell,
			\mathsf E_\ell
		\}
	\) are well-known. It suffices to construct some \(\Phi\)-graded group with the last property and the
	\(
		(\Phi, \Phi)
	\)-ring isomorphic to
	\(
		(R, \Delta)
	\). Let
	\[
		G_0 = \widehat G(\Phi, \varnothing)
		= \bigl(
			\prod_{\alpha \in \Phi}
				t_\alpha(
					\Ker(
						\widehat P_\alpha \to P_\alpha
					)
				)
		\bigr)
		\times \widehat T.
	\]
	For any
	\(
		\alpha \in \Phi
	\) and
	\(
		\zeta \in P_\alpha
	\) let
	\(
		t_\alpha(\zeta) \in \Aut(G_0)
	\) be the conjugation inside
	\(
		\widehat G(\Phi, \{\alpha\})
	\). Such automorphisms satisfy all relations of Steinberg groups by lemma \ref{lie-alg}. It remains to check that
	\(
		U_\Sigma = \bigl\langle
			t_\alpha(P_\beta)
			\mid
			\alpha \in \Sigma
		\bigr\rangle
	\) trivially intersects
	\(
		U_{-\Sigma}
	\) for any special closed subset
	\(
		\Sigma \subseteq \Phi
	\).

	Let us proceed by induction on \(|\Sigma|\). Suppose that \(\Sigma\) contains a long extreme root \(\alpha\) (or just an extreme root \(\alpha\) if
	\(
		\Phi = \mathsf F_4
	\)). Choose
	\(
		\beta \in \Phi
	\) such that
	\(
		\widehat{\alpha \beta} = 2 \pi / 3
	\) and \(\beta\) together with \(\Sigma\) are contained in a common special closed subset \(\Sigma'\) such that both \(\alpha\) and \(\beta\) are extreme roots of \(\Sigma'\). Namely, if \(\Sigma\) consists of positive roots with respect to a base and \(\alpha\) is basic, then we may take \(\Sigma'\) to be the set of all positive roots and \(\beta\) to be a basic root near \(\alpha\) in the Dynkin diagram. Then for any
	\(
		g \in U_{\Sigma \setminus \{\alpha\}}
	\) we have
	\[
		\up{g\, t_\alpha(x)}{t_\beta(\eps)}
		= t_{\alpha + \beta}(\upsilon(x)\, \eps)\,
		t_\beta(\eps)\, h
	\]
	for some
	\(
		h
		\in U_{
			\Sigma'
			\setminus \{\alpha, \beta, \alpha + \beta\}
		}
	\) and group automorphism \(\upsilon\). On the other hand
	\(
		\up{U_{-\Sigma}}{t_\beta(\eps)}
	\) does not contain factors with the root
	\(
		\alpha + \beta
	\). From this and the dual argument it follows that
	\(
		U_\Sigma \cap U_{-\Sigma}
		= U_{\Sigma \setminus \{\alpha\}}
		\cap U_{(-\Sigma) \setminus \{-\alpha\}}
	\) and we may apply the induction hypothesis.

	Now we may assume that
	\(
		\Phi = \mathsf B_\ell
	\) and all extreme roots of \(\Sigma\) are short, i.e.
	\[
		\Sigma
		= \{
			\mathrm e_i
			\mid
			1 \leq i \leq k
		\}
		\cup \{
			\mathrm e_i + \mathrm e_j
			\mid
			1 \leq i < j \leq k
		\}
	\]
	for some \(k\) up to the action of the Weyl group. For any
	\(
		u \in \Delta
	\) and
	\(
		g \in U_{\Sigma \setminus \{\mathrm e_1\}}
	\) the component of
	\(
		\up{g\, t_1(u)}{t_{1, -2}(\eps)}
	\) with the root
	\(
		\mathrm e_{-2}
	\) is
	\(
		t_{-2}(
			v \cdot \eps^2
			\dotminus u \cdot (-\eps)
		)
	\) for some
	\(
		v \in \Delta / \phi(R)
	\) by the formulae from lemma \ref{a-swap-short}. Here we use that
	\(
		\Delta \cdot \eps \subseteq \widetilde \Delta
	\), for another (larger) choice of
	\(
		(\widetilde R, \widetilde \Delta)
	\) as a factor-\((\Phi, \Phi)\)-ring of
	\(
		(R[[\eps]], \Delta[[\eps]])
	\) additional high-order summands may appear. On the other hand,
	\(
		\up{U_{-\Sigma}}{t_{1, -2}(\eps)}
	\) does not contain components with the root
	\(
		\mathrm e_1 - \mathrm e_2
	\). From this and the dual argument we have
	\(
		U_\Sigma \cap U_{-\Sigma}
		= U_{\Sigma \setminus \{\mathrm e_1\}}
		\cap U_{(-\Sigma) \setminus \{-\mathrm e_1\}}
	\).
\end{proof}

\section*{Declarations}

Research is supported by ``Native towns'', a social investment program of PJSC ``Gazprom Neft''.

\bibliographystyle{plain}
\bibliography{references}

\end{document}